\documentclass[preprint,11pt]{elsarticle}

\usepackage{fullpage,amsmath,amsthm,amsfonts,url,amssymb,stmaryrd}
\usepackage{mathrsfs}
\usepackage{mathtools}
\usepackage{amssymb,amscd}
\usepackage{color}
\usepackage{chemarrow}
\usepackage{pictexwd,dcpic}
\usepackage{extarrows}
\usepackage{enumitem}
\usepackage{hyperref}
\makeatletter
\def\ps@pprintTitle{
 \let\@oddhead\@empty
 \let\@evenhead\@empty
 \def\@oddfoot{\centerline{\thepage}}
 \let\@evenfoot\@oddfoot}
\makeatother

\newtheorem{theorem}{Theorem}[section]
\newtheorem{corollary}[theorem]{Corollary}
\newtheorem{lemma}[theorem]{Lemma}
\newtheorem{proposition}[theorem]{Proposition}
\newtheorem{definition}[theorem]{Definition}
\newtheorem{hypothesis}[theorem]{Hypothesis}
\newtheorem{remark}[theorem]{Remark}

\newtheorem{conjecture}[theorem]{Conjecture}

\newcommand{\hooklongrightarrow}{\lhook\joinrel\longrightarrow}
\newcommand{\twoheadlongrightarrow}{\relbar\joinrel\twoheadrightarrow}

\newcommand{\ra}{\rightarrow}
\newcommand{\lra}{\longrightarrow}

\newcommand{\ul}{\underline}

\newcommand{\hV}{\textbf V}
\newcommand{\hR}{\textbf R}

\newcommand{\cY}{\mathcal Y}
\newcommand{\bA}{\mathbb A}
\newcommand{\bC}{\mathbb C}
\newcommand{\F}{\mathbb F}

\newcommand{\Q}{\mathbb Q}
\newcommand{\bR}{\mathbb R}

\newcommand{\bT}{\mathbb T}
\newcommand{\Z}{\mathbb Z}

\newcommand{\co}{\mathcal O}
\newcommand{\cA}{\mathcal A}

\newcommand{\cC}{\mathcal C}
\newcommand{\cS}{\mathcal S}

\newcommand{\cI}{\mathcal I}

\newcommand{\cF}{\mathcal F}

\newcommand{\cE}{\mathcal E}
\newcommand{\cG}{\mathcal G}
\newcommand{\cJ}{\mathcal J}

\newcommand{\fn}{\mathfrak n}

\newcommand{\fm}{\mathfrak{m}}
\newcommand{\ub}{\mathfrak b}

\newcommand{\fp}{\mathfrak p}

\newcommand{\fF}{\mathfrak F}

\newcommand{\fx}{\mathfrak x}

\newcommand{\fC}{\mathfrak C}

\newcommand{\fa}{\mathfrak a}

\newcommand{\fc}{\mathfrak c}

\newcommand{\tm}{\text{\texttt{m}}}
\newcommand{\fB}{\mathfrak B}

\DeclareMathAlphabet{\mathpzc}{OT1}{pzc}{m}{it}
\DeclareMathOperator{\sm}{\mathrm sm}
\DeclareMathOperator{\la}{\mathrm la}

\DeclareMathOperator{\gl}{\mathfrak gl}

\DeclareMathOperator{\tr}{\mathrm tr}
\DeclareMathOperator{\GL}{\mathrm GL}
\DeclareMathOperator{\gr}{\mathrm gr}

\DeclareMathOperator{\Fil}{\mathrm Fil}
\DeclareMathOperator{\Res}{\mathrm Res}

\DeclareMathOperator{\Gal}{\mathrm Gal}
\DeclareMathOperator{\Hom}{\mathrm Hom}

\DeclareMathOperator{\End}{\mathrm End}
\DeclareMathOperator{\cris}{\mathrm cris}
\DeclareMathOperator{\rig}{\mathrm rig}
\DeclareMathOperator{\an}{\mathrm an}
\DeclareMathOperator{\Spec}{\mathrm Spec}

\DeclareMathOperator{\Frob}{\mathrm Frob}
\DeclareMathOperator{\Ind}{\mathrm Ind}
\DeclareMathOperator{\unr}{\mathrm unr}

\DeclareMathOperator{\Ker}{\mathrm Ker}
\DeclareMathOperator{\pr}{\mathrm pr}

\DeclareMathOperator{\ord}{\mathrm ord}

\DeclareMathOperator{\Ext}{\mathrm Ext}

\DeclareMathOperator{\Spf}{\mathrm Spf}
\DeclareMathOperator{\Ima}{\mathrm Im}

\DeclareMathOperator{\lalg}{\mathrm lalg}

\DeclareMathOperator{\id}{\mathrm id}
\DeclareMathOperator{\op}{\mathrm op}

\DeclareMathOperator{\dett}{\mathrm det}
\DeclareMathOperator{\alg}{\mathrm alg}
\DeclareMathOperator{\cyc}{\mathrm cyc}
\DeclareMathOperator{\soc}{\mathrm soc}

\DeclareMathOperator{\Supp}{\mathrm Supp}

\DeclareMathOperator{\Ad}{\mathrm Ad}
\DeclareMathOperator{\ad}{\mathrm ad}

\DeclareMathOperator{\sss}{\mathrm ss}

\DeclareMathOperator{\red}{\mathrm red}
\DeclareMathOperator{\nc}{\mathrm nc}

\DeclareMathOperator{\Art}{\mathrm Art}
\DeclareMathOperator{\ab}{\mathrm ab}
\DeclareMathOperator{\WD}{\mathrm WD}

\DeclareMathOperator{\rk}{\mathrm rk}

\DeclareMathOperator{\val}{\mathrm val}
\DeclareMathOperator{\univ}{\mathrm univ}

\DeclareMathOperator{\Def}{\mathrm Def}

\DeclareMathOperator{\cts}{\mathrm cts}

\DeclareMathOperator{\Ord}{\mathrm Ord}

\DeclareMathOperator{\lfin}{\mathrm lfin}
\DeclareMathOperator{\Mod}{\mathrm Mod}
\DeclareMathOperator{\aug}{\mathrm aug}
\DeclareMathOperator{\pro}{\mathrm pro}

\DeclareMathOperator{\Sp}{\mathrm Sp}

\DeclareMathOperator{\cosoc}{\mathrm cosoc}
\DeclareMathOperator{\pss}{\mathrm ps}
\DeclareMathOperator{\loc}{\mathrm loc}
\DeclareMathOperator{\tw}{\mathrm tw}
\DeclareMathOperator{\tf}{\mathrm tf}
\DeclareMathOperator{\Iw}{\mathrm Iw}
\DeclareMathOperator{\Nm}{\mathrm Nm}
\begin{document}
\title{$\GL_2(\Q_p)$-ordinary families and automorphy lifting}
\author[YD]{Yiwen Ding}\ead[YD]{yiwen.ding@bicmr.pku.edu.cn}

\address[YD]{B.I.C.M.R., Peking University, Beijing}

\begin{abstract}
We prove automorphy lifting results for certain essentially conjugate self-dual $p$-adic Galois representations $\rho$ over CM imaginary fields $F$, which satisfy in particular that $p$ splits in $F$, and that the restriction of $\rho$ on any decomposition group above $p$ is reducible with all the Jordan-H\"older factors of dimension at most $2$. We also show some results on Breuil's locally analytic socle conjecture in certain non-trianguline case. The main results are obtained by establishing an $R=\bT$-type result over the $\GL_2(\Q_p)$-ordinary families considered in \cite{BD1}.
\end{abstract}
\maketitle
\tableofcontents
\numberwithin{equation}{section}
\section{Introduction}
\noindent Let $p>2$ be a prime number, $F/F^+$ be a CM imaginary field such that $p$ splits in $F$ and that $F/F^+$ is unramified. For each place $v|p$ of $F^+$, we fix a place $\widetilde{v}$ of $F$ with $\widetilde{v}|v$. Let $E$ be a sufficiently large finite extension of $\Q_p$. In this note, we prove automorphy lifting results for certain essentially conjugate self-dual $p$-adic Galois representations $\rho$ of $\Gal_F$ over $E$.  For simplicity, we summarize our main result for the case where $\dim_E \rho=3$  in the the following theorem (see Theorem \ref{main} for the statement for general $n$) .
\begin{theorem}\label{mainIntro}
Let $\rho: \Gal_F \ra \GL_3(E)$ be a continuous representation satisfying the following conditions:
  \begin{enumerate}
    \item $\rho^c\cong \rho^{\vee} \varepsilon^{1-n}$, where $\rho^c(g)=\rho(cgc^{-1})$, $1\neq c\in \Gal(F/F^+)$, and $\varepsilon$ denotes the cyclotomic character.
    \item $\rho$ is unramified for all but finitely many primes, and unramified for all the places that do not split over $F^+$. We denote by $S$ the union of the complement set and the places dividing $p$.
    \item $\overline{\rho}$ is absolutely irreducible, $\overline{\rho}(\Gal_{F(\zeta_p)})\subseteq \GL_3(k_E)$ is adequate and $\overline{F}^{\Ker \ad \overline{\rho}}$ does not contain $F(\zeta_p)$.
    \item For all $v|p$, $\rho_{\widetilde{v}}$ is reducible, i.e. is of the form
     \begin{equation}\label{fil00}
       \rho_{\widetilde{v}} \cong \begin{pmatrix}
         \rho_{\widetilde{v},1} & * \\
         0 & \rho_{\widetilde{v},2}
       \end{pmatrix}.
     \end{equation}%\footnote{\label{fnt1}Thus there exists a parabolic subgroup $B\subsetneq P_{\widetilde{v}} \subsetneq \GL_3$ (where $B$ denotes the Borel subgroup of $\GL_3$ of upper triangular matrices), such that under a certain basis, $\rho_{\widetilde{v}}$ has image in $P_{\widetilde{v}}(E)$. We call such $\rho_{\widetilde{v}}$ \emph{$P_{\widetilde{v}}$-ordinary}. We denote by $\{\rho_{\widetilde{v},i}\}_{i=1,2}$ the corresponding constituents where $\rho_{\widetilde{v},1}$ is the sub, and $\rho_{\widetilde{v},2}$ is the quotient.}
    \item For all $v|p$, $\rho_{\widetilde{v}}$ is de Rham of distinct Hodge-Tate weights. Suppose moreover one of the following two conditions holds
     \begin{enumerate} \item for all $v|p$, and $i=1,2$, $\rho_{\widetilde{v},i}$ is absolutely irreducible and the Hodge-Tate weights of $\rho_{\widetilde{v},1}$ are strictly bigger than those of $\rho_{\widetilde{v},2}$\footnote{Where we use the convention that the Hodge-Tate weight of the cyclotomic character is $1$.};
     \item for all $v|p$, $\rho_{\widetilde{v}}$ is crystalline and generic in the sense of \cite{BHS3}\footnote{I.e. the eigenvalues $(\phi_1, \phi_2,\phi_3)$ of the crystalline Frobeinus satisfy $\phi_i\phi_j^{-1}\notin\{1,p\}$ for $i\neq j$.}.
     \end{enumerate}
    \item Let $\overline{\rho}_{\widetilde{v},i}$ be the mod $p$ reduction of $\rho_{\widetilde{v},i}$ (induced by $\overline{\rho}$), $\omega$ the modulo $p$ cyclotomic character. Suppose for all $v|p$\footnote{We need some more technical assumptions when $p=3$, that we ignore in the introduction.}
    \begin{enumerate} \item $\Hom_{\Gal_{\Q_p}}(\overline{\rho}_{\widetilde{v},i},\overline{\rho}_{\widetilde{v},j})=0$ for $i\neq j$;
    \item $\Hom_{\Gal_{\Q_p}}(\overline{\rho}_{\widetilde{v},1}, \overline{\rho}_{\widetilde{v},1}\otimes_{k_E} \omega)=0$;
    \item $\Hom_{\Gal_{\Q_p}}(\overline{\rho}_{\widetilde{v}}, \overline{\rho}_{\widetilde{v},2}\otimes_{k_E} \omega)=0$.
    \end{enumerate}
    \item There exist a definite unitary group $G/F^+$ attached to $F/F^+$ such that $G$ is quasi-split at all finite places of $F^+$, and an automorphic representation $\pi$ of $G$ with the associated Galois representation $\rho_{\pi}: \Gal_F \ra \GL_3(E)$ satisfying
    \begin{enumerate}
    \item $\overline{\rho}_{\pi}\cong \overline{\rho}$;
      \item $\pi_v$ is unramified for all $v\notin S$;
      \item $\pi$ is $\fB$-ordinary (cf. Definition \ref{fBord}, see also Lemma \ref{Bord}). \footnote{\label{fn2}Where $\fB$ is the block associated to $\{\overline{\rho}_{\widetilde{v},i}\}$ in the category of locally finite length smooth $L_P(\Q_p)$-representations (we refer to (\ref{block0}) and \S~\ref{sec: Pas} for details), and where $L_P$ is the Levi subgroup (containing the  subgroup of diagonal matrices) of $P:=\prod_{v|p} P_{\widetilde{v}}(\Q_p)$ with $P_{\widetilde{v}}\subseteq \GL_3$ the parabolic subgroup corresponding to the filtration (\ref{fil00}). In particular, $L_P$ is equal, up to the order of the factors, to $\prod_{v|p} (\GL_2(\Q_p)\times \Q_p^{\times} )$.}
          % \footnote{Which means that $\Ord_{P}((\otimes_{v\in S_p}\pi_{\widetilde{v}}) \otimes_E W_{p})_{\fB}\neq 0$ where $W_p$ is the algebraic representation of $G(F^+\otimes_{\Q} \Q_p)$ attached to $\pi$ (e.g. see  (\ref{equ: lalgaut})), $P:=\prod_{v|p} P_{\widetilde{v}}(\Q_p)$, $P_{\widetilde{v}}\subseteq \GL_3$ is the parabolic corresponding to the filtration (\ref{fil00}),   $\Ord_P(\cdot)$ denotes the $P$-ordinary part functor (e.g. see \cite[\S~4.3]{BD1}), and  $\fB$ is the block associated to $\{\overline{\rho}_{\widetilde{v},i}\}$ (we refer to (\ref{block0}) and \S~\ref{sec: Pas} for details).}
      %\item for all $v\in S\setminus S_p$, $\rho_{\pi,\widetilde{v}} \rightsquigarrow \rho_{\widetilde{v}}$.
    \end{enumerate}
  \end{enumerate}
  Then $\rho$ is automorphic, i.e. there exists an automorphic representation $\pi'$ of $G$ such that $\rho\cong \rho_{\pi'}$.
\end{theorem}
\noindent We make a few remarks on the assumptions.
%The main novelty of the paper is that we can deal with certain reducible representations which are neither ordinary nor Fontaine-Laffaille at places above $p$.
The assumptions in 1, 2, 3, the first part of 5, and 7(a), 7(b) are standard for automorphy lifting theorems (e.g. see \cite{CHT}, \cite{Tay08}, \cite{Thor}, \cite{Ger19}, \cite{BLGGT},...). The assumption 4 is crucial for this paper, which gives a necessary condition such that $\rho$ appears in the \emph{$\GL_2(\Q_p)$-ordinary family} that we work with (see the discussions below). The assumption that $p$ splits in $F$ is also crucial because we use some results in $p$-adic Langlands program, those that are only known for $\GL_2(\Q_p)$. The assumption 5(a) is a non-critical assumption, which is used for a classicality criterion; when $\rho_{\widetilde{v}}$ is crystalline and generic for all $v|p$ (as in the assumption 5(b)), we apply the classicality result of Breuil-Hellmann-Schraen \cite{BHS2}\cite{BHS3} to remove such  non-critical assumption. The assumption 6 is rather technical, and we make this assumption so that the Galois deformations are easier to study. Finally the assumption 7(c) is to ensure that certain automorphy lifting of $\overline{\rho}$ can appear in our  $\GL_2(\Q_p)$-ordinary family. One can find analogues of these assumptions (except for 5(b) and the generic assumption 6) in \cite[Thm. 5.11]{Ger19} in ordinary case. Note that, since we crucially use $p$-adic Langlands correspondence for $\GL_2(\Q_p)$, any base-change of $F$ that we can use in this paper has to be split at $p$.
\\

\noindent We sketch the proof of the theorem.
The main object that we work with is (a generalization/variation of) the $\GL_2(\Q_p)$-ordinary families considered in \cite{BD1}. We fix a compact open subgroup $U^p$ of $G(\bA_{F^+}^{\infty,p})$, and let $\widehat{S}(U^p,\co_E):=\{f: G(F^+)\backslash G(\bA^{\infty}_{F^+})/U^p \ra \co_E\ |\ \text{$f$ is continuous}\}$. This $\co_E$-module is equipped with a natural action of $\widetilde{\bT}(U^p)\times G(F^+\otimes_{\Q} \Q_p)$ where $\widetilde{\bT}(U^p)$ is a (semi-local) complete commutative $\co_E$-algebra generated by certain Hecke operators outside $p$ acting on $\widehat{S}(U^p,\co_E)$. To $\overline{\rho}$, we can associate a maximal ideal $\fm_{\overline{\rho}}\subset \widetilde{\bT}(U^p)$, and we denote by $\widehat{S}(U^p,\co_E)_{\overline{\rho}}$ (resp. $\widetilde{\bT}(U^p)_{\overline{\rho}}$) the localisation of $\widehat{S}(U^p,\co_E)$ (resp. $\widetilde{\bT}(U^p)$) at $\fm_{\overline{\rho}}$. We have a natural surjection $R_{\overline{\rho}, \cS}\twoheadrightarrow \widetilde{\bT}(U^p)_{\overline{\rho}}$, where $R_{\overline{\rho}, \cS}$ denotes the universal deformation ring of a certain deformation problem $\cS$ of $\overline{\rho}$.
 \\

\noindent Applying Emerton's $P$-ordinary part functor (see footnote \ref{fn2} for $P$), we obtain an admissible  Banach representation $\Ord_P(\widehat{S}(U^p,\co_E)_{\overline{\rho}})$ of $L_P$. We can further decompose $\Ord_P(\widehat{S}(U^p,\co_E)_{\overline{\rho}})$  using the theory of blocks of  \cite{Pas13}, in particular,  we can associate to the block $\fB$ (as in the theorem) a $\widetilde{\bT}(U^p)_{\overline{\rho}}\times L_P$-equivariant direct summand $\Ord_{P}(\widehat{S}(U^p,\co_E)_{\overline{\rho}})_{\fB}$ of $\Ord_P(\widehat{S}(U^p,\co_E)_{\overline{\rho}})$.  The $\widetilde{\bT}(U^p)_{\overline{\rho}}$-action on $\Ord_P(\widehat{S}(U^p,\co_E)_{\overline{\rho}})_{\fB}$  factors through a certain quotient $\widetilde{\bT}(U^p)_{\overline{\rho},\fB}^{P-\ord}$.
\\

\noindent On the Galois side, we let  $R_{\fB}:=\widehat{\otimes}_{v|p} R_{\fB,\widetilde{v}}:= \widehat{\otimes}_{v|p} \big(R_{\tr \overline{\rho}_{\widetilde{v},1}}^{\pss}\widehat{\otimes} R_{\tr \overline{\rho}_{\widetilde{v},2}}^{\pss}\big)$, where $R_{\tr \overline{\rho}_{\widetilde{v},i}}^{\pss}$ denotes the universal deformation ring of the pseudo-character $\tr \overline{\rho}_{\widetilde{v},i}$. Let $R_{\overline{\rho}_{\widetilde{v}}, \cF_{\widetilde{v}}}^{P_{\widetilde{v}}-\ord, \square}$ be the framed universal $P_{\widetilde{v}}$-ordinary deformation ring of $\overline{\rho}_{\widetilde{v}}$ with respect to the $P_{\widetilde{v}}$-filtration $\cF_{\widetilde{v}}$ on $\overline{\rho}_{\widetilde{v}}$ induced by  (\ref{fil00}) (see \S~\ref{Porddef} for details). There is a natural morphism $R_{\fB,\widetilde{v}} \ra R_{\overline{\rho}_{\widetilde{v}},\cF_{\widetilde{v}}}^{P_{\widetilde{v}}-\ord, \square}$.
Adding the local conditions $\{R_{\overline{\rho}_{\widetilde{v}}, \cF_{\widetilde{v}}}^{P_{\widetilde{v}}-\ord, \square}\}_{v|p}$ to the deformation problem $\cS$, we obtain a quotient $R_{\overline{\rho},\cS,\fB}^{P-\ord}$ of $R_{\overline{\rho}, \cS}$. There is a natural morphism $R_{\fB} \ra R_{\overline{\rho}, \cS,\fB}^{P-\ord}$. One can prove that the composition $R_{\overline{\rho}, \cS}\twoheadrightarrow \widetilde{\bT}(U^p)_{\overline{\rho}} \twoheadrightarrow \widetilde{\bT}(U^p)_{\overline{\rho},\fB}^{P-\ord}$ factors through
\begin{equation*}
  R_{\overline{\rho}, \cS,\fB}^{P-\ord} \lra \widetilde{\bT}(U^p)_{\overline{\rho},\fB}^{P-\ord}.
\end{equation*}
This is the ``$R \ra \bT$" map of the paper.
\\

\noindent We then prove a local-global compatibility result on $\Ord_{P}(\widehat{S}(U^p,\co_E)_{\overline{\rho}})_{\fB}$. We use a similar  formulation as in \cite{Pan2}. The first key point is that the  $L_P$-action on $\Ord_{P}(\widehat{S}(U^p,E)_{\overline{\rho}})_{\fB}$ can be parameterized by $R_{\fB}$ using $p$-adic local Langlands correspondence. More precisely, by the theory of Pa{\v{s}}k{\=u}nas, we can associate to the block $\fB$ an $L_P$-representation $\widetilde{P}_{\fB}$ (which is projective in a certain category, see \S~\ref{sec: Pas}), and we have a natural injection (induced by the $p$-adic Langlands correspondence) $R_{\fB} \hookrightarrow \End_{L_P}(\widetilde{P}_{\fB})$. Put
\begin{equation*}\tm(U^p,\fB):=\Hom_{L_P}\big(\widetilde{P}_{\fB}, \Ord_P(\widehat{S}(U^p,\co_E)_{\overline{\rho}})_{\fB}^d\big)\end{equation*}
where $\Ord_P(\widehat{S}(U^p,\co_E)_{\overline{\rho}})_{\fB}^d$ denotes the Schikhof dual of  $\Ord_P(\widehat{S}(U^p,\co_E)_{\overline{\rho}})_{\fB}$ (which lies in the same category as $\widetilde{P}_{\fB}$). The natural action of $R_{\fB}$ on $\widetilde{P}_{\fB}$ induces an $R_{\fB}$-action on $\tm(U^p,\fB)$. One can moreover show that $\tm(U^p,\fB)$ is a finitely generated $R_{\fB}$-module. We remark that this $R_{\fB}$-action is obtained in a purely \emph{local} way, and characterizes the $L_P$-action on  $\Ord_P(\widehat{S}(U^p,\co_E)_{\overline{\rho}})_{\fB}$. On the other hand, $\tm(U^p, \fB)$ inherits from $\Ord_P(\widehat{S}(U^p,\co_E)_{\overline{\rho}})_{\fB}^d$ an action of $\widetilde{\bT}(U^p)_{\overline{\rho},\fB}^{P-\ord}$, hence is equipped with another $R_{\fB}$-action via
\begin{equation*}
 R_{\fB}  \lra  R_{\overline{\rho}, \cS,\fB}^{P-\ord} \lra \widetilde{\bT}(U^p)_{\overline{\rho},\fB}^{P-\ord}.
\end{equation*}
Note that this $R_{\fB}$-action is obtained in a \emph{global} way (since it comes from the global Galois deformation ring).  Then we show that these two $R_{\fB}$-actions on $\tm(U^p,\fB)$ coincide (up to a certain twist, that we ignore in the introduction). In summary, we find ourselves in a similar situation as Hida's ordinary families (see \S~\ref{GL2ord}, \S~\ref{sec: lg} for details): we have a finite morphism $R_{\fB} \ra \widetilde{\bT}(U^p)_{\overline{\rho},\fB}^{P-\ord}$ and a finitely generated $\widetilde{\bT}(U^p)_{\overline{\rho},\fB}^{P-\ord}$-module $\tm(U^p,\fB)$ with nice properties as $R_{\fB}$-module.
%
 %have the following data: a finitely generated $\widetilde{\bT}(U^p)_{\overline{\rho},\fB}^{P-\ord}$-module $\tm(U^p,\fB)$ and a finite morphism $R_{\fB} \ra \widetilde{\bT}(U^p)_{\overline{\rho},\fB}^{P-\ord}$. We would prefer to view $\Spf \widetilde{\bT}(U^p)_{\overline{\rho},\fB}^{P-\ord}$ as a ``higher" Hida family with $\Spf R_{\fB}$ as a ``higher" weight space.
\\

\noindent We then apply the Taylor-Wiles patching argument (\cite{CEGGPS} \cite{SchoLT}  \cite{Pan2}) to our $\GL_2(\Q_p)$-ordinary families and obtain the following data:
\begin{equation*}
  S_{\infty} \ra R_{\infty} \curvearrowright \tm_1^{\infty}(\fB),
\end{equation*}
where $S_{\infty}$ is a formal power series over $\co_E$, $R_{\infty}$ is a patched global deformation ring, and $\tm_1^{\infty}(\fB)$ is a finitely generated $R_{\infty}$-module, which is flat over $S_{\infty}$. We remark that
\begin{equation*}\dim R_{\infty}=\dim S_{\infty} +\dim (B_p\cap L_P)
 \end{equation*}where $B_p:=\prod_{v|p} B(\Q_p)$. We also have a closed ideal $\fa_1\subset S_{\infty}$ such that $R_{\infty}/\fa_1 \twoheadrightarrow R_{\cS,\overline{\rho}, \fB}^{P-\ord}$, and we have  an $R_{\infty}$-equivariant isomorphism $\tm_1^{\infty}(\fB)/\fa_1\cong \tm(U^p,\fB)$ (where $R_{\infty}$ acts on $\tm(U^p,\fB)$ via the precedent projection).
Using Taylor's Ihara avoidance and arguments on supports of modules, one can show that  $\rho$ appears in the $\GL_2(\Q_p)$-ordinary family, in other words, $\rho$ can be attached to $P$-ordinary $p$-adic automorphic representations. Finally we use the assumption 5 to show that $\rho$ can be attached to classical automorphic representations. Assuming 5(a), the result follows from the  existence of locally algebraic vectors in $\GL_2(\Q_p)$-representations in de Rham case and an adjunction property of the functor $\Ord_P(-)$. For 5(b), we first use $p$-adic local Langlands correspondence for $\GL_2(\Q_p)$ (the result ``trianguline implying finite slope") to show that $\rho$ appears in the eigenvariety, and then deduce the theorem from the classicality result of Breuil-Hellmann-Schraen \cite{BHS2} \cite{BHS3}.
\\

\noindent Finally, under similar assumptions as in Theorem \ref{mainIntro} except the assumption 5, and assuming $\rho$ is automorphic \emph{\`a priori} (which implies $\rho_{\widetilde{v}}$ is de Rham of distinct Hodge-Tate weights for $v|p$), when the Hodge-Tate weights of $\rho_{\widetilde{v},1}$ are not bigger that those of $\rho_{\widetilde{v},2}$ (which is contrary to the assumption 5(a), and is  often referred to as the \emph{critical} case), then the automorphy lifting method allows us to find a \emph{non-classical} point in the $\GL_2(\Q_p)$-ordinary family $\Spf \widetilde{\bT}(U^p)_{\overline{\rho},\fB}^{P-\ord}$ associated to $\rho$. We then deduce from the existence of the non-classical point some results towards Breuil's locally analytic socle conjecture (cf. Theorem \ref{main2}, Remark \ref{soc}). Let us mention that when $\rho_{\widetilde{v}}$ is non-trianguline (i.e. $\rho_{\widetilde{v},1}$ or $\rho_{\widetilde{v},2}$ is not trianguline), these results  provide probably a first known example (to the author's knowledge) on the conjecture in the non-trianguline case.

\subsection*{Acknowledgement}
\noindent This work grows out from discussions with Lue Pan, and I want to thank him for the helpful discussions and for answering my questions. I also want to thank Yongquan Hu for answering my questions during the preparation of this note. The work was supported by  Grant No. 7101502007 and No. 8102600246 from Peking University.
\subsection{Some notations}
\noindent Throughout the paper, $E$ will be a finite extension of $\Q_p$, with $\co_E$ its ring of integers, $\varpi_E$ a uniformizer of $\co_E$, and $k_E:=\co_E/\varpi_E$. Let $\varepsilon: \Gal_{\Q_p} \ra E^{\times}$ denote the cyclotomic character, $\omega: \Gal_{\Q_p} \ra k_E^{\times}$ the modulo $p$ cyclotomic character. We use the convention that the Hodge-Tate weight of $\varepsilon$ is $1$. We normalize local class field theory by sending a uniformizer to a (lift of the) geometric
Frobenius. In this way, we view characters of $\Gal_{\Q_p}$ as characters of $\Q_p^{\times}$ without further
mention.
\\

\noindent For a torsion $\co_E$-module $N$, let $N^{\vee}:=\Hom_{\co_E}(N, E/\co_E)$ be the Pontryagain dual of $N$.  For $M$ a $p$-adically complete torsion free $\co_E$-module (so $M\cong \varprojlim_n M/\varpi_E^n$),  we let
\begin{equation*}
  M^d:=\Hom_{\co_E}(M, \co_E)
\end{equation*}
equipped with the point-wise convergence topology, be the Schikhof dual of $M$. We have
\begin{equation}\label{duala}
  \Hom_{\co_E}(M,\co_E)\cong \varprojlim_n \Hom_{\co_E/\varpi_E^n}(M/\varpi_E^n, \co_E/\varpi_E^n)\cong \varprojlim_n (M/\varpi_E^n)^{\vee},
\end{equation}
where the map $(M/\varpi_E^n)^{\vee}\twoheadrightarrow (M/\varpi_E^{n-1})^{\vee}$ is induced by the injection $M/\varpi_E^{n-1} \xrightarrow{\varpi_E} M/\varpi_E^n$. We also have (e.g. see the proof of \cite[Thm. 1.2]{ST})
\begin{equation}\label{duali}
  M\cong \Hom_{\co_E}^{\cts}(M^d, \co_E)
\end{equation}
where the right hand side is equipped with the compact-open topology.

\section{$P$-ordinary Galois deformations}\label{Porddef}
\noindent In this section, for $P$ a parabolic subgroup of $\GL_n$, we study $P$-ordinary Galois deformations, and show some (standard) properties of $P$-ordinary Galois deformation rings.
\\

\noindent Let $L$ be a finite extension over $\Q_p$. We enlarge $E$ such that $E$ contains all the embeddings of $L$ in $\overline{\Q_p}$.  Let $B$ be the Borel subgroup of $\GL_n$ of upper triangular matrices, and let $P$ be a parabolic subgroup containing $B$ with a Levi subgroup  $L_P$ given by
(where $\sum_{i=1}^k n_i=n$):
\begin{equation}\label{equ: levi}
 \begin{pmatrix}
  \GL_{n_1} & 0 & \cdots & 0 \\
  0 & \GL_{n_2} & \cdots & 0 \\
  \vdots & \vdots & \ddots & 0 \\
  0 & 0 &\cdots & \GL_{n_k}
 \end{pmatrix}.
\end{equation}
Denote by $s_i:=\sum_{j=0}^{i-1} n_j$ where we set $n_0=0$ (hence $s_1=0$).
\\

\noindent Let $(\overline{\rho}, V_{k_E})$ be an $n$-dimensional $P$-ordinary representation of $\Gal_L$ over $k_E$ in the sense of  \cite[Def. 5.1]{BD1}, i.e.  there exists an increasing $\Gal_L$-equivariant filtration
\begin{equation*}
 \cF:\  0=\Fil^0 V_{k_E} \subsetneq \Fil^1 V_{k_E}\subsetneq \cdots \subsetneq \Fil^{k} V_{k_E}=V_{k_E}
\end{equation*}
such that $\dim_{k_E} \gr^i \cF:=\Fil^i V_{k_E}/\Fil^{i-1} V_{k_E}=n_i$. Let  $(\overline{\rho}_i, \gr^i \cF)$ be the $\Gal_L$-representation given by the graded piece. We choose a basis $\{e_1, \cdots, e_n\}$ of $V_{k_E}$ such that $\{e_1,\cdots, e_{s_i}\}$ is a basis of $\Fil^{i-1} V_{k_E}$ for all $i$. Under this basis, $\overline{\rho}$ corresponds to a continuous morphism $\overline{\rho}: \Gal_L \ra P(k_E)$.
\\

\noindent  Let $\Art(\co_E)$ be the category of local artinian $\co_E$-algebras with residue field $k_E$ and $\Def^{\square}_{\overline{\rho}}$ the functor of framed deformations of $\overline{\rho}$, i.e. the functor from $\Art(\co_E)$ to sets which sends $A\in \Art(\co_E)$ to the set $\{\rho_A: \Gal_L \ra \GL_n(A)\ |\ \rho_A \equiv \rho \pmod{\fm_A}\}$. Let $\Def^{P-\ord, \square}_{\overline{\rho}, \cF}$ be the subfunctor of $\Def^{\square}_{\overline{\rho}}$ which sends $A\in \Art(\co_E)$ to the set of $\rho_A\in \Def^{\square}_{\overline{\rho}}(A)$ satisfying that the underlying $A$-module $V_A$ of $\rho_A$ admits an increasing filtration $\cF_A=\Fil^\bullet V_A$ by $\Gal_L$ invariant free $A$ submodules which are direct summands as $A$-modules such that $\Fil^i V_A \cong \Fil^i V_{k_E} \pmod{\fm_A}$. We assume the following hypothesis.
\begin{hypothesis}\label{hypo: gene1}
  Suppose $\Hom_{\Gal_L}(\overline{\rho}_i, \overline{\rho}_j)=0$ for all $i\neq j$.
\end{hypothesis}
\noindent By the same argument as in the proof of \cite[Lem. 5.3]{BD1}, we have:
\begin{lemma}\label{pord0}
  Assume Hypothesis \ref{hypo: gene1} and let $\rho_A \in \Def^{\square}_{\overline{\rho}}$.

\noindent (1) $\rho_A\in \Def^{P-\ord,\square}_{\overline{\rho},\cF}$ if and only if there exists $M\in \GL_n(A)$ such that $M\rho_A M^{-1}$ has image in $P(A)$, and $M\rho_A M^{-1}\equiv \overline{\rho} \pmod{\fm_A}$.

\noindent   (2) Suppose there exist $M_1$, $M_2\in \GL_n(A)$ such that $M_i\rho_A M_i^{-1}$ has image in $P(A)$, and $M_i\rho_A M_i^{-1} \equiv \overline{\rho} \pmod{\fm_A}$, then $M_1M_2^{-1}\in P(A)$.
\end{lemma}
\noindent In particular, if $\rho_A\in \Def^{P-\ord, \square}_{\overline{\rho}, \cF}(A)$, then the associated increasing filtration $\Fil^{\bullet} V_A$ is unique. Recall that $\Def^{\square}_{\overline{\rho}}$ is pro-representable by a complete local noetherian $\co_E$-algebra $R_{\overline{\rho}}^{\square}$ of residue field $k_E$.
\begin{proposition}\label{propPord}
  Assume Hypothesis \ref{hypo: gene1}, the functor $\Def^{P-\ord,\square}_{\overline{\rho},\cF}$ is pro-representable by a complete local noetherian $\co_E$-algebra $R_{\overline{\rho},\cF}^{P-\ord, \square}$, which is a quotient of $R^{\square}_{\overline{\rho}}$.
\end{proposition}
\begin{proof}By Schlessinger's criterion, the fact that $\Def^{P-\ord,\square}_{\overline{\rho},\cF}$ is a subfunctor of $\Def^{\square}_{\overline{\rho}}$ which is pro-representable, it suffices to show that given morphisms $f_1: A\ra C$, $f_2: B\ra C$ in $\Art(\co_E)$ with $f_2$ surjective and small, the induced maps
  \begin{equation*}
    \Def^{P-\ord,\square}_{\overline{\rho},\cF}(A\times_C B) \lra \Def^{P-\ord,\square}_{\overline{\rho},\cF}(A)\times_{\Def^{P-\ord,\square}_{\overline{\rho},\cF}(C)} \Def^{P-\ord,\square}_{\overline{\rho},\cF}(B)
  \end{equation*}
  is surjective. Let $(\rho_A, \rho_B)\in  \Def^{P-\ord,\square}_{\overline{\rho},\cF}(A)\times_{\Def^{P-\ord,\square}_{\overline{\rho},\cF}(C)} \Def^{P-\ord,\square}_{\overline{\rho},\cF}(B)$ and let $\tilde{\rho}$ be a lifting of $(\rho_A,\rho_B)$ in  $\Def^{\square}_{\overline{\rho}}(A\times_C B)$. By Lemma \ref{pord0} (1), there exists $M_A\in \GL_n(A)$ (resp. $M_B$)  such that $M_A\rho_AM_A^{-1}$ (resp. $M_B\rho_B M_B^{-1}$) has image in $P(A)$ (resp. in $P(B)$) and that $M_A\rho_A M_1^{-1}\equiv \overline{\rho}\pmod{\fm_A}$ (resp. $M_B \rho_B M_B^{-1}\equiv \overline{\rho} \pmod{\fm_B}$). Denote by $\overline{M_A}$ and $\overline{M_B}$ the image of $M_A$, $M_B$ in $\GL_n(C)$ respectively. By Lemma \ref{pord0} (2), there exists $N_C\in P(C)$ such that $\overline{M_B}=N_C\overline{M_A}$. Let $N_B\in P(B)$ be a lifting of $N_C$. Then we see $N_BM_B \rho_B M_B^{-1} N_B^{-1}$ has image in $P(B)$. Let $\tilde{M}$ be a lifting of $(M_A, N_B M_B)\in \GL_n(A\times_C B)$. It is easy to check that $\tilde{M}\tilde{\rho}\tilde{M}^{-1}$ has image in $P(A\times_C B)\cong P(A)\times_{P(C)} P(B)$. The proposition follows.
\end{proof}
\noindent Let $\Hom_{\cF}(V_{k_E}, V_{k_E})$ be the $k_E$-vector subspace of $\Hom_{k_E}(V_{k_E}, V_{k_E})$ consisting of morphisms of filtered $k_E$-vector spaces, i.e. $f: V_{k_E} \ra V_{k_E}$ lies in $\Hom_{\cF}(V_{k_E}, V_{k_E})$ if and only if $f|_{\Fil^i V_{k_E}} \subseteq \Fil^i V_{k_E}$ for all $i$. We have the following easy lemma.
\begin{lemma}\label{Pord: dim1}
  $\dim_{k_E} \Hom_{\cF}(V_{k_E}, V_{k_E})=\sum_{i=1}^k (n_i(n-s_i))$.
\end{lemma}
\begin{proof}
Using the basis $\{e_1,\cdots, e_n\}$ of $V_{k_E}$, we identify the $k_E$-vector space $\Hom_{k_E}(V_{k_E}, V_{k_E})$ (resp. $\Hom_{\cF}(V_{k_E}, V_{k_E})$) with $M_n(k_E)$ (resp. $\fp(k_E)$) (where $\fp$ denotes the Lie algebra of $P$). The lemma follows.
\end{proof}
\noindent The $k_E$-vector space $\Hom_{k_E}(V_{k_E}, V_{k_E})$ is equipped with a natural $\Gal_L$-action given by
\begin{equation}\label{equ: adj}
(gf)(v)=gf(g^{-1}v),\end{equation} and we denote by $\Ad \overline{\rho}$ the corresponding representation. Since the filtration $\cF$ on $V_{k_E}$ is $\Gal_L$-equivariant, one easily check $\Hom_{\cF}(V_{k_E}, V_{k_E})$ is $\Gal_L$-invariant. We denote by $\Ad_{\cF}\overline{\rho}$ the corresponding $\Gal_L$-representation (which is  a subrepresentation of $\Ad \overline{\rho}$).
\begin{proposition}
Assume Hypothesis \ref{hypo: gene1},  we have a natural isomorphism of $k_E$ vector-spaces
  \begin{equation*}
 \Def^{P-\ord,\square}_{\overline{\rho},\cF}(k_E[\epsilon]/\epsilon^2) \cong B^1(\Gal_L, \Ad_{\cF} \overline{\rho}) + Z^1(\Gal_L, \Ad \overline{\rho}),
  \end{equation*}
  where we use the standard notation with $B^1$ for the $1$-coboundary and $Z^1$ for the $1$-cocycle.
\end{proposition}
\begin{proof}
  Let $\tilde{\rho}\in  \Def^{\square}_{\overline{\rho}}(k_E[\epsilon]/\epsilon^2)$, and $c\in B^1(\Gal_L, \Ad \overline{\rho})$ be the associated $1$-th coboundary, i.e.  $\tilde{\rho}(g)=\overline{\rho}(g)(1+c(g)\epsilon)$ for $g\in \Gal_L$. It is sufficient to show $\tilde{\rho}\in \Def^{P-\ord,\square}_{\overline{\rho},\cF}(k_E[\epsilon]/\epsilon^2)$ if and only if $c \in B^1(\Gal_L, \Ad_{\cF} \overline{\rho}) + Z^1(\Gal, \Ad \overline{\rho})$.

\noindent Suppose $\tilde{\rho}\in \Def^{P-\ord,\square}_{\overline{\rho},\cF}(k_E[\epsilon]/\epsilon^2)$. By Lemma \ref{pord0}, there exists $M \in \GL_n(k_E[\epsilon]/\epsilon^2)$ such that $M\tilde{\rho}M^{-1}$ has image in $P(k_E[\epsilon]/\epsilon^2)$, $M\tilde{\rho}M^{-1}\equiv \overline{\rho} \pmod{\epsilon}$ and $M$ modulo $\epsilon$ lies in $P(k_E)$. There exist then $U\in P(k_E)$ and $A\in M_n(k_E)$ such that $M=U(1+A\epsilon)$. For any $g\in \Gal_L$, using
$(1+A\epsilon) \tilde{\rho}(g) (1-A\epsilon) \in P(k_E[\epsilon]/\epsilon^2)$,
  we deduce
  \begin{equation}\label{equ: cg}c(g)+\overline{\rho}(g)^{-1} A \overline{\rho}(g) - A \in \fp(k_E),\end{equation}
  hence  $c(g)\in B^1(\Gal_L, \Ad_{\cF} \overline{\rho}) + Z^1(\Gal, \Ad \overline{\rho})$.

\noindent Conversely, if $c(g)\in B^1(\Gal_L, \Ad_{\cF} \overline{\rho}) + Z^1(\Gal, \Ad \overline{\rho})$, there exists $A\in M_n(k_E)$ such that (\ref{equ: cg}) holds, and it is easy to check $(1+A\epsilon)\tilde{\rho}(1-A\epsilon)$ has image in $P(k_E[\epsilon]/\epsilon^2)$ and is equal to $\overline{\rho}$ modulo $\epsilon$. This concludes the proof.
\end{proof}

\begin{lemma}\label{Pordsm}
Assume Hypothesis \ref{hypo: gene1},   if $H^2(\Gal_L, \Ad_{\cF} \overline{\rho})=0$, then $R_{\overline{\rho}, \cF}^{P-\ord, \square}$ is formally smooth over $\co_E$.
\end{lemma}
\begin{proof}
Let $A\twoheadrightarrow A/I$ be a small extension (i.e. $I=(\epsilon)$ with $\epsilon \fm_A=0$). It is sufficient to show that the natural map $\Def^{P-\ord,\square}_{\overline{\rho},\cF}(A) \ra\Def^{P-\ord,\square}_{\overline{\rho},\cF}(A/I)$ is surjective. Let $\rho_{A/I} \in \Def^{P-\ord,\square}_{\overline{\rho},\cF}(A/I)$, replacing $\rho_{A/I}$ be a certain conjugate of $\rho_{A/I}$, we assume $\rho_{A/I}$ has image in $P(A/I)$, and it is sufficient to show there exists $\rho_{A}: \Gal_L \ra P(A)$ such that $\rho_A\equiv \rho_{A/I} \pmod{\epsilon}$. Let $\rho_A: \Gal_L \ra P(A)$ be a set theoretic lift of $\rho_{A/I}$. By standard arguments in Galois deformation theory, the obstruction for $\rho_A$ being a group homomorphism corresponds to an element $c\in H^2(\Gal_L, \Ad_{\cF} \overline{\rho})$ given by $\rho_A(g_1, g_2)\rho_A(g_2)^{-1} \rho_A(g_1)^{-1}=1+c(g_1,g_2)\epsilon$ for $g_1$, $g_2\in \Gal_L$. Since $H^2(\Gal_L, \Ad_{\cF} \overline{\rho})=0$, the existence of a homomorphism $\rho_A$ follows, from which we deduce the lemma.
\end{proof}
\noindent For $1\leq i<j \leq k$, we denote by $\overline{\rho}_i^j:=\Fil^{j} \overline{\rho} /\Fil^{i-1} \overline{\rho}$.
\begin{lemma}\label{lem: H2van}
  Suppose that for any $i$, $\Hom_{\Gal_L}(\overline{\rho}_1^{i}, \overline{\rho}_i \otimes_{k_E} \omega)=0$, then $H^2(\Gal_L, \Ad_{\cF} \overline{\rho})=0$.
\end{lemma}
\begin{proof}
  We have a natural $\Gal_L$-equivariant exact sequence
  \begin{equation*}
    0 \lra \Hom_{k_E}(\overline{\rho}_k, \overline{\rho}) \lra \Hom_{\cF}(\overline{\rho}, \overline{\rho}) \lra \Hom_{\cF}(\Fil^{k-1} \overline{\rho}, \Fil^{k-1} \overline{\rho}) \lra 0,
  \end{equation*}
  where $\Fil^{k-1} \overline{\rho}$ is equipped with the induced filtration. By assumption, $\Ext^2_{\Gal_L}(\overline{\rho}_k, \overline{\rho})=0$. The lemma follows then by an easy d\'evissage/induction argument.
 \end{proof}
\begin{corollary}\label{Porddim0}
Assume Hypothesis \ref{hypo: gene1}, and keep the assumption in Lemma \ref{lem: H2van}. Then $R_{\overline{\rho}, \cF}^{P-\ord, \square}$ is formally smooth of relative dimension $n^2+[L:\Q_p]\sum_{i=1}^k (n_i(n-s_i))$ over $\co_E$.
\end{corollary}
\begin{proof}
It is not difficult to see that  $\Hom_{k_E}(\overline{\rho}, \overline{\rho})/\Hom_{\cF}(\overline{\rho}, \overline{\rho})$ is isomorphic (as a $\Gal_L$-representation) to a successive extension of $\Hom_{k_E}(\overline{\rho}_i, \overline{\rho}_j)$ with $i\neq j$. By Hypothesis \ref{hypo: gene1} and d\'evissage, we deduce that $H^0(\Gal_L, \Ad_{\cF} \overline{\rho}) \xrightarrow{\sim} H^0(\Gal_L, \Ad \overline{\rho})$ and
\begin{equation*}
  H^1(\Gal_L, \Ad_{\cF} \overline{\rho}) \hooklongrightarrow H^1(\Gal_L, \Ad \overline{\rho}).
\end{equation*}
Consequently, we deduce $Z^1(\Gal_L, \Ad \overline{\rho}) \cap B^1(\Gal_L, \Ad_{\cF}\overline{\rho})=Z^1(\Gal_L, \Ad_{\cF} \overline{\rho})$. Hence
\begin{multline*}
  \dim_{k_E}(B^1(\Gal_L, \Ad_{\cF} \overline{\rho}) + Z^1(\Gal_L, \Ad \overline{\rho}))\\=\dim_{k_E} H^1(\Gal_L, \Ad_{\cF} \overline{\rho}) + \dim_{k_E} Z^1(\Gal_L, \Ad \overline{\rho})
\\  =\dim_{k_E} H^1(\Gal_L,\Ad_{\cF} \overline{\rho})+n^2-\dim_{k_E} H^0(\Gal_L, \Ad \overline{\rho})
\\  =n^2+[L:\Q_p]\sum_{i=1}^k (n_i(n-s_i)),
\end{multline*}
where the last equation follows from the Euler characteristic formula, Lemma \ref{Pord: dim1} and Lemma \ref{lem: H2van}. Together with Lemma \ref{Pordsm}, the corollary follows.
\end{proof}
\section{$P$-ordinary automorphic representations}\label{ordAut}
\noindent In this section, we recall (and generalize) some results of \cite[\S~6]{BD1} on $P$-ordinary automorphic representations.
\subsection{Global setup}\label{prelprel}
\noindent
We fix field embeddings $\iota_{\infty}: \overline{\Q} \hookrightarrow \bC$, $\iota_p: \overline{\Q} \hookrightarrow \overline{\Q_p}$. We also fix $F^+$ a totally real number field, $F$ a quadratic totally imaginary extension of $F^+$ such that any place of $F^+$ above $p$ is split in $F$, and $G/F^+$ a unitary group attached to the quadratic extension $F/F^+$ as in \cite[\S~6.2.2]{Bch} such that $G\times_{F^+} F\cong \GL_n$ ($n\geq 2$) and $G(F^+\otimes_{\Q} \bR)$ is compact. For a finite place $v$ of $F^+$ which is totally split in $F$, we fix a place $\tilde{v}$ of $F$ dividing $v$, and we have an isomorphism $ i_{G, \widetilde{v}}: G(F^+_v) \xrightarrow{\sim} \GL_n(F_{\tilde{v}})$. We let $S_p$ denote the set of places of $F^+$ dividing $p$. For  an open compact subgroup $U^p=\prod_{v\nmid p} U_v$ of $G(\bA_{F^+}^{p,\infty})$, we put
\begin{equation*}
\widehat{S}(U^p,E):=\Big\{f: G(F^+) \setminus G(\bA_{F^+}^{\infty})/U^p\lra E,\ f \text{ is continuous}\Big\}.
\end{equation*}
Since $G(F^+\otimes_{\Q} \bR)$ is compact, $G(F^+)\setminus G(\bA_{F^+}^{\infty})/U^p$ is a profinite set, and we see that $\widehat{S}(U^p,E)$ is a Banach space over $E$ with the norm defined by the (complete) $\co_E$-lattice:
\begin{equation*}
\widehat{S}(U^p,\co_E):=\Big\{f: G(F^+)\setminus G(\bA_{F^+}^{\infty})/U^p \lra \co_E,\ f \text{ is continuous}\Big\}.\end{equation*}
Moreover, $\widehat{S}(U^p,E)$ is equipped with a continuous action of $G(F^+\otimes_{\Q}\Q_p)$ given by $(g'f)(g)=f(gg')$ for $f\in \widehat{S}(U^p,E)$, $g'\in G(F^+\otimes_{\Q} \Q_p)$, $g\in G(\bA_{F^+}^{\infty})$. The lattice $\widehat{S}(U^p,\co_E)$ is obviously stable by this action, so the Banach representation $\widehat{S}(U^p,E)$ of $G(F^+\otimes_{\Q} \Q_p)$ is unitary.
\\

\noindent Let $S$ be a finite set of finite places of $F^+$ consisting of those $v$ such that $v|p$, or $v$ ramifies in $F$, or $v$ is unramified and $U_v$ is not maximal hyperspecial. Let $\bT(U^p):=\co_E[T_{\tilde{v}}^{(j)}]$ be the commutative polynomial $\co_E$-algebra generated by the formal variables $T_{\tilde{v}}^{(j)}$ where $j\in \{1,\cdots,n\}$ and $v\notin S$ splits in $F$. The $\co_E$-algebra $\bT(U^p)$ acts on $\widehat{S}(U^p,E)$ and $\widehat{S}(U^p, \co_E)$ by making $T_{\tilde{v}}^{(j)}$ act by the double coset operator:
\begin{equation}\label{equ: ord-hecke1}
 T_{\tilde{v}}^{(j)}:=\Big[U_{v} g_v i_{G,\tilde{v}}^{-1}\begin{pmatrix}
  \text{\textbf{1}}_{n-j} & 0 \\ 0 & \text{$\varpi_{\tilde{v}}$ \textbf{1}}_{j}
 \end{pmatrix} g_v^{-1}U_{v}\Big]
\end{equation}
where $\varpi_{\tilde{v}}$ is a uniformizer of $F_{\tilde{v}}$, and where $g_v\in G(F_v^+)$ is such that $i_{G,\tilde{v}}(g_v^{-1} U_v g_v)=\GL_n(\co_{F_{\tilde{v}}})$. This action commutes with that of $G(F^+ \otimes_{\Q} \Q_p)$.\\

\noindent
Recall that the automorphic representations of $G(\bA_{F^+})$ are the irreducible constituents of the $\bC$-vector space of functions $f: G(F^+)\backslash G(\bA_{F^+})\lra \bC$ which are:
\begin{itemize}
 \item $\cC^{\infty}$ when restricted to $G(F^+\otimes_{\Q} \bR)$
 \item locally constant when restricted to $G(\bA_{F^+}^{\infty})$
 \item $G(F^+\otimes_{\Q} \bR)$-finite,
\end{itemize}
where $G(\bA_{F^+})$ acts on this space via right translation. An automorphic representation $\pi$ is isomorphic to $\pi_{\infty}\otimes_{\bC} \pi^{\infty}$ where $\pi_{\infty}=W_{\infty}$ is an irreducible algebraic representation of $(\Res_{F^+/\Q}G)(\bR)=G(F^+\otimes_{\Q} \bR)$ over $\bC$ and $\pi^{\infty}\cong \Hom_{G(F^+\otimes_{\Q}\bR)}(W_{\infty}, \pi)\cong \otimes'_{v}\pi_v$ is an irreducible smooth representation of $G(\bA_{F^+}^{\infty})$. The algebraic representation $W_{\infty}\vert_{(\Res_{F^+/\Q}G)(\Q)}$ is defined over $\overline{\Q}$ via $\iota_{\infty}$ and we denote by $W_p$ its base change to $\overline{\Q_p}$ via $\iota_p$, which is thus an irreducible algebraic representation of $(\Res_{F^+/\Q}G)(\Q_p)=G(F^+\otimes_{\Q} \Q_p)$ over $\overline{\Q_p}$. Via the decomposition $G(F^+\otimes_{\Q} \Q_p)\xrightarrow{\sim} \prod_{v\in S_p} G(F^+_v)$, one has $W_p\cong \otimes_{v\in S_p} W_v$ where $W_v$ is an irreducible algebraic representation of $G(F^+_v)$ over $\overline{\Q_p}$. One can also prove $\pi^{\infty}$ is defined over a number field via $\iota_{\infty}$ (e.g. see \cite[\S~6.2.3]{Bch}). Denote by $\pi^{\infty,p}:=\otimes'_{v\nmid p} \pi_v$, so that we have $\pi^\infty \cong \pi^{\infty,p}\otimes_{\overline{\Q}} \pi_p$ (seen over $\overline{\Q}$ via $\iota_\infty$), and by $m(\pi)\in \Z_{\geq 1}$ the multiplicity of $\pi$ in the above space of functions $f: G(F^+)\backslash G(\bA_{F^+})\ra \bC$. Denote by $\widehat{S}(U^p,E)^{\lalg}$ the subspace of $\widehat{S}(U^p,E)$ of locally algebraic vectors for the $(\Res_{F^+/\Q}G)(\Q_p)=G(F^+\otimes_{\Q} \Q_p)$-action, which is stable by $\bT(U^p)$. We have an isomorphism which is equivariant under the action of $G(F^+\otimes_{\Q} \Q_p)\times \bT(U^p)$ (see e.g. \cite[Prop. 5.1]{Br13II} and the references in \cite[\S~5]{Br13II}):
\begin{equation}\label{equ: lalgaut}
 \widehat{S}(U^p,E)^{\lalg} \otimes_E \overline{\Q_p} \cong \bigoplus_{\pi}\Big((\pi^{\infty,p})^{U^p} \otimes_{\overline{\Q}}(\pi_p \otimes_{\overline{\Q}} W_p)\Big)^{\oplus m(\pi)}
\end{equation}
where $\pi\cong \pi_{\infty}\otimes_{\overline{\Q}} \pi^{\infty}$ runs through the automorphic representations of $G(\bA_{F^+})$ and $W_p$ is associated to $\pi_{\infty}=W_\infty$ as above, and where $T_{\tilde{v}}^{(j)}\in \bT(U^p)$ acts on $(\pi^{\infty,p})^{U^p}$ by the double coset operator (\ref{equ: ord-hecke1}).\\

\noindent
Following \cite[\S~3.3]{CHT}, we say that $U^p$ is {\it sufficiently small} if there is a place $v\nmid p$ such that $1$ is the only element of finite order in $U_v$. We have (e.g. see \cite[Lem. 6.1]{BD1})
\begin{lemma}\label{projH}
Assume $U^p$ sufficiently small, then for any compact open subgroup $U_p$ of $G(F^+\otimes_{\Q} \Q_p)$ there is an integer $r\geq 1$ such that $\widehat{S}(U^p,\co_E)|_{U_{p}}$ is isomorphic to $\cC(U_p,\co_E)^{\oplus r}$.
\end{lemma}
\noindent In the following, we assume $U^p$ sufficiently small.
\subsection{Galois deformations}\label{sec: galdef}
\noindent Let $\cG_n$ be the group scheme over $\Z$ which is the semi-direct product of $\{1,\jmath\}$ acting on $\GL_1 \times \GL_n$ via
\begin{equation*}
  \jmath (\mu, g)\jmath^{-1}=(\mu,(g^t)^{-1}\mu).
\end{equation*}
Denote by $\nu: \cG_n \ra \GL_1$ the morphism given by $(\mu,g) \jmath\mapsto -\mu$.
Let $\overline{\rho}: \Gal_{F^+} \ra \cG_n(k_E)$ be a continuous representation such that
\begin{itemize}\item $\overline{\rho}(\Gal_F)\subseteq \GL_1(k_E)\times \GL_n(k_E)$,
 \item $\overline{\rho}$ is unramified outside $S$,
 \item $\overline{\rho}(c)\in \cG_n(k_E)\setminus \GL_n(k_E)$ (where $c$ denote the complex conjugation),
 \item the composition $\Gal_{F^+} \xrightarrow{\overline{\rho}} \cG_n(k_E) \xrightarrow{\nu} k_E^{\times}$ is equal to $\omega^{1-n} \delta_{F/F^+}^n$, where $\delta_{F/F^+}$ is the unique non-trivial character of $\Gal(F/F^+)$.
\end{itemize}
Denote by $\overline{\rho}_F$ the composition $\Gal_F \xrightarrow{\overline{\rho}} \GL_1 \times \GL_n \twoheadrightarrow \GL_n(k_E)$ (where the second the map denotes the natural projection), then we have $\overline{\rho}_F^c \cong \overline{\rho}_F^{\vee}\otimes_{k_E} \chi_{\cyc}^{1-n}$. For a place $v$ of $F^+$ such that $v=\widetilde{v}\widetilde{v}^c$ in $F$, denote by $\overline{\rho}_{\widetilde{v}}:=\overline{\rho}_F |_{F_{\widetilde{v}}}$.
\\

\noindent Denote by $\cS$ the following deformation problem (cf. \cite[\S~2.3]{CHT} \cite[\S~3]{Thor})
\begin{equation*}
  (F/F^+, S, \widetilde{S}, \co_E, \overline{\rho}, \varepsilon^{1-n}\delta_{F/F^+}, \{R_{\overline{\rho}_{\widetilde{v}}}^{\bar{\square}}\}_{v\in S}),
\end{equation*}
where $R_{\overline{\rho}_{\widetilde{v}}}^{\bar{\square}}$ denotes the reduced quotient of the universal framed deformation ring of $\overline{\rho}_{\widetilde{v}}$, and $\widetilde{S}=\{\widetilde{v}\ |\ v\in S\}$. Suppose $\overline{\rho}$ is absolutely irreducible. By \cite[Prop. 2.2.9]{CHT}, the deformation problem $\cS$ is pro-represented by a complete local noetherian $\co_E$-algebra $R_{\overline{\rho}, \cS}$. Denote by $R_{\overline{\rho}, \cS}^{\square}$ the \emph{$S$-framed} deformation ring of $\cS$-deformations (cf. \cite[Def. 3.1]{Thor}). By definition, $R_{\overline{\rho}, \cS}^{\square}$ is formally smooth over $R_{\overline{\rho}, \cS}$ of relative dimension $n^2|S|$. Let $R^{\loc}:=\widehat{\otimes}_{v\in S} R_{\overline{v}}^{\bar{\square}}$, where the tensor product is taken over $\co_E$. We have by definition a natural morphism
\begin{equation*}
  R^{\loc} \lra R_{\overline{\rho}, \cS}^{\square}.
\end{equation*}
\subsection{Hecke operators}\label{sec: ord-hecke}

\noindent
We recall the definition of some useful pro-$p$-Hecke algebras and of their localisations.\\

\noindent
For $s\in \Z_{>0}$ and a compact open subgroup $U_p$ of $G(F^+\otimes_{\Q} \Q_p)\cong \prod_{v\in S_p} \GL_n(F_{\widetilde{v}})$, we let $\bT(U^pU_p, \co_E/\varpi_E^s)$ (resp. $\bT(U^pU_p, \co_E)$) be the $\co_E/\varpi_E^s$-subalgebra (resp. $\co_E$-subalgebra) of the endomorphism ring of $S(U^{p}U_{p}, \co_E/\varpi_E^s)$ (resp. $S(U^{p}U_{p}, \co_E)$) generated by the operators in $\bT(U^p)$. Then $\bT(U^{p}U_{p}, \co_E/\varpi_E^s)$ is a finite $\co_E/\varpi_E^s$ algebra \big(resp. $\bT(U^{p}U_{p}, \co_E)$ is an $\co_E$-algebra which is finite free as $\co_E$-module\big). We have
\begin{equation}\label{equ: ord-redT}
\bT(U^{p}U_{p}, \co_E)\xlongrightarrow{\sim} \varprojlim_s \bT(U^{p}U_{p}, \co_E/\varpi_E^s),
\end{equation}
\begin{equation}\small\label{tildeproj}
\widetilde{\bT}(U^{p}):=\varprojlim_s \varprojlim_{U_{p}} \bT(U^{p}U_{p}, \co_E/\varpi_E^s)
\cong \varprojlim_{U_{p}}\varprojlim_s \bT(U^{p}U_{p}, \co_E/\varpi_E^s) \cong \varprojlim_{U_{p}} \bT(U^{p}U_{p}, \co_E).
\end{equation}
We have as in \cite[Lem. 6.3]{BD1}:
\begin{lemma}\label{lem: ord-hecke}
The $\co_E$-algebra $\widetilde{\bT}(U^{p})$ is reduced and acts faithfully on $\widehat{S}(U^{p}, E)$.
\end{lemma}

\noindent
To $\overline{\rho}$ (as in \S~\ref{sec: galdef}), we associate a maximal ideal   $\fm_{\overline{\rho}}$ of residue field $k_E$ of $\bT(U^p)$
 such that for $v\notin S$ splitting in $F$,   the characteristic polynomial of ${\overline{\rho}_{\widetilde{v}}}(\Frob_{\tilde{v}})$, where $\Frob_{\tilde{v}}$ is a {\it geometric} Frobenius at $\tilde v$, is given by:
\begin{equation}\label{equ: ord-ideal}
X^n+\cdots + (-1)^j (\Nm\tilde{v})^{\frac{j(j-1)}{2}} \theta_{{\rho}}(T_{\tilde{v}}^{(j)}) X^{n-j}+\cdots +(-1)^n(\Nm\tilde{v})^{\frac{n(n-1)}{2}} \theta_{{\rho}}(T_{\tilde{v}}^{(n)})
\end{equation}
where $\Nm\tilde{v}$ is the cardinality of the residue field at $\tilde v$ and $\theta_{\overline{\rho}}: \bT(U^p)/\fm_{\overline{\rho}}\buildrel\sim\over\lra k_E$. For a $\bT(U^p)$-module $M$, denote by $M_{\overline{\rho}}$ the localisation of $M$ at $\fm(\overline{\rho})$.
Recall a maximal ideal $\fm(\overline{\rho})$ of $\bT(U^p)$ is called \emph{$U^p$-automorphic} if there exist $s$, $U_{p}$ as above such that the localisation $S(U^{p}U_{p}, \co_E/\varpi_E^s)_{\overline{\rho}}$ is nonzero. Suppose $\fm(\overline{\rho})$ is $U^p$-automorphic, then $\fm(\overline{\rho})$ corresponds to a maximal ideal, still denoted by $\fm(\overline{\rho})$, of $\widetilde{\bT}(U^p)$. The localisation $\widetilde{\bT}(U^p)_{\overline{\rho}}$ is a direct factor of $\widetilde{\bT}(U^p)$ (e.g. by \cite[Lem. 6.5]{BD1}),  and there is a natural isomorphism
\begin{equation}\label{projtrho}
\widetilde{\bT}(U^{p})_{\overline{\rho}}\cong \varprojlim_s \varprojlim_{U_{p}} \bT(U^{p}U_{p}, \co_E/\varpi_E^s)_{\overline{\rho}}
\cong \varprojlim_{U_p} \bT(U^{p}U_{p}, \co_E)_{\overline{\rho}}.
\end{equation}
Put $\widehat{S}(U^{p}, \co_E)_{\overline{\rho}}:= \varprojlim_s \varinjlim_{U_p} S(U^{p}U_{p}, \co_E/\varpi_E^s)_{\overline{\rho}}$.
We have as in \cite[\S~6.2]{BD1}:
\begin{lemma}Suppose $\overline{\rho}$ is $U^p$-automorphic.

\noindent (1) $\widehat{S}(U^{p}, \co_E)_{\overline{\rho}}$ is a $\widetilde{\bT}(U^p) \times G(F\otimes_{\Q} \Q_p)$-equivariant direct summand of $\widehat{S}(U^p, \co_E)$.

\noindent (2) The action of $\widetilde{\bT}(U^p)$ on $\widehat{S}(U^{p}, \co_E)_{\overline{\rho}}$ factors through $\widetilde{\bT}(U^p)_{\overline{\rho}}$.

\noindent (3) The $\co_E$-algebra $\widetilde{\bT}(U^p)_{\overline{\rho}}$ is reduced and acts faithfully on $\widehat{S}(U^{p}, \co_E)_{\overline{\rho}}$.
\end{lemma}
\noindent We assume the following hypothesis:
\begin{hypothesis}\label{hypoglobal}
We have $p>2$, $F/F^+$ is unramified and $G$ is quasi-split at all finite places of $F^+$.
\end{hypothesis}
\noindent Under the hypothesis, by \cite[Prop. 6.7]{Thor}, there exists a natural surjection of complete $\co_E$-algebras
\begin{equation*}
  R_{\overline{\rho}, \cS} \twoheadlongrightarrow \widetilde{\bT}(U^p)_{\overline{\rho}}.
\end{equation*}
In particular, $\widetilde{\bT}(U^p)_{\overline{\rho}}$ is a noetherian (local complete) $\co_E$-algebra.
\subsection{$P$-ordinary part}\label{sec: Pord-3}

\noindent We fix a parabolic subgroup $$P\cong \prod_{v\in S_p} \Res^{F_{\widetilde{v}}}_{\Q_p} P_{\widetilde{v}}$$ of $\prod_{v\in S_p} \Res^{F_{\widetilde{v}}}_{\Q_p} \GL_n$, such that $P_{\widetilde{v}}$ is a parabolic subgroup of $\GL_n$ containing the Borel subgroup $B$ of upper triangular matrices. Let $L_{\widetilde{v}}\cong \GL_{n_{\widetilde{v},1}} \times \cdots \times \GL_{n_{\widetilde{v},k_{\widetilde{v}}}}$ be the Levi subgroup of $P_{\widetilde{v}}$ containing the diagonal subgroup $T$, $\overline{P}_{\widetilde{v}}$ be the parabolic subgroup of $\GL_n$ opposite to $P_{\widetilde{v}}$, $N_{P_{\widetilde{v}}}$ (resp. $N_{\overline{P}_{\widetilde{v}}}$) be the unipotent radical of $P_{\widetilde{v}}$ (resp. $\overline{P}_{\widetilde{v}}$), and $Z_{L_{\widetilde{v}}}$ be the center of $L_{\widetilde{v}}$. For $i=1,\cdots, k_{\widetilde{v}}$, let $s_{\widetilde{v}, i}:=\sum_{j=0}^{i-1} n_{\widetilde{v},j}$, where $n_{\widetilde{v},0}:=0$. Put
\begin{equation*}L_P:=\prod_{v\in S_p} \Res^{F_{\widetilde{v}}}_{\Q_p} L_{\widetilde{v}}, \ \overline{P}:=\prod_{\sigma \in S_p} \Res^{F_{\widetilde{v}}}_{\Q_p} \overline{P}_{\widetilde{v}},\end{equation*}
 \begin{equation*}N_P:=\prod_{v\in S_p} \Res^{F_{\widetilde{v}}}_{\Q_p} N_{P_{\widetilde{v}}}, \ N_{\overline{P}}:=\prod_{v\in S_p} \Res^{F_{\widetilde{v}}}_{\Q_p} N_{\overline{P}_{\widetilde{v}}}, \ Z_{L_P}:=\prod_{v\in S_p} \Res^{F_{\widetilde{v}}}_{\Q_p} Z_{L_{\widetilde{v}}}.\end{equation*} Thus $L_P$ is the Levi subgroup of $P$ containing $\prod_{v\in S_p} \Res^{F_{\widetilde{v}}}_{\Q_p} T$, $\overline{P}$ is the parabolic subgroup opposite to $P$,  $Z_{L_P}$ is the center of $L_P$, and $N_P$ (resp. $N_{\overline{P}}$) is the unipotent radical of $P$ (resp. of $\overline{P}$).
\\

\noindent For $v\in S_p$, $i\in \Z_{\geq 0}$, let
\begin{equation*}K_{i,\widetilde{v}}:=\{g\in \GL_n(\co_{F_{\widetilde{v}}})\ |\ g\equiv 1 \pmod{\varpi_{\widetilde{v}}^i}\},\end{equation*}
 \begin{equation*}N_{i,\widetilde{v}}:=N_{P_{\widetilde{v}}}(F_{\widetilde{v}}) \cap K_{i,\widetilde{v}}, \ L_{i,\widetilde{v}}:=L_{\widetilde{v}}(F_{\widetilde{v}}) \cap K_{i,\widetilde{v}}, \ \overline{N}_{i,\widetilde{v}}:=N_{\overline{P}_{\widetilde{v}}}(F_{\widetilde{v}}) \cap K_{i,\widetilde{v}}.\end{equation*}
 For $i\geq j\geq 0$, put $K_{i,j,\widetilde{v}}:=\overline{N}_{i,\widetilde{v}} L_{j,\widetilde{v}} N_{0,\widetilde{v}}$. Let
 \begin{equation*}Z_{L_{\widetilde{v}}}^+:=\{(a_1, \cdots, a_{k_{\widetilde{v}}})\in Z_{L_{\widetilde{v}}}(F_{\widetilde{v}})\ |\ \val_p(a_1)\geq \cdots \geq \val_p(a_{k_{\widetilde{v}}})\}.\end{equation*}
 Finally, we put $K_i:=\prod_{v\in S_p} K_{i,\widetilde{v}}$, $N_i:=\prod_{v\in S_p} N_{i,\widetilde{v}}$, $L_i:=\prod_{v\in S_p} L_{i,\widetilde{v}}$, $\overline{N}_i:=\prod_{v\in S_p} \overline{N}_{i,\widetilde{v}}$, $K_{i,j}:=\prod_{v\in S_p} K_{i,j,\widetilde{v}}\cong \overline{N}_i L_j N_0$, and $Z_{L_P}^+:=\prod_{v\in S_p} Z_{L_{\widetilde{v}}}^+$. The proof of the following lemma is  straightforward (and we omit).
 \begin{lemma}
   (1) For $i\in \Z_{\geq 0}$, $K_i$ is a normal subgroup of $K_0$, and $\overline{N}_i \times L_i \times N_i\xrightarrow{\sim} K_i$.

 \noindent (2) For  $i \geq j \geq 0$, we have $\overline{N}_i\times L_j \times N_0 \xrightarrow{\sim} K_{i,j}$.

\noindent    (3) For  $z\in Z_{L_P}^+$, $i\in \Z_{\geq 0}$, we have $\overline{N}_i\subseteq z\overline{N}_i z^{-1}$.
 \end{lemma}
\noindent Applying Emerton's ordinary part functor (\cite[\S~3]{EOrd1}), we obtain a unitary admissible Banach representation of $L_P(\Q_p)$:
\begin{equation*}
  \Ord_P(\widehat{S}(U^p, E)_{\overline{\rho}})\cong \Ord_P(\widehat{S}(U^p, \co_E)_{\overline{\rho}}) \otimes_{\co_E} E.
\end{equation*}

\noindent Let $M$ be a finitely generated $\co_E$-module, equipped with an $\co_E$-linear action of $Z_{L_P}^+$. Consider the   $\co_E$-subalgebra $B$ of $\End_{\co_E}(M)$ generated by the image of $\iota: Z_{L_P}^+ \ra \End_{\co_E}(M)$. It is clear that $B$ is  a finite $\co_E$-algebra. Recall that a maximal ideal $\fn$ of $B$ is called \emph{ordinary} if $\fn \cap \iota(Z_{L_P}^+)=\emptyset$. Denote by $M_{\ord}:=\prod_{\fn\ \text{ordinary}} M_{\fn}$, which is a direct summand of $M$. The induced action of $Z_{L_P}^+$ on $M_{\ord}$ is invertible and hence extends naturally to an action of $Z_{L_P}$.
We have then
\begin{multline}\label{isoOrd}
  \Ord_P(\widehat{S}(U^p, \co_E)_{\overline{\rho}})\cong \varprojlim_k \Ord_P(S(U^p, \co_E/\varpi_E^k)_{\overline{\rho}}) \cong \varprojlim_k \varinjlim_{i} \Ord_P(S(U^p, \co_E/\varpi_E^k)_{\overline{\rho}})^{L_i}
  \\ \cong \varprojlim_k \varinjlim_{i} (S(U^p, \co_E/\varpi_E^k)^{K_{i,i}})_{\ord}
  \cong \varprojlim_k \varinjlim_i (S(U^pK_{i,i}, \co_E/\varpi_E^k)_{\overline{\rho}})_{\ord},
\end{multline}
where the first isomorphism is by definition, the second  from the fact that  $ \Ord_P(S(U^p, \co_E/\varpi_E^k)_{\overline{\rho}})$ is a smooth representation of $L_P(L)$ over $\co_E/\varpi_E^k$, the third isomorphism from step (c) in the proof of \cite[Thm. 4.4]{BD1}, and the last isomorphism from  $S(U^p, \co_E/\varpi_E^k)^{K_{i,i}}_{\overline{\rho}}\xrightarrow{\sim} S(U^pK_{i,i}, \co_E/\varpi_E^k)_{\overline{\rho}}$.
\\

\noindent
We define as in \cite[(4.15)]{BD1}:
\begin{equation*}
\Ord_P(S(U^p, \co_E)_{\overline{\rho}})\cong \varinjlim_{i} \big(S(U^p, \co_E)_{\overline{\rho}}^{K_{i,i}}\big)_{\ord},
\end{equation*}
which is equipped with a natural  smooth action of $L_P(\Q_p)$. As in the proof of \cite[Lem. 6.8 (1)]{BD1}, we have for $k\geq 1$:
\begin{equation}\label{Pord: dense1}
  \Ord_P(S(U^p, \co_E)_{\overline{\rho}})/\varpi_E^k \xlongrightarrow{\sim} \Ord_P(S(U^p, \co_E/\varpi_E^k)_{\overline{\rho}}).
\end{equation}
In particular, $\Ord_P(S(U^p, \co_E)_{\overline{\rho}}$ is dense in $ \Ord_P(\widehat{S}(U^p, \co_E)_{\overline{\rho}})$. Finally, by \cite[Cor. 4.6]{BD1}, we have
\begin{lemma}\label{LPzp}
  The  representation  $\Ord_P(\widehat S(U^p, \co_E)_{\overline{\rho}})|_{L_P(\Z_p)}$  is  isomorphic  to  a  direct summand of $\cC(L_P(\Z_p), \co_E)^{\oplus r}$ for some $r\geq 0$.
\end{lemma}

\subsection{$P$-ordinary Hecke algebra}
\noindent Assume $\Ord_P(\widehat S(U^p, \co_E)_{\overline{\rho}})\neq 0$. Note that by (\ref{Pord: dense1}), this implies $\Ord_P(S(U^p, \co_E)_{\overline{\rho}})\neq 0$. By the local-global compatibility for classical local Langlands correspondence (see \cite[Thm. 6.5 (v)]{Thor}, \cite{Car12}) and using \cite[Prop. 5.10]{BD1}, we can deduce that $\overline{\rho}_{\widetilde{v}}$ is $P_{\widetilde{v}}$-ordinary, for $v\in S_p$.
\\

\noindent
For any $i\geq 0$, $*\in \{\co_E, \co_E/\varpi_E^s\}$, $(S(U^p, *)_{\overline{\rho}}^{K_{i,i}})_{\ord}=(S(U^p K_{i,i}, *)_{\overline{\rho}})_{\ord}$ is stable by $\bT(U^p)$ (since the action of $\bT(U^p)$ on $S(U^{p}, *)^{K_{i,i}}_{\overline{\rho}}$ commutes with that of $L_P^+$), and we denote by $\bT(U^p K_{i,i}, *)_{\overline{\rho}}^{P-\ord}$ the $\co_E$-subalgebra of the endomorphism ring of $(S(U^p, *)_{\overline{\rho}}^{K_{i,i}})_{\ord}$ generated by the operators in $\bT(U^p)$. From the natural $\bT(U^p)$-equivariant injection
\begin{equation*}
\big(S(U^p,*)_{\overline{\rho}}^{K_{i,i}}\big)_{\ord} \hooklongrightarrow S(U^p, *)_{\overline{\rho}}^{K_{i,i}}\cong S(U^{p}K_{i,i}, *)_{\overline{\rho}},
\end{equation*}
we have a natural surjection of local $\co_E$-algebras (finite  over $\co_E$): $\bT(U^{p}K_{i,i}, *)_{\overline{\rho}} \twoheadrightarrow \bT(U^{p}K_{i,i}, *)_{\overline{\rho}}^{P-\ord}.$
We have $\bT(U^pK_{i,i}, \co_E)_{\overline{\rho}}^{P-\ord}\cong \varprojlim_s \bT(U^pK_{i,i}, \co_E/\varpi_E^s)_{\overline{\rho}}^{P-\ord}$.
We set:
$$\widetilde{\bT}(U^p)^{P-\ord}_{\overline{\rho}}:=\varprojlim_i \bT(U^p K_{i,i}, \co_E)_{\overline{\rho}}^{P-\ord}$$
which is thus easily checked to be a quotient of $\widetilde{\bT}(U^p)_{\overline{\rho}}$ and is also a complete local $\co_E$-algebra of residue field $k_E$. We have as in \cite[Lem. 6.7, 6.8 (1)]{BD1}:
\begin{lemma}\label{lem: Pord-reduce}
The $\co_E$-algebra $\widetilde{\bT}(U^{p})_{\overline{\rho}}^{P-\ord}$ is reduced and the natural action of $\widetilde{\bT}(U^p)_{\overline{\rho}}^{P-\ord}$ on $\Ord_P(S(U^p, E)_{\overline{\rho}})$ is faithful, and extends to a faithful action on $\Ord_P(\widehat{S}(U^p, \co_E)_{\overline{\rho}})$.
\end{lemma}

\noindent Let $\fm$ be a maximal ideal of $\widetilde{\bT}(U^p)^{P-\ord}_{\overline{\rho}}[1/p]$ such that $\Ord_P(\widehat{S}(U^p,E)_{\overline{\rho}})[\fm]\neq 0$. Using the natural composition
\begin{equation*}
  R_{\overline{\rho},\cS} \twoheadlongrightarrow \widetilde{\bT}(U^p)_{\overline{\rho}} \twoheadlongrightarrow \widetilde{\bT}(U^p)_{\overline{\rho}}^{P-\ord},
\end{equation*}
we associate to $\fm$ a continuous representation $\rho_{\fm}: \Gal_{F^+} \ra \cG_n(E)$. Let $\rho_{\fm,F}: \Gal_F\xrightarrow{\rho_{\fm}} \GL_n(E)\times \GL_1(E)\ra \GL_n(E)$, and $\rho_{\fm,\widetilde{v}}:=\rho_{\fm,F}|_{\Gal_{F_{\widetilde{v}}}}$.
\begin{conjecture}\label{conjPord}
$\rho_{\fm,\widetilde{v}}$ is $P_{\widetilde{v}}$-ordinary for all $v|p$.
\end{conjecture}
\begin{proposition}
  Suppose that for all $v|p$, any $P_{\widetilde{v}}$-filtration on $\overline{\rho}_{\widetilde{v}}$ satisfies Hypothesis \ref{hypo: gene1}, then Conjecture \ref{conjPord} holds.
\end{proposition}
\begin{proof}
  For $v|p$, and a $P_{\widetilde{v}}$-filtration $\cF_{\widetilde{v}}$ on $\overline{\rho}_{\widetilde{v}}$, by Hypothesis \ref{hypo: gene1}, (the proof of) Proposition \ref{propPord}, the reduced quotient $R_{\overline{\rho}_{\widetilde{v}}, \cF_{\widetilde{v}}}^{P-\ord, \bar{\square}}$ of $R_{\overline{\rho}_{\widetilde{v}}, \cF_{\widetilde{v}}}^{P_{\widetilde{v}}-\ord, \square}$ is a local deformation problem (cf. \cite[Def. 3.2]{Thor}). Denote by $I_{\cF_{\widetilde{v}}}$ the kernel of $R_{\overline{\rho}}^{\square} \twoheadrightarrow R_{\overline{\rho}_{\widetilde{v}}, \cF_{\widetilde{v}}}^{P_{\widetilde{v}}-\ord, \bar{\square}}$, and put $I_{\widetilde{v}}:=\cap_{\cF_{\widetilde{v}}} I_{\cF_{\widetilde{v}}}$ where $\cF_{\widetilde{v}}$ runs through all the (finitely many) $P_{\widetilde{v}}$-filtrations on $\overline{\rho}_{\widetilde{v}}$. Let $R_{\overline{\rho}_{\widetilde{v}}}^{P_{\widetilde{v}}-\ord, \bar{\square}}:=R_{\overline{\rho}}^{\square}/I_{\widetilde{v}}$, which, by \cite[Lem. 3.2]{BLGHT}, is a local deformation problem at the place $\widetilde{v}$. Let $R_{\overline{\rho}, \cS}^{P-\ord}$ be the universal deformation ring of the deformation problem
  \begin{equation*}
  (F/F^+, S, \widetilde{S}, \co_E, \overline{\rho}, \varepsilon^{1-n}\delta_{F/F^+}, \{R_{\overline{\rho}_{\widetilde{v}}}^{\bar{\square}}\}_{v\in S\setminus S_p}, \{R_{\overline{\rho}_{\widetilde{v}}}^{P_{\widetilde{v}}-\ord, \bar{\square}}\}_{v\in S_p}).
\end{equation*}
  Denote by $I$ the kernel of the natural surjection $R_{\overline{\rho}, \cS} \twoheadrightarrow R_{\overline{\rho}, \cS}^{P-\ord}$. By the same argument as in the proof of \cite[Thm. 6.12 (1)]{BD1} (which relies on \cite[Prop. 5.10]{BD1}, noting also that since $\overline{\rho}$ is absolutely irreducible, for any continuous $\Gal_F$-representation $\rho$ with modulo $\varpi_E$ reduction isomorphic to $\overline{\rho}$,  any $P_{\widetilde{v}}$-filtration on $\rho_{\widetilde{v}}:=\rho|_{\Gal_{F_{\widetilde{v}}}}$ naturally induces a $P_{\widetilde{v}}$-filtration on $\overline{\rho}_{\widetilde{v}}$), we have $I\big(\Ord_P(\widehat{S}(U^p,E)_{\overline{\rho}})\big)=0$. The proposition then follows by the same argument as for \cite[Thm. 6.12 (2)]{BD1}.
\end{proof}

%\begin{remark}
 % By the same argument as in the proof of \cite[Thm. 6.12]{BD1}, one can prove the conjecture when $\overline{\rho}_{\widetilde{v}}$ is strictly $P_{\widetilde{v}}$-ordinary in the sense of \cite[Def. 5.8]{BD1}. One obstruction when generalizing the argument to the general case is that $\overline{\rho}_{\widetilde{v}}$ might admit infinitely many $P_{\widetilde{v}}$-filtrations $\cF_j$ such that for $j\neq j'$, there exists $i$ such that $\gr^i \cF_j\ncong \gr^i \cF_{j'}$.
%\end{remark}
%We have as in \cite[Lem. 6.11]{BD1}:
%\begin{lemma}
%The representation $\Ord_P(\widehat{S}(U^{p}, \co_E)_{\overline{\rho}})$ is a $\varpi_E$-adically admissible representation of $L_P(\Q_p)$ over $\widetilde{\bT}(U^p)_{\overline{\rho}}^{P-\ord}$ in the sense of \cite[Def.~3.1.1]{Em4}.
%\end{lemma}

\section{$\GL_2(\Q_p)$-ordinary families}
\noindent Keep the notation and assumptions in \S~\ref{ordAut}, in particular, we assume $\Ord_P(\widehat{S}(U^{p}, \co_E)_{\overline{\rho}})\neq 0$ and hence $\overline{\rho}_{\widetilde{v}}$ is $P_{\widetilde{v}}$-ordinary for all $v|p$. Assume moreover for all $v\in S_p$, $F_v^+\cong \Q_p$ and $n_{\widetilde{v},i}\leq 2$ for $i=1, \cdots, k_{\widetilde{v}}$. In this section, using $p$-adic Langlands correspondence for $\GL_2(\Q_p)$, we construct $\GL_2(\Q_p)$-ordinary families from $\Ord_P(\widehat{S}(U^p,E)_{\overline{\rho}})$. We also show a local-global compatibility result of these families, formulated in a similar way as in \cite{Pan2} (in particular, by using the theory of Pa{\v{s}}k{\=u}nas \cite{Pas13}). Under more restrictive assumptions, similar results were essentially obtained in \cite{BD1}, but were stated in a different formulation, due to Emerton \cite{Em4} (using deformations).
\subsection{Benign points}
\noindent As in \cite[\S~7.1.1]{BD1}, we put
\begin{multline}\label{LpZp+}
\Ord_P(\widehat{S}(U^p,E)_{\overline{\rho}})_+^{L_P(\Z_p)-\alg}:=\bigoplus_{\sigma} \Hom_{L_P(\Z_p)}(\sigma, \Ord_P(\widehat{S}(U^p,E)_{\overline{\rho}}))\otimes_E \sigma\\
\cong \bigoplus_{\sigma} (\Ord_P(\widehat{S}(U^p,E)_{\overline{\rho}})\otimes_E \sigma^\vee)^{L_P(\Z_p)} \otimes_E \sigma
\end{multline}
where $\sigma$ runs through algebraic representations of $L_P(\Q_p)\cong \prod_{v\in S_p} L_{P_{\widetilde{v}}}(\Q_p)$ of highest weight $(\lambda_{\widetilde{v},1}, \cdots, \lambda_{\widetilde{v},n})_{v\in S_p}$ satisfying $\lambda_{\widetilde{v},i}\geq \lambda_{\widetilde{v},i+1}$ for all $i\in\{1,\cdots, n-1\}$, $v\in S_p$.  We have as in \cite[Prop. 7.2]{BD1}:
\begin{proposition}\label{densalg}
  $\Ord_P(\widehat{S}(U^p,E)_{\overline{\rho}})_+^{L_P(\Z_p)-\alg}$ is dense in $\Ord_P(\widehat{S}(U^p,E)_{\overline{\rho}})$.
\end{proposition}
\begin{definition}
(1) A closed point $x\in \Spec \widetilde{\bT}(U^p)_{\overline{\rho}}^{P-\ord}[1/p]$ is benign if:
\begin{equation*}
\Ord_P\big(\widehat{S}(U^p, E\big)_{\overline{\rho}}[\fm_x]\big)^{L_P(\Z_p)-\alg}_+\neq 0.
\end{equation*}

\noindent (2) A closed point $x \in \Spec \widetilde{\bT}(U^p)_{\overline{\rho}}[1/p]$ is classical if $\widehat{S}(U^p,E)_{\overline{\rho}}[\fm_x]^{\lalg}\neq 0$.
\end{definition}
\noindent By the same argument as in the proof of \cite[Prop. 7.5]{BD1}, we have
\begin{proposition}\label{beni1}
(1) A benign point is classical.

\noindent (2) The benign points are Zariski-sense in $\Spec \widetilde{\bT}(U^p)_{\overline{\rho}}^{P-\ord}[1/p]$.
\end{proposition}
\noindent Let $x$ be a benign point, we can attach to $x$ a dominant weight (i.e. $\lambda_{\widetilde{v},1}\geq \cdots \geq \lambda_{\widetilde{v},n}$)
\begin{equation*}\ul{\lambda}=\prod_{v\in S_p} \ul{\lambda}_{\widetilde{v}}=\prod_{v\in S_p}(\lambda_{\widetilde{v},1}, \cdots, \lambda_{\widetilde{v}, n})\end{equation*}
 such that $\rho_{x,\widetilde{v}}$ is de Rham of Hodge-Tate weights $(\lambda_{\widetilde{v},1}, \lambda_{\widetilde{v},2}-1, \cdots, \lambda_{\widetilde{v}, n}-n+1)$ for $v\in S_p$ (e.g. by Proposition \ref{beni1} and \cite[Thm. 6.5(v)]{Thor}). We have
\begin{proposition}\label{beni2}Let $x$ be a benign point.

\noindent (1) $\rho_{x,\widetilde{v}}$ is semi-stable for all $v\in S_p$.

\noindent (2) $\rho_{x,\widetilde{v}}$ is $P_{\widetilde{v}}$-ordinary with a $P_{\widetilde{v}}$-filtration $\cF_{x,\widetilde{v}}$ satisfying that $\rho_{x, \widetilde{v},i}:=\gr^i \cF_{x,\widetilde{v}}$ (of dimension $n_{\widetilde{v},i}$) is crystalline of Hodge-Tate weights $(\lambda_{\widetilde{v},s_{\widetilde{v},i}+1}-s_{\widetilde{v},i}, \lambda_{\widetilde{v},s_{\widetilde{v},i}+n_{\widetilde{v},i}}-(s_{\widetilde{v},i}+n_{\widetilde{v},i}-1))$. Moreover, let $\alpha_{\widetilde{v},s_{\widetilde{v},i}+1}$, $\alpha_{\widetilde{v},s_{\widetilde{v},i}+n_{\widetilde{v},i}}$ be the eigenvalues of $\varphi$ on $D_{\cris}(\rho_{x,\widetilde{v},i})$, then $\alpha_{\widetilde{v},s_{\widetilde{v},i}+1}\alpha_{\widetilde{v},s_{\widetilde{v},i}+n_{\widetilde{v},i}}^{-1}\neq p^{\pm 1}$.

\noindent (3) There exists an $L_P(\Q_p)$-equivariant injection
  \begin{equation}\label{Pord: emb0}
 \otimes_{v\in S_p}\big(\otimes_{i=1,\cdots, k_{\widetilde{v}}} \widehat{\pi}(\rho_{x, \widetilde{v}, i})^{\lalg}\otimes_{k(x)} \varepsilon^{s_{\widetilde{v},i+1}-1}\circ \dett\big) \hooklongrightarrow \Ord_P\big(\widehat{S}(U^p, E\big)_{\overline{\rho}}[\fm_x]\big),
  \end{equation}
where $\widehat{\pi}(\rho_{x,\widetilde{v},i})$ denotes the continuous finite length representation of $\GL_{n_{\widetilde{v},i}}(\Q_p)$ over $k(x)$ (the residue field at $x$) associated to $\rho_{x,\widetilde{v},i}$ via the $p$-adic local Langlands correspondence for $\GL_2(\Q_p)$ (\cite{Colm10a}) normalized as in \emph{loc. cit.} and \cite{Pas13} \footnote{note that the normalization is slightly different from that in \cite{BD1}.} (so that the central character of $\widehat{\pi}(\rho_{x,\widetilde{v},i})$ is equal to $(\wedge^2 \rho_{x,\widetilde{v},i})\varepsilon^{-1}$) when $n_{\widetilde{v},i}=2$, via local class filed theory normalized by sending $p$ to a (lift of) geometric Frobenius when $n_{\widetilde{v},i}=1$.
\end{proposition}
\begin{proof}
The proposition follows by verbatim of the proof of \cite[Prop. 7.6, Cor. 7.10]{BD1}. Note that the strict $P$-ordinary assumption of \emph{loc. cit.} is only used to compare $\rho_{x,\widetilde{v},i}$ with a representation obtained by another way (which we don't use here).
\end{proof}
%\begin{remark}
% Note that the $P_{\widetilde{v}}$-ordinary filtration of $\rho_{x,\widetilde{v}}$ with decreasing Hodge-Tate weights is unique.
 %(2) By some more arguments, one can show that (\ref{Pord: emb0}) extends to an $L_P(\Q_p)$-equivariant injection
 % \begin{equation*}
 %   \widehat{\otimes}_{v\in S_p} \big(\widehat{\otimes}_{i=1,\cdots, k_{\widetilde{v}}} \widehat{\pi}(\rho_{x,\widetilde{v},i}) \otimes_E \varepsilon^{s_{v,i}} \circ \dett\big) \hooklongrightarrow \Ord_P\big(\widehat{S}(U^p, E\big)_{\overline{\rho}}[\fm_x]\big).
 % \end{equation*}
%\end{remark}
\begin{proposition}\label{beni4}
(1) If $x$ is benign, then there exists $r(x)\geq 1$ such that
 \begin{equation*}
  \Big( \otimes_{v\in S_p}\big(\otimes_{i=1,\cdots, k_{\widetilde{v}}} \widehat{\pi}(\rho_{x, \widetilde{v}, i})^{\lalg}\otimes_{k(x)} \varepsilon^{s_{\widetilde{v},i+1}-1}\circ \dett\big)\Big)^{\oplus r(x)} \xlongrightarrow{\sim} \Ord_P\big(\widehat{S}(U^p, E\big)_{\overline{\rho}}[\fm_x]^{\lalg}\big),
 \end{equation*}
where we refer to \cite[\S~4.3]{BD1} for the definition of $\Ord_P\big(\widehat{S}(U^p, E\big)_{\overline{\rho}}[\fm_x]^{\lalg}\big)$.

\noindent (2) The action of $\widetilde{\bT}(U^p)_{\overline{\rho}}[1/p]$ on $\Ord_P(\widehat{S}(U^p, E)_{\overline{\rho}})^{L_P(\Z_p)-\alg}_+$ is semi-simple.
\end{proposition}
\begin{proof}
  (1) follows from the same argument as in \cite[Lem. 7.8, Prop. 7.9, Cor. 7.10]{BD1}. (2) follows from the proof of \cite[Prop. 7.5]{BD1} (which proves that all the vectors in $\Ord_P(\widehat{S}(U^p, E)_{\overline{\rho}})^{L_P(\Z_p)-\alg}_+$ come from locally algebraic vectors in $\widehat{S}(U^p,E)_{\overline{\rho}}$ via the adjunction property \cite[Prop. 4.21]{BD1}, on which the action of $\widetilde{\bT}(U^p)_{\overline{\rho}}[1/p]$ is semi-simple by (\ref{equ: lalgaut})).
\end{proof}

\subsection{Pa{\v{s}}k{\=u}nas' theory}\label{sec: Pas}
\noindent We recall Pa{\v{s}}k{\=u}nas' theory of blocks (\cite{Pas13}), which we will use to construct our $\GL_2(\Q_p)$-ordinary families.
\\

\noindent Let $H$ be a $p$-adic analytic group. Denote by $\Mod_{H}^{\sm}(\co_E)$ the category of smooth representations of $H$ over $\co_E$ in the sense of \cite[Def. 2.2.1]{EOrd1}, and $\Mod_{H}^{\lfin}(\co_E)$ the full subcategoy of $\Mod_{H}^{\sm}(\co_E)$ consisting of those objects which are locally of finite length. For an irreducible representation $\pi\in \Mod_{H}^{\lfin}(\co_E)$, denote by $\cJ_{\pi}$ the injective enveloppe of $\pi$ in $\Mod_{H}^{\lfin}(\co_E)$. A block $\fB$ of $\Mod_{H}^{\lfin}(\co_E)$ is a set of irreducible representations, such that if $\tau\in \fB$, then $\tau'\in \fB$ if and only if there exists a sequence of irreducible representaitons $\tau=\tau_0, \tau_1, \cdots, \tau_m=\tau'$ such that $\tau_i=\tau_{i+1}$, $\Ext^1_{H}(\tau_i, \tau_{i+1})\neq 0$ or $\Ext^1_{H}(\tau_{i+1}, \tau_i)=0$.
We have a decomposition
\begin{equation*}\Mod_{H}^{\lfin}(\co_E)\cong \prod_{\fB} \Mod_{H}^{\lfin}(\co_E)^{\fB},
 \end{equation*}
where $\fB$ runs through the blocks of $\Mod_{H}^{\lfin}(\co_E)$, and $\Mod_{H}^{\lfin}(\co_E)^{\fB}$ denotes the full subcategory of $\Mod_{H}^{\lfin}(\co_E)$ consisting of those objects such that all the irreducible subquotients lie in $\fB$. In particular, for any $\tau\in \Mod_{H}^{\lfin}(\co_E)$, $\tau\cong \oplus_{\fB} \tau_{\fB}$ with $\tau_{\fB}\in \Mod_{H}^{\lfin}(\co_E)$. If $\tau$ is moreover admissible, then there exists a finite set $\cI$ of blocks such that
\begin{equation*}\tau\cong \oplus_{\fB\in \cI} \tau_{\fB}.\end{equation*}
For a block $\fB$, denote by $\pi_{\fB}:=\oplus_{\pi \in \fB} \pi$. Denote by $\cJ_{\fB}\cong \oplus_{\pi\in \fB} \cJ_{\pi}$ the injective envelope of $\pi_{\fB}$ in $\Mod_{H}^{\lfin}(\co_E)$, and $\widetilde{E}_{\fB}:=\End_{H}(\cJ_{\fB})$.
\\

\noindent By \cite[(2.2.8)]{EOrd1}, taking Pontryagain dual induces an anti-equivalence of categories between the category $\Mod_H^{\sm}(\co_E)$ and the category $\Mod_H^{\pro \ \aug}(\co_E)$ of profinite augmented $H$-representations over $\co_E$ (cf. \cite[Def. 2.1.6]{EOrd1}). Denote by $\fC_H(\co_E)$ the full subcategory of $\Mod_H^{\pro \ \aug}(\co_E)$ consisting of those objects that are the Pontryagain duals of the representations in $\Mod_H^{\lfin}(\co_E)$. For a block $\fB$ of $\Mod_H^{\lfin}(\co_E)$, we see
\begin{equation*}\widetilde{P}_{\fB}:=\cJ_{\fB}^{\vee} \cong \oplus_{\pi\in \fB} \cJ_{\pi}^{\vee}\cong \oplus_{\pi\in \fB} \widetilde{P}_{\pi^{\vee}}
 \end{equation*}is a projective envelope of $\pi_{\fB}^{\vee}$ in $\fC_H(\co_E)$, where $\widetilde{P}_{\pi^{\vee}}$ denotes the projective enveloppe of $\pi^{\vee}$ in $\fC_H(\co_E)$. And we have $\End_{\fC_H(\co_E)}(\widetilde{P}_{\fB}) \cong \widetilde{E}_{\fB}$. Denote by $\fC_H(\co_E)^{\fB}$ the full subcategory of $\fC_H(\co_E)$ consisting of those objects whose Pontryagain dual lies in $\Mod_H^{\lfin}(\co_E)^{\fB}$. The functor sending $M\in \fC_H(\co_E)^{\fB}$ to $\Hom_{\fC_H(\co_E)}(\widetilde{P}_{\fB}, M)$ induces an anti-equivalence of categories between $\fC_H(\co_E)^{\fB}$ and the category of pseudo-compact $\widetilde{E}_{\fB}$-modules, with the inverse given by $M\mapsto M\widehat{\otimes}_{\widetilde{E}_{\fB}} \widetilde{P}_{\fB}$ \big(cf. \cite[Lem. 2.9, 2.10]{Pas13}, note that a similar argument as in the proof of \cite[Lem. 2.10]{Pas13} also shows that $\Hom_{\fC_H(\co_E)}(\widetilde{P}_{\fB}, M) \widehat{\otimes}_{\widetilde{E}_{\fB}} \widetilde{P}_{\fB} \xrightarrow{\sim} M$ for $M\in \fC_H(\co_E)^{\fB}$\big).
\\

\noindent By \cite[\S~3.2]{Pas13}, the blocks of $\Mod_{\Q_p^{\times}}^{\lfin}(\co_E)$ that contain an absolutely irreducible representation  are given by $\fB=\{\chi:\Q_p^{\times} \ra k_E^{\times}\}$. For such $\fB$, we have $\widetilde{E}_{\fB}\cong \co_E[[x,y]]$, and that $\widetilde{P}_{\fB}$ is a free $\widetilde{E}_{\fB}$-module of rank $1$. Actually, let $\Def_{\chi}: \Art(\co_E) \ra \{\text{Sets}\}$ denote the standard deformation functor of $\chi$, then $\Def_{\chi}$ is pro-represented by $\widetilde{E}_{\fB}$ and $\widetilde{P}_{\fB}$ is isomorphic to the universal deformation of $\chi$ over $\widetilde{E}_{\fB}$. We denote by $1_{\univ}$ the universal deformation of the trivial character over $\Lambda:=\co_E[[x,y]]$. By the local class field theory, we have an isomorphism between $\Def_1$ and the deformation functor of the trivial character of $\Gal_{\Q_p}$, which we also denote by $\Def_1$. The $\Q_p^{\times}$ action on $1_{\univ}$ naturally extends to a $\Gal_{\Q_p}$-action, and the resulting $\Gal_{\Q_p}$-representation over $\Lambda$ is actually the universal deformation of $1$.
\\

\noindent By \cite[Cor. 1.2]{Pas14}, the blocks of $\Mod_{\GL_2(\Q_p)}^{\lfin}(\co_E)$ that contain an absolutely irreducible representation are given by (when $p>2$)%(noting that such blocks are the same as the corresponding blocks of the full subcategory of $\Mod_{\GL_2(\Q_p)}^{\lfin}(\co_E)$ consisting of those representations with fixed central character, considered in \cite{Pas13}). \footnote{Indeed, for absolutely irreducible representations $\tau$, $\tau'$ of $\GL_2(\Q_p)$ over $k_E$, and suppose $\tau\neq \tau'$, then any non-zero element (if it exists) in $\Ext^1_{\GL_2(\Q_p)}(\tau,\tau')$ has a central character.}
\begin{itemize}
  \item[(1)] $\fB=\{\pi\}$, supersingular,
  \item[(2)] $\fB=\{\Ind_{\overline{B}(\Q_p)}^{\GL_2(\Q_p)} (\chi_1\omega^{-1}\otimes \chi_2 ), \Ind_{\overline{B}(\Q_p)}^{\GL_2(\Q_p)} (\chi_2\omega^{-1}\otimes \chi_1)\}$, $\chi_1\chi_2^{-1}\neq 1$, $\omega^{\pm 1}$,
  \item[(3)] $\fB=\{\Ind_{\overline{B}(\Q_p)}^{\GL_2(\Q_p)} (\chi\omega^{-1} \otimes \chi )\}$,
  \item[(4)] $\fB=\{\eta \circ \dett, \Sp \otimes_{k_E} \eta \circ \dett, (\Ind_{\overline{B}(\Q_p)}^{\GL_2(\Q_p)} \omega^{-1}\otimes \omega)\otimes_{k_E} \eta \circ \dett\}$ if $p\geq 5$,
  \item[(4')]     $\fB=\{\eta \circ \dett, \Sp \otimes_{k_E} \eta \circ \dett, (\eta\omega) \circ \dett, \Sp \otimes_{k_E} (\eta\omega)\circ \dett\}$ if $p=3$.
\end{itemize}
For each $\fB$ as above, we can attach a $2$-dimensional semi-simple representation $\overline{\rho}_{\fB}$ of $\Gal_{\Q_p}$ over $k_E$ such that
\begin{itemize}
  \item if $\pi\in \fB$ supersingular, then $\overline{\rho}_{\fB}=\hV(\pi)$, where $\hV$ is the Colmez's functor normalized as in \cite[\S~5.7]{Pas13};
  \item if $\Ind_{\overline{B}(\Q_p)}^{\GL_2(\Q_p)} (\chi_1 \omega^{-1}\otimes \chi_2)\in \fB$ with $\chi_1\chi_2^{-1}\neq \omega^{\pm 1}$, then $\overline{\rho}_{\fB}=\chi_1\oplus \chi_2 $,
  \item if $\eta \circ \dett\in \fB$, then $\overline{\rho}_{\fB}=\eta \oplus \eta\omega$.
\end{itemize}
Note that under this normalization, if $\pi\in \fB$ has central character $\zeta: \Q_p^{\times} \ra k_E^{\times}$, then $\wedge^2 \overline{\rho}_{\fB}=\zeta \omega$.
Suppose $p\geq 3$.
Denote by $R_{\fB}^{\pss}$ the universal deformation ring which parametrizes all $2$-dimensional pseudo-representations of $\Gal_{\Q_p}$ lifting $\tr \overline{\rho}_{\fB}$ (cf. \cite[Lem. 1.4.2]{Kis08a}).
\begin{theorem}[Pa{\v{s}}k{\=u}nas]\label{blo0}Suppose the block $\fB$ of $\Mod_{\GL_2(\Q_p)}^{\lfin}(\co_E)$ lies in case (1) (2) (3) (4).

\noindent (1) There exists a natural isomorphism between the centre of $\widetilde{E}_{\fB}$ and $R_{\fB}^{\pss}$.

\noindent (2) $\widetilde{E}_{\fB}$ is a finitely generated module over $R_{\fB}^{\pss}$.
\end{theorem}
\begin{proof}Let $\zeta: \Q_p^{\times} \ra\co_E^{\times}$ be a continuous character, let $\Mod_{\GL_2(\Q_p),\zeta}^{\lfin}(\co_E)$ be the full subcategory of $\Mod_{\GL_2(\Q_p)}^{\lfin}(\co_E)$ consisting of those representations that have central character $\zeta$, and $\fC_{\GL_2(\Q_p),\zeta}(\co_E)$ be the full subcategory of $\fC_{\GL_2(\Q_p)}(\co_E)$ consisting of the objects whose dual lies in $\Mod_{\GL_2(\Q_p),\zeta}^{\lfin}(\co_E)$.  We denote by $\widetilde{P}_{\fB,\zeta}$ the projective enveloppe of $\pi_{\fB}^{\vee}$ in $\fC_{\GL_2(\Q_p),\zeta}(\co_E)$ (recall $\pi_{\fB}=\oplus_{\pi\in \fB} \pi$), and put $\widetilde{E}_{\fB,\zeta}:=\End_{\fC_{\GL_2(\Q_p),\zeta}(\co_E)}(\widetilde{P}_{\fB,\zeta})$. For $\pi\in \fB$, denote by $\widetilde{P}_{\pi^{\vee}, \zeta}$ the projective enveloppe of $\pi^{\vee}$ in $\fC_{\GL_2(\Q_p),\zeta}(\co_E)$, thus $\widetilde{P}_{\fB, \zeta}\cong \oplus_{\pi\in \fB} \widetilde{P}_{\pi^{\vee},\zeta}$.
\\

\noindent  Suppose $\fB$ is in case (1) (2) (4):

\noindent Let $\pi$ be an arbitrary representation in  $\fB$ if $\fB$ is in the case (1) or (2), and let $\pi:=(\Ind_{\overline{B}(\Q_p)}^{\GL_2(\Q_p)} \omega^{-1}\otimes \omega)\otimes_{k_E} \eta \circ \dett\in \fB$ if $\fB$ is in the case (4). By \cite[Prop. 6.18]{CEGGPS2}, we have $\widetilde{P}_{\pi^{\vee}}\cong \widetilde{P}_{\pi^{\vee}, \zeta}\widehat{\otimes}_{\co_E} 1_{\univ}\circ \dett$. Denote by $\overline{\rho}_{\pi}$  the (unique) two dimensional representation of $\Gal_{\Q_p}$ over $k_E$ satisfying  that $\overline{\rho}_{\pi}^{\sss} \cong \overline{\rho}_{\fB}$    and  that if $\fB$ is moreover in the case (2) or (4), then $\hV(\pi)^{-1} \omega \overline{\zeta} \cong \cosoc_{\Gal_{\Q_p}}\overline{\rho}_{\pi}$.
By \cite[Cor. 6.23]{CEGGPS2}, we have  a natural isomorphism (see also Remark \ref{normB})
\begin{equation}\label{equ: Colm}\End_{\fC_{\GL_2(\Q_p)}(\co_E)}(\widetilde{P}_{\pi^{\vee}} )\cong R_{\overline{\rho}_{\pi}},
 \end{equation}where $R_{\overline{\rho}_{\pi}}$
 denotes the universal deformation ring of $\overline{\rho}_{\pi}$ (note that by the assumption on $\pi$, $\End_{\Gal_{\Q_p}}(\overline{\rho}_{\pi})\cong k_E$ hence $R_{\overline{\rho}_{\pi}}$ exists).
\begin{itemize}
\item If $\fB$ is in case (1), then $\widetilde{P}_{\fB}\cong \widetilde{P}_{\pi^{\vee}}$, and $\widetilde{E}_{\fB}\cong R_{\overline{\rho}_{\pi}}\cong R_{\overline{\rho}_{\fB}}^{\pss}$.
\item  If $\fB$ is in case (2), we write $\fB=\{\pi_1,\pi_2\}$. The statement follows by the same argument as in \cite[Cor. 8.11]{Pas13} replacing the isomorphism in \cite[Cor. 8.7]{Pas13} by (\ref{equ: Colm}) applied to $\pi_1$ and $\pi_2$.
\item If $\fB$ is in case (4). The statement follows by the same arguments as for \cite[Thm. 10.87]{Pas13} replacing the isomorphism in \cite[Thm. 10.71]{Pas13} by (\ref{equ: Colm}).
\end{itemize}

\noindent Suppose $\fB$ is in case (3). The statement in this case follows by a similar argument. We include a proof (with several steps) for the convenience of the reader.
\\

\noindent (a) Let $\pi\in \fB$. We first show $\widetilde{P}_{\pi^{\vee}}\cong \widetilde{P}_{\pi^{\vee}, \zeta}\widehat{\otimes}_{\co_E} 1_{\univ}\circ \dett$.  Put $\widetilde{P}':=\widetilde{P}_{\pi^{\vee},\zeta} \otimes_{\co_E} 1_{\univ}$, equipped with the diagonal action of $\GL_2(\Q_p)$, where $\GL_2(\Q_p)$ acts on the second factor via $\dett: \GL_2(\Q_p) \ra \Q_p^{\times}$. It is not difficult to see $\widetilde{P}'\in \fC_{\GL_2(\Q_p)}(\co_E)$. Indeed, we can write $\widetilde{P}_{\pi^{\vee},\zeta}\cong \varprojlim_n \widetilde{P}_{\pi^{\vee},\zeta,n}$(resp. $\Lambda\cong \varinjlim_n \Lambda_n$) such that the Pontryagain dual of each  $\widetilde{P}_{\pi^{\vee},\zeta,n}$ (resp. $\Lambda_n$) is a finite length representation of $\GL_2(\Q_p)$ (resp. of $\Q_p^{\times}$), and hence $\widetilde{P}'\cong \varprojlim_n (\widetilde{P}_{\pi^{\vee},\zeta,n} \otimes_{\co_E} \Lambda_n)\in \fC_{\GL_2(\Q_p)}(\co_E)$. By \cite[Thm. 3.26]{Pas13}, $\widetilde{P}_{\pi^{\vee}, \zeta}$ is a deformation of $\pi^{\vee}$ over $\widetilde{E}_{\fB,\zeta}$. We see by definition that $\widetilde{P}'$ is a deformation of $\pi^{\vee}$ over $\widetilde{E}':=\widetilde{E}_{\fB,\zeta}\widehat{\otimes}_{\co_E} \Lambda$.

\noindent We show $\cosoc_{\GL_2(\Q_p)} \widetilde{P}'\cong \pi^{\vee}$. By the proof of \cite[Lem. B.8]{GN16}, we know
\begin{equation*}\Hom_{\fC_{\GL_2(\Q_p), \zeta}(\co_E)\times \fC_{\Q_p^{\times}}(\co_E)}(\widetilde{P}_{\pi^{\vee}, \zeta} \widehat{\otimes}_{\co_E} 1_{\univ}, \pi^{\vee}\otimes_{k_E} k_E)\cong k_E.
\end{equation*}
We deduce then $\Hom_{\fC_{\GL_2(\Q_p)}(\co_E)}(\widetilde{P}', \pi^{\vee})\hookrightarrow k_E$. Since any irreducible constituent of $\widetilde{P}'$ is isomorphic to $\pi^{\vee}$, we deduce then  $\cosoc_{\GL_2(\Q_p)} \widetilde{P}'\cong \pi^{\vee}$.

\noindent We have thus a projection $\widetilde{P}_{\pi^{\vee}} \twoheadrightarrow  \widetilde{P}'$. Applying $\Hom_{\fC_{\GL_2(\Q_p)}(\co_E)}(\widetilde{P}_{\pi^{\vee}}, -)$, we obtain a surjection
\begin{equation}\label{def01}
  \widetilde{E}_{\fB} \twoheadlongrightarrow \Hom_{\fC_{\GL_2(\Q_p)}(\co_E)}(\widetilde{P}_{\pi^{\vee}}, \widetilde{P}').
\end{equation}
By the same argument as in \cite[Lem. 3.25]{Pas13}, we have an isomorphism of $\widetilde{E}'$-module:
\begin{equation}\label{deftw}
  \Hom_{\fC_{\GL_2(\Q_p)}(\co_E)}(\widetilde{P}_{\pi^{\vee}}, \widetilde{P}') \cong \widetilde{E}_{\fB,\zeta}\widehat{\otimes}_{\co_E} \Lambda.
\end{equation}
such that the composition of (\ref{def01}) with (\ref{deftw}) gives a surjective homomorphism of $\co_E$-algebra $\delta: \widetilde{E}_{\fB}\twoheadrightarrow \widetilde{E}'$, and $\widetilde{P}'\cong \widetilde{P}_{\pi^{\vee}}\widehat{\otimes}_{\widetilde{E}_{\fB}} \widetilde{E'}$.
 By the same argument as in \cite[\S~9.1]{Pas13} and using
 $$\dim_{k_E} \Ext^1_{\GL_2(\Q_p)}(\pi,\pi)=4$$
  (which for example follows from the fact $\dim_{k_E} \Ext^1_{\GL_2(\Q_p),Z}(\pi,\pi)=2$ (\cite[Prop. 9.1]{Pas13}), and the same argument as in the  proof of \cite[Lem. A.3]{BD1}), we have $\widetilde{E}_{\fB}\twoheadrightarrow \widetilde{E}_{\fB}^{\ab}\cong \co_E[[x_1, x_2, x_3, x_4]]$ (where $\widetilde{E}_{\fB}^{\ab}$ is the  maximal commutative quotient of $\widetilde{E}$). Moreover, one can check that \cite[Lem. 9.2, Lem. 9.3]{Pas13} hold with $\widetilde{E}$ of \emph{loc. cit.} replaced by $\widetilde{E}_{\fB}$, and $\co[[x,y]]$ replaced by $\co_E[[x_1,x_2,x_3,x_4]]$. By \cite[Lem. 9.2]{Pas13}, there exists $t\in \widetilde{E}_{\fB,\zeta}$ such that
\begin{equation*}
  0 \ra \widetilde{E}_{\fB,\zeta} \xlongrightarrow{t} \widetilde{E}_{\fB,\zeta} \ra (\widetilde{E}_{\fB,\zeta})^{\ab} \ra 0,
\end{equation*}
which then induces (noting $\Lambda$ is flat over $\co_E$, and $\widetilde{E}_{\fB,\zeta}$ is $\co_E$-torsion free)
\begin{equation}\label{compts}
  0 \ra \widetilde{E}' \xlongrightarrow{t} \widetilde{E}' \ra (\widetilde{E}_{\fB,\zeta})^{\ab}\widehat{\otimes}_{\co_E} \Lambda \ra 0.
\end{equation}
%Indeed, since $\Lambda$ is flat over $\co_E$, we first have a short exact sequence of torsion free $\co_E$-modules
%\begin{equation}\label{compts0}
%  0\ra \widetilde{E}_{\fB,\zeta} \otimes_{\co_E} \Lambda \xlongrightarrow{t} \widetilde{E}_{\fB,\zeta} \otimes_{\co_E} \Lambda \ra (\widetilde{E}_{\fB,\zeta})^{\ab}\otimes_{\co_E} \Lambda \ra 0.
%\end{equation}
%The exact sequence (\ref{compts}) follows by modulo (\ref{compts0}) by $p^n$ (which is still exact since all the non-zero $\co_E$-modules in (\ref{compts0}) are torsion free) and taking inverse limit (the Mittag-Leffler condition is easily verified).
Using the same argument as in the proof of \cite[Lem. 9.3]{Pas13}, we deduce then $\delta$ is an isomorphism, and $\widetilde{P}_{\pi^{\vee}}\cong \widetilde{P}' \cong \widetilde{P}_{\pi^{\vee},\zeta}\widehat{\otimes}_{\co_E} \Lambda$.
\\

\noindent (b) Let $\chi: \Q_p^{\times} \ra k_E^{\times}$ be such that $\pi\cong \Ind_{\overline{B}(\Q_p)}^{\GL_2(\Q_p)} (\chi \omega^{-1} \otimes \chi)$ (hence $\overline{\zeta}=\chi^2\omega^{-1}$). Thus $R^{\pss}_{\fB} = R^{\pss}_{2\chi}$, the universal deformation ring of the pseudo-character $2\chi$. Denote by $R_{2 \chi}^{\pss, \zeta\varepsilon}$ the universal deformation ring parameterizing $2$-dimensional pseudo-characters of $\Gal_{\Q_p}$ with determinant $\zeta\varepsilon$ lifting $2\chi$. We show there is a natural isomorphism $R_{2\chi}^{\pss} \xrightarrow{\sim} R_{2\chi}^{\pss, \zeta \varepsilon} \widehat{\otimes}_{\co_E} \Lambda$. Let $T^{\univ,\zeta \varepsilon}: \Gal_{\Q_p} \ra R_{2\chi}^{\pss, \zeta \varepsilon}$ be the universal deformation with determinant $\zeta\varepsilon$ of $2\chi$, and put $T': \Gal_{\Q_p} \ra  R_{2\chi}^{\pss, \zeta \varepsilon} \widehat{\otimes}_{\co_E} \Lambda$ be the pseudo-character sending $g$ to $T^{\univ,\zeta\varepsilon}(g) \otimes 1_{\univ}(g)$. By the universal property of $R_{2\chi}^{\pss}$, we obtain a morphism of complete $\co_E$-algebras:
\begin{equation}\label{cent3}
R_{2\chi}^{\pss} \lra  R_{2\chi}^{\pss, \zeta \varepsilon}\widehat{\otimes}_{\co_E} \Lambda.
\end{equation}
By \cite[Cor. 9.13]{Pas13}, $ R_{2\chi}^{\pss, \zeta \varepsilon}\cong \co_E[[x_1,x_2,x_3]]$. Using the fact that taking determinant induces a surjective map $R_{2\chi}^{\pss}(k[\epsilon]/\epsilon^2) \twoheadrightarrow \Lambda(k[\epsilon]/\epsilon^2)$ (since it is easy to construct a section of this map), it is not difficult to see the tangent map of (\ref{cent3}) is bijective, from which we deduce  (\ref{cent3}) is an isomorphism.
\\

\noindent (c) By \cite[Cor. 9.27]{Pas13}, we have a natural isomorphism (which is unique up to conjugation by $\widetilde{E}_{\fB,\zeta}^{\times}$)
\begin{equation}\label{cenPas2}
 \big(R_{2\chi}^{\pss,\zeta \varepsilon}[[\cG]]/J_{\zeta\varepsilon}\big)^{\op} \xlongrightarrow{\sim} \widetilde{E}_{\fB,\zeta},
\end{equation}
where $\cG$ denotes the maximal pro-$p$ quotient of $\Gal_{\Q_p}$, which is a free pro-$p$ group generated by $2$ elements $\gamma$, $\delta$, and where $J_{\zeta\varepsilon}$ denotes the closed two-sided ideal generated by $g^2-T^{\univ,\zeta \varepsilon}(g) g+ \zeta \varepsilon (g)$, for all $g\in \cG$. By (a) and (b), we have
\begin{equation}\label{cenPas}
  \widetilde{E}_{\fB}\cong \widetilde{E}_{\fB,\zeta}\widehat{\otimes}_{\co_E} \Lambda \cong \big(R_{2\chi}^{\pss,\zeta \varepsilon}[[\cG]]/J_{\zeta\varepsilon}\big)^{\op} \widehat{\otimes}_{\co_E} \Lambda \cong  \big(R_{2\chi}^{\pss}[[\cG]]/J_{\zeta\varepsilon}\big)^{\op}.
\end{equation}
In particular, we see by \cite[Cor. 9.25]{Pas13} that $\widetilde{E}_{\fB}$ is a free $R_{2\chi}^{\pss}$-module of rank $4$. Composing with an automorphism of $\Lambda$ if needed, we assume $1_{\univ}(\gamma)=1+x$, and $1_{\univ}(\delta)=1+y$ (recall $1_{\univ}: \cG^{\ab} \ra \Lambda$, and where  we  use $\gamma$, $\delta$ to denote their images in $\cG^{\ab}$). The induced isomorphism $1_{\univ}: \co_E[[\cG^{\ab}]] \xrightarrow{\sim} \co_E[[x,y]]$ lifts to an isomorphism $\co_E[[\cG]] \xrightarrow{\sim} \co_E[[x,y]]^{\nc}$ sending $\gamma$ to $1+x$ and  $\delta$ to $1+y$ (``$\nc$" means non-commutative). Let $\widetilde{J}$ be the closed two-sided ideal of $R_{2\chi}^{\pss}[[\cG]]$ generated by $g^2-T^{\univ}(g)g+\dett(T^{\univ})(g)$. Consider the following  isomorphism of $R_{2\chi}^{\pss}$-algebras
\begin{equation*}
  R_{2\chi}^{\pss}[[\cG]] \xlongrightarrow{\sim} \big(R_{2\chi}^{\pss, \zeta\varepsilon} \widehat{\otimes}_{\co_E}\co_E[[x,y]]\big)[[\cG]]
\end{equation*}
which sends $\gamma$ to $\gamma(1+x)$ and $\delta$ to $\delta(1+y)$. One can check (using $T^{\univ}=T^{\univ,\zeta\varepsilon} \otimes 1_{\univ}$) that this isomorphism induces an isomorphism
\begin{equation*}
   R_{2\chi}^{\pss}[[\cG]]/\widetilde{J} \xlongrightarrow{\sim} R_{2\chi}^{\pss, \zeta\varepsilon}[[\cG]]/J_{\zeta\varepsilon} \widehat{\otimes}_{\co_E} \Lambda.
\end{equation*}
Using the same argument as in \cite[Cor. 9.24]{Pas13}, we have that the center of $R_{2\chi}^{\pss}[[\cG]]/\widetilde{J}$ (hence of $\widetilde{E}_{\fB}$) is equal to $R_{2\chi}^{\pss}$.
\\

\noindent (d) We show the injection $R_{2\chi}^{\pss} \hookrightarrow \widetilde{E}_{\fB}$ is independent of the choice of $\zeta$, by unwinding a little the isomorphism in (\ref{cenPas2}). Let $\eta:\Q_p^{\times} \ra \co_E^{\times}$ be such that $\eta\equiv 1 \pmod{\varpi_E}$. We have a natural $\GL_2(\Q_p)$-equivariant isomorphism $\widetilde{P}_{\pi^{\vee}, \zeta\eta^2} \cong \widetilde{P}_{\pi^{\vee}, \zeta} \otimes_{\co_E} (\eta^{-1}\circ \dett)$, which induces an isomorphism $\tw_{\eta}: \widetilde{E}_{\fB,\zeta\eta^2} \xrightarrow{\sim}\widetilde{E}_{\fB,\zeta}$. Twisting $\eta$ also induces an isomorphism $\tw_{\eta}: R^{\pss,\zeta\eta^2}_{2\chi} \xrightarrow{\sim}R^{\pss,\zeta}_{2\chi}$.  Denote by $\check{\hV}_{\zeta}$ (resp. $\check{\hV}_{\zeta\eta^2}$) the functor $\hV$ of \cite[\S~5.7]{Pas13} on $\fC_{\GL_2(\Q_p),\zeta}(\co_E)$ (resp. on $\fC_{\GL_2(\Q_p), \zeta \eta^2}(\co_E)$) associated to $\zeta$ (resp. to $\zeta\eta^2$). By definition (cf. \emph{loc. cit.}), we have a $\Gal_{\Q_p}$-equivariant isomorphism
\begin{equation}\label{twist0}
  \check{\hV}_{\zeta \eta^2}(\widetilde{P}_{\pi^{\vee}, \zeta \eta^2}) \cong \check{\hV}_{\zeta}(\widetilde{P}_{\pi^{\vee}, \zeta}) \otimes_{\co_E} \eta.
\end{equation}
As in the discussion below \cite[Lem. 9.3]{Pas13}, we can deduce from (\ref{twist0}) a commutative diagram
\begin{equation}\label{pascom}\begin{CD}
  \co_E[[\cG]]^{\op} @> \id >> \co_E[[\cG]]^{\op}\\
  @V \varphi_{\check{\hV}_{\zeta \eta^2}} VV @V \eta \otimes \varphi_{\check{\hV}_{\zeta}} VV \\
  \widetilde{E}_{\fB,\zeta\eta^2} @>\tw_{\eta} >> \widetilde{E}_{\fB, \zeta},
  \end{CD}
\end{equation}
where ``$\eta$" in the right vertical map denotes the composition $\co_E[[\cG]]^{\op} \xrightarrow{\eta} \co_E \hookrightarrow \widetilde{E}_{\fB, \zeta}$, and where $\varphi_{\check{\hV}_{\zeta \eta^2}}$ (resp. $\varphi_{\check{\hV}_{\zeta}}$) is the map $\varphi_{\check{\hV}}$ of \cite[\S~9.1]{Pas13} (which is unique up to conjugation by $\widetilde{E}_{\fB,\zeta\eta^2}^{\times}$ (resp. $\widetilde{E}_{\fB,\zeta}^{\times}$), but we can choose the maps so that (\ref{pascom}) commutes). Hence $\varphi_{\check{\hV}_{\zeta\eta^2}}$ is equal to the composition
\begin{equation*}
  \co_E[[\cG]]^{\op} \xlongrightarrow{ \eta\otimes\id} \co_E[[\cG]]^{\op} \xlongrightarrow{\varphi_{\check{\hV}_{\zeta}}} \widetilde{E}_{\fB,\zeta}\xlongrightarrow[\sim]{\tw_{\eta^{-1}}} \widetilde{E}_{\fB,\zeta\eta^2}.
\end{equation*}
We deduce that the following diagram commutes
\begin{equation*}
\begin{CD}
\co_E[[\cG]] @> \eta \otimes \id>> \co_E[[\cG]]\\
 @VVV @VVV \\
 R^{\pss, \zeta\eta^2}_{2\chi}[[\cG]]/J_{\zeta\eta^2\varepsilon} @>>> R_{2\chi}^{\pss, \zeta}[[\cG]]/J_{\zeta\varepsilon} \\
  @V \sim VV @V \sim VV \\
  \big(\widetilde{E}_{\fB,\zeta\eta^2}\big)^{\op} @> \tw_{\eta}>> \big(\widetilde{E}_{\fB,\zeta}\big)^{\op}
  \end{CD}
\end{equation*}
where the middle horizontal map sends $g$ to $\eta(g) g$ for $g\in \cG$, and sends $a$ to $\tw_{\eta}(a)$ for $a\in  R^{\pss, \zeta\eta^2}_{2\chi}$, where the vertical maps in the top square are the surjections given as in \cite[(150)]{Pas13}, and where the vertical maps in the bottom square are given as in (\ref{cenPas2}), induced by  $\varphi_{\check{\hV}_{\zeta \eta^2}}$ , $\varphi_{\check{\hV}_{\zeta}}$ respectively (see \cite[\S~9.2]{Pas13} for details). In particular, the following diagram commutes
\begin{equation*}
   \begin{CD}
 R^{\pss, \zeta\eta^2}_{2\chi} @> \tw_{\eta} >> R_{2\chi}^{\pss, \zeta} \\
  @VVV @VVV \\
  \big(\widetilde{E}_{\fB,\zeta\eta^2}\big)^{\op} @> \tw_{\eta}>> \big(\widetilde{E}_{\fB,\zeta}\big)^{\op}.
   \end{CD}
 \end{equation*}Together with similar commutative diagrams as in \cite[(6.4)]{CEGGPS2} replacing ``$R_p$" by $\widetilde{E}_{\fB}$ and $R^{\pss}_{2\chi}$, we deduce that the composition
\begin{equation*}
  R^{\pss}_{2\chi} \cong R^{\pss, \zeta\eta^2}_{2\chi} \widehat{\otimes}_{\co_E} \Lambda \hooklongrightarrow \big(\widetilde{E}_{\fB,\zeta\eta^2}\big)^{\op} \widehat{\otimes}_{\co_E} \Lambda \cong \widetilde{E}_{\fB}^{\op}
\end{equation*}
coincides with the one induced by (\ref{cenPas}). This concludes the proof.
%$\Ker[\widetilde{E}^{\xi} \ra (\widetilde{E}^{\xi})^{\ab}]=\widetilde{E}^{\xi} t$, and $at\neq 0$ for all non-zero $a\in \widetilde{E}^{\xi}$. We deduce then
%\begin{equation*}
%  \Ker[\widetilde{E}' \twoheadrightarrow (\widetilde{E}')^{\ab}]=\widetilde{E}' t.
%\end{equation*}
%Indeed, we have $(\widetilde{E}')^{\ab}\twoheadrightarrow (\widetilde{E}^{\xi})^{\ab}\widehat{\otimes}_{\co_E} \Lambda$. By comparing the the tangent space, we see this is an isomorphism.
\end{proof}
\begin{remark}\label{normB}
 Let $\pi$ be the $\GL_2(\Q_p)$-representation in the proof of Theorem \ref{blo0} for the case (1)(2)(4). Let $\zeta: \Q_p^{\times} \ra \co_E^{\times}$ be such that $\overline{\zeta}$ is equal to the central character of $\pi$. Let $\check{\hV}_{\zeta}$ be  the functor $\check{\hV}$ of \cite[\S~5.7]{Pas13} on $\fC_{\GL_2(\Q_p),\zeta}(\co_E)$ (which depends on the choice of $\zeta$). As in \cite[Prop. 6.3, Cor. 8.7, Thm. 10.71]{Pas13}, the functor $\check{\hV}_{\zeta}$ induces an isomorphism
  \begin{equation*}
  R_{\overline{\rho}_{\pi}}^{\zeta \varepsilon} \xlongrightarrow{\sim} \End_{\fC_{\GL_2(\Q_p)}(\co_E)}(\widetilde{P}_{\pi^{\vee}})
  \end{equation*}
  where $R_{\overline{\rho}_{\pi}}^{\zeta \varepsilon}$ denotes the universal deformation over $\overline{\rho}$ of deformations with determinant equal to $\zeta \varepsilon$. The isomorphism in (\ref{equ: Colm}) is given  by the composition
  \begin{equation*}
    R_{\overline{\rho}_{\pi}}\cong R_{\overline{\rho}_{\pi}}^{\zeta \varepsilon} \widehat{\otimes}_{\co_E} \Lambda \xlongrightarrow{\sim}  \End_{\fC_{\GL_2(\Q_p),\zeta}(\co_E)}(\widetilde{P}_{\pi^{\vee}}) \widehat{\otimes}_{\co_E} \Lambda  \cong \End_{\fC_{\GL_2(\Q_p),\zeta}(\co_E)}(\widetilde{P}_{\pi^{\vee}}).
  \end{equation*}
We deduce by \cite[(6.4)]{CEGGPS2} and the diagram in the proof of \cite[Lem. 6.9]{CEGGPS2} that the isomorphism in (\ref{equ: Colm}) is (also) independent of the choice of $\zeta$.
%\\
%
%\noindent (2) Let $\rho: \Gal_{\Q_p} \ra \GL_2(E)$ be a continuous representation, $\rho_0$ be a $\Gal_{\Q_p}$-invariant lattice of $\rho$, and  $\overline{\rho}^{\sss}$ be the semi-simplification of the modulo $\varpi_E$ reduction of $\rho_0$. Let $\fB$ be the block such that $\overline{\rho}_{\fB}\cong \overline{\rho}^{\sss}$. Let $\fp$ be the prime ideal of $R_{\tr\overline{\rho}^{\sss}}^{\pss}$ associated to $\tr\rho_0: \Gal_{\Q_p} \ra \co_E$.  Let $\widehat{\pi}(\rho)$ be the $\GL_2(\Q_p)$-representation associated to $\rho$ via the $p$-adic local Langlands correspondence, and let $\widehat{\pi}(\rho)^0$ be a $\GL_2(\Q_p)$-invariant lattice. Then the action of $R_{\tr\overline{\rho}_{\fB}}^{\pss}[1/p]$ on
%\begin{equation*}
%  \Hom_{\fC_{\GL_2(\Q_p)}(\co_E)}\big(\widetilde{P}_{\fB}, \big(\widehat{\pi}(\rho)^0\big)^d\big)[1/p]
%\end{equation*}
%factors through $R_{\tr\overline{\rho}_{\fB}}^{\pss}[1/p] \twoheadrightarrow R_{\tr\overline{\rho}_{\fB}}^{\pss}/\fp[1/p]\cong E$.
\end{remark}
\noindent For $\fm$ a maximal ideal of $R_{\fB}^{\pss}[1/p]$ with $\fp:=\fm\cap R_{\fB}^{\pss}$, we denote by $\widehat{\pi}_{\fB,\fm}$ the multiplicity free direct sum of the irreducible constituents of the finite length Banach representation
\begin{equation*}\Hom_{\co_E}^{\cts}\big(\widetilde{P}_{\fB} \widehat{\otimes}_{R_{\fB}^{\pss}} (R_{\fB}^{\pss}/\fp), E\big).%=\Hom_{\co_E}^{\cts}\big(\widetilde{P}_{\fB} \widehat{\otimes}_{R_{\fB}} (R_{\fB}/\fp)_{\tf}, E\big),
 \end{equation*}

\noindent Let $H= \prod_i H_i$ be a finite product with $H_i\cong \Q_p^{\times}$ or $\GL_2(\Q_p)$. By \cite[Lem. 3.4.10, Cor. 3.4.11]{Pan2} (and the proof), we have
\begin{proposition}\label{bloten}Any block $\fB$ of $\Mod_H^{\lfin}(\co_E)$ is of the form
  \begin{equation*}
    \fB=\otimes_i \fB_i:=\{\otimes_{\pi_i\in \fB_i}  \pi_i\}
  \end{equation*}
  where $\fB_i$ is a block of $\Mod_{H_i}^{\lfin}(\co_E)$. And we have $\widetilde{P}_{\fB} \cong \widehat{\otimes}_i \widetilde{P}_{\fB_i}$,
$ \widetilde{E}_{\fB} \cong \widehat{\otimes}_i \widetilde{E}_{\fB_i}$.
\end{proposition}

\subsection{$\GL_2(\Q_p)$-ordinary families}\label{GL2ord}
\noindent We apply Pa{\v{s}}k{\=u}nas' theory to construct $\GL_2(\Q_p)$-ordinary families.
\\

\noindent Let $\fC:=\fC_{L_P(\Q_p)}(\co_E)$. We decompose the space of $P$-ordinary automorphic representations using the theory of blocks. We have
\begin{equation*}
  \Ord_P(\widehat{S}(U^p,\co_E)_{\overline{\rho}})^d=\Hom_{\co_E}(\Ord_P(\widehat{S}(U^p,\co_E)_{\overline{\rho}}),\co_E)\in \fC.
\end{equation*}For $k\in \Z_{\geq 1}$, the Pontryagain dual $\Ord_P(\widehat{S}(U^p, \co_E/\varpi_E^k)_{\overline{\rho}})^{\vee}$  is also an onject in $\fC$. We have an $\widetilde{\bT}(U^p)^{P-\ord}_{\overline{\rho}}$-equivariant isomorphism in $\fC$ (cf. (\ref{duala})):
\begin{equation}\label{Pord: ordd}
  \Ord_P(\widehat{S}(U^p, \co_E)_{\overline{\rho}})^d \cong \varprojlim_k \Ord_P(\widehat{S}(U^p, \co_E/\varpi_E^k)_{\overline{\rho}})^{\vee}.
\end{equation}
%where the transition map
%\begin{equation*}\Ord_P(\widehat{S}(U^p, \co_E/\varpi_E^{k+1})_{\overline{\rho}})^{\vee} \lra \Ord_P(\widehat{S}(U^p, \co_E/\varpi_E^k)_{\overline{\rho}})^{\vee}\end{equation*} is induced by the injection
%\begin{equation*}
%   \Ord_P(\widehat{S}(U^p, \co_E/\varpi_E^k)_{\overline{\rho}})^{\vee} \xlongrightarrow{\varpi_E} \Ord_P(\widehat{S}(U^p, \co_E/\varpi_E^{k+1})_{\overline{\rho}}).
%\end{equation*}
For $M\in \Mod_{L_P(\Q_p)}^{\lfin}(\co_E)$ (resp. in $\fC$), and $\fB$ a block of $\Mod_{L_P(\Q_p)}^{\lfin}(\co_E)$ (hence can also be viewed as a block of $\fC$), we denote by $M_{\fB}$ the maximal direct summand of $M$ such that all the irreducible subquotients $\pi$ of $M_{\fB}$ satisfy  $\pi\in \fB$ (resp. $\pi^{\vee}\in \fB$).
We have thus decompositions (cf. \cite[Prop. 5.36]{Pas13})
\begin{eqnarray}\label{decomp00}
    \Ord_P(\widehat{S}(U^p,\co_E)_{\overline{\rho}})^d &\cong& \oplus_{\fB} \Ord_P(\widehat{S}(U^p,\co_E)_{\overline{\rho}})^d_{\fB}, \\
    \Ord_P(\widehat{S}(U^p, \co_E/\varpi_E^k)_{\overline{\rho}}) & \cong & \oplus_{\fB}  \Ord_P(\widehat{S}(U^p, \co_E/\varpi_E^k)_{\overline{\rho}})_{\fB}, \nonumber \\
        \Ord_P(\widehat{S}(U^p, \co_E/\varpi_E^k)_{\overline{\rho}})^{\vee}& \cong & \oplus_{\fB}  \Ord_P(\widehat{S}(U^p, \co_E/\varpi_E^k)_{\overline{\rho}})^{\vee}_{\fB}. \nonumber
\end{eqnarray}
It is also clear that the isomorphism in (\ref{Pord: ordd}) respects the decompositions.
\\

\noindent Recall that  $\Ord_P(\widehat{S}(U^p,\co_E)_{\overline{\rho}})$ is admissible, hence $\Ord_P(\widehat{S}(U^p,\co_E)_{\overline{\rho}})^d$ is finitely generated over $\co_E[[L_P(\Z_p)]]$. By \cite[(2.2.12)]{EOrd1}, the Pontryagain dual of  $\Ord_P(\widehat{S}(U^p,\co_E)_{\overline{\rho}})^d$ is a smooth admissible representation of $L_P(\Q_p)$ over $\co_E$. Hence there are finitely many blocks $\fB$ of $\Mod_{L_P(\Q_p)}^{\lfin}(\co_E)$ such that $\Ord_P(\widehat{S}(U^p,\co_E)_{\overline{\rho}})^d_{\fB}\neq 0$ \big(which is equivalent to
$\Ord_P(\widehat{S}(U^p, \co_E/\varpi_E^k)_{\overline{\rho}})^{\vee}_{\fB} \neq 0$, by (\ref{Pord: ordd})\big). Put
\begin{eqnarray*}
  \Ord_P(\widehat{S}(U^p,\co_E)_{\overline{\rho}})_{\fB}&:=&\Hom^{\cts}_{\co_E}(    \Ord_P(\widehat{S}(U^p,\co_E)_{\overline{\rho}})^d_{\fB}, \co_E), \\
    \Ord_P(\widehat{S}(U^p,E)_{\overline{\rho}})_{\fB}&:=&   \Ord_P(\widehat{S}(U^p,\co_E)_{\overline{\rho}})_{\fB} \otimes_{\co_E} E,
\end{eqnarray*}
which are equipped with the supreme norm.
For $*\in \{\co_E, E, \co_E/\varpi_E^k\}$, $\Ord_P(\widehat{S}(U^p,*)_{\overline{\rho}})_{\fB}$ is a direct summand of $  \Ord_P(\widehat{S}(U^p,*)_{\overline{\rho}})$ (e.g. using \ref{decomp00} and (\ref{duali})),  and we have
\begin{equation}\label{pord: proj00}
  \Ord_P(\widehat{S}(U^p,\co_E)_{\overline{\rho}})_{\fB}\cong \varprojlim_n \Ord_P(S(U^p,\co_E/\varpi_E^n)_{\overline{\rho}})_{\fB}.
\end{equation}
The following lemma follows easily from Lemma \ref{LPzp} and \cite[Cor. 7.7]{BD1}.
 \begin{lemma}\label{densbl}
   Let $\fB$ be a block of $\Mod_{L_P(\Q_p)}^{\lfin}(\co_E)$ such that $ \Ord_P(\widehat{S}(U^p,\co_E)_{\overline{\rho}})_{\fB}\neq 0$.

\noindent (1) $ \Ord_P(\widehat{S}(U^p,\co_E)_{\overline{\rho}})_{\fB}|_{L_P(\Z_p)}$ is isomorphic to a direct summand of $\cC(L_P(\Z_p),\co_E)^{\oplus r}$ for some $r>1$.

\noindent (2) $\big(\Ord_P(\widehat{S}(U^p,\co_E)_{\overline{\rho}})_{\fB}\big)^{L_P(\Z_p)-\alg}_+$ is dense in $ \Ord_P(\widehat{S}(U^p,\co_E)_{\overline{\rho}})_{\fB}$.
 \end{lemma}
\noindent We have
\begin{multline}\label{fB0}
\Ord_P(\widehat{S}(U^p,\co_E)_{\overline{\rho}})^d_{\fB} \cong \widetilde{P}_{\fB}\widehat{\otimes}_{\widetilde{E}_{\fB}} \Hom_{\fC}\big(\widetilde{P}_{\fB}, \Ord_P(\widehat{S}(U^p,\co_E)_{\overline{\rho}})^d_{\fB}\big) \\
  \cong \widetilde{P}_{\fB}\widehat{\otimes}_{\widetilde{E}_{\fB}} \Hom_{\fC}\big(\widetilde{P}_{\fB}, \Ord_P(\widehat{S}(U^p,\co_E)_{\overline{\rho}})^d\big) \hooklongrightarrow \Ord_P(\widehat{S}(U^p,\co_E)_{\overline{\rho}})^d
\end{multline}
where the first isomorphism follows from \cite[Lem. 2.10]{Pas13} (applying $\Hom_{\fC}(\widetilde{P}_{\fB},-)$ to \cite[(6)]{Pas13} with $M=\Ord_P(\widehat{S}(U^p,\co_E)_{\overline{\rho}})^d_{\fB}$, we easily deduce that the kernel is zero), the second isomorphism follows from (\ref{decomp00}) and $\Hom_{\fC}\big(\widetilde{P}_{\fB},  \Ord_P(\widehat{S}(U^p,\co_E)_{\overline{\rho}})^d_{\fB'})=0$ for $\fB'\neq \fB$, and where the last injection is the evaluation map (indeed, applying $\Hom_{\fC}(\widetilde{P}_{\fB'}, -)$ to the kernel of this map, we get zero for all $\fB'$, from which we deduce that the kernel has to be zero). We see that $ \Ord_P(\widehat{S}(U^p,\co_E)_{\overline{\rho}})^d_{\fB}$ inherits a natural $\widetilde{\bT}(U^p)^{P-\ord}_{\overline{\rho}}$-action from $\Ord_P(\widehat{S}(U^p,\co_E)_{\overline{\rho}})^d$ (via the first two isomorphisms in (\ref{fB0})), so that the decomposition (\ref{decomp00}) is in fact $\widetilde{\bT}(U^p)^{P-\ord}_{\overline{\rho}}$-equivariant. Similarly, for all $k\geq 1$, $\Ord_P(S(U^p,\co_E/\varpi_E^k)_{\overline{\rho}})_{\fB}$ is also a $\widetilde{\bT}(U^p)^{P-\ord}_{\overline{\rho}}$-equivariant direct summand of $\Ord_P(S(U^p,\co_E/\varpi_E^k)_{\overline{\rho}})$. For $i\geq 0$,  $\big(\Ord_P(S(U^p,\co_E/\varpi_E^k)_{\overline{\rho}})_{\fB}\big)^{L_i}$
is hence  a $\widetilde{\bT}(U^p)^{P-\ord}_{\overline{\rho}}$-equivariant direct summand of (cf. (\ref{isoOrd}))
\begin{equation*}\Ord_P(S(U^p, \co_E/\varpi_E^k)_{\overline{\rho}})^{L_i}\cong S(U^p K_{i,i},\co_E/\varpi_E^k)_{\overline{\rho},\ord}.
\end{equation*}
 Denote by $\bT(U^p K_{i,i}, \co_E/\varpi_E^k)_{\overline{\rho},\fB}^{P-\ord}$ the image of
\begin{equation*}
\bT(U^p)_{\overline{\rho}} \lra \End_{\co_E}\big(\big(\Ord_P(S(U^p,\co_E/\varpi_E^k)_{\overline{\rho}})_{\fB}\big)^{L_i}\big),
\end{equation*}
and put
\begin{equation*}
 \widetilde{\bT}(U^p)_{\overline{\rho},\fB}^{P-\ord}:=\varprojlim_k \varprojlim_i \bT(U^pK_{i,i}, \co_E/\varpi_E^k)_{\overline{\rho},\fB}^{P-\ord}.
\end{equation*}
It is clear that    $\widetilde{\bT}(U^p)_{\overline{\rho},\fB}^{P-\ord}$ is  a quotient of $\widetilde{\bT}(U^p)_{\overline{\rho}}^{P-\ord}$ hence is also a complete local noetherian $\co_E$-algebra of residue field $k_E$. Similarly as in Lemma \ref{lem: Pord-reduce}, we have
\begin{lemma}
  The $\co_E$-algebra $\widetilde{\bT}(U^p)_{\overline{\rho},\fB}^{P-\ord}$ is reduced and the natural action of $\widetilde{\bT}(U^p)_{\overline{\rho},\fB}^{P-\ord}$ on $\Ord_P(\widehat{S}(U^{p}, \co_E)_{\overline{\rho},\fB})$ and $\Ord_P(\widehat{S}(U^p, E)_{\overline{\rho}, \fB})$ is faithful.
\end{lemma}
\noindent Since $ \Ord_P(\widehat{S}(U^p,E)_{\overline{\rho}})_{\fB}$ is an $L_P(\Q_p)\times \widetilde{\bT}(U^p)_{\overline{\rho},\fB}^{P-\ord}$-equivariant direct summand of $\Ord_P(\widehat{S}(U^p,E)_{\overline{\rho}})$, by the same argument, we have as in  Proposition \ref{beni1}, \ref{beni2}:
\begin{proposition}\label{beni3}
  (1) The benign points are Zariski-dense in $\Spec \widetilde{\bT}(U^p)_{\overline{\rho},\fB}^{P-\ord}[1/p]$.

\noindent (2) Let $x$ be a benign point of $\Spec \widetilde{\bT}(U^p)_{\overline{\rho},\fB}^{P-\ord}[1/p]$, the statements in Proposition \ref{beni1} (2), Proposition \ref{beni2} hold with $\Ord_P\big(\widehat{S}(U^p, E\big)_{\overline{\rho}}[\fm_x]\big)$ replaced by $\Ord_P\big(\widehat{S}(U^p, E\big)_{\overline{\rho}}[\fm_x]\big)_{\fB}$.
\end{proposition}
\noindent By Proposition \ref{bloten}, there exist blocks $\fB_{\widetilde{v},i}$ for $v\in S_p$, $i=1,\cdots, k_{\widetilde{v}}$ such that %(note $s_{\widetilde{v},i+1}=s_{\widetilde{v},i}+n_{\widetilde{v},i}$)
\begin{equation}\label{block0}\fB=\otimes_{v\in S_p} (\otimes_{i=1,\cdots, k_{\widetilde{v}}} \fB_{\widetilde{v},i}(s_{\widetilde{v},i+1}-1))=:\otimes_{v\in S_p} \fB_{\widetilde{v}},\end{equation}
where  $\fB(r)$ denotes the block $\{\pi\otimes_{k_E} \omega^r \circ \dett\ |\ \pi\in \fB\}$ for a block $\fB$. If $p=3$, we assume that $\fB_{\widetilde{v},i}$ is not in case (4') for all $v$, $i$ with $n_{\widetilde{v},i}=2$. For $v\in S_p$, $i=1,\cdots, k_{\widetilde{v}}$, twisting $\varepsilon^{1-s_{\widetilde{v},i+1}}: \Gal_{\Q_p} \ra \co_E^{\times}$ induces an isomorphism
$\tw_{\widetilde{v},i}: R_{\fB_{\widetilde{v},i}}^{\pss} \xrightarrow{\sim} R_{\fB_{\widetilde{v},i}(s_{\widetilde{v},i+1}-1)}^{\pss}$.
Put \begin{equation}\label{RpB}R_{p,\fB}:=\widehat{\otimes}_{v\in S_p}  \big(\widehat{\otimes}_{i=1,\cdots, k_{\widetilde{v}}} R_{\fB_{\widetilde{v},i}}^{\pss}\big).
\end{equation}
Denote by $$\tm(U^p, \fB):=\Hom_{\fC}\big(\widetilde{P}_{\fB}, \Ord_P(\widehat{S}(U^p,\co_E)_{\overline{\rho}})^d_{\fB}\big),$$ which is a compact $\widetilde{E}_{\fB}$-module. By Theorem \ref{blo0}, Proposition \ref{bloten} and  the fact $\Ord_P(\widehat{S}(U^p,\co_E)_{\overline{\rho}})_{\fB}$ is admissible, $\tm(U^p,\fB)$ is finitely generated over $R_{p,\fB}$ where the action of $R_{p,\fB}$ is induced from the natural action of $\widetilde{E}_{\fB}$ via
\begin{equation*}
  R_{p,\fB}\xlongrightarrow[\sim]{(\tw_{\widetilde{v},i})}  \widehat{\otimes}_{v\in S_p}  \big(\widehat{\otimes}_{i=1,\cdots, k_{\widetilde{v}}} R_{\fB_{\widetilde{v},i}(s_{\widetilde{v},i+1}-1)}^{\pss}\big) \hooklongrightarrow \widetilde{E}_{\fB}.
\end{equation*}
The  (faithful) $\widetilde{\bT}(U^p)_{\overline{\rho},\fB}^{P-\ord}$-action on $\Ord_P(\widehat{S}(U^p,\co_E)_{\overline{\rho}})^d_{\fB}$ (commuting with $L_P(\Q_p)$) induces a faithful $\widetilde{E}_{\fB}$-linear action of $\widetilde{\bT}(U^p)_{\overline{\rho},\fB}^{P-\ord}$ on $\tm(U^p,\fB)$. We denote by $\cA$ the $R_{p,\fB}$-subalgebra of $\End_{R_{p,\fB}}\big(\tm(U^p,\fB)\big)$ generated by the image of the composition
\begin{equation}\label{PordcA}
  \widetilde{\bT}(U^p)_{\overline{\rho},\fB}^{P-\ord} \hooklongrightarrow \End_{\widetilde{E}_{\fB}}\big(\tm(U^p,\fB)\big) \hooklongrightarrow \End_{R_{p,\fB}}\big(\tm(U^p,\fB)\big).
\end{equation}
By definition, $\cA$ is commutative and finite over $R_{p,\fB}$, and it is not difficult to see that $\cA$ is $\co_E$-torsion free and $\cap_{j\geq 0} \varpi_E^j \cA=0$ (using similar properties for $\tm(U^p,\fB)$, and the fact that the $\cA$-action on $\tm(U^p,\fB)$ is faithful).  We also have a surjective map
\begin{equation*}
  \widetilde{\bT}(U^p)_{\overline{\rho}, \fB}^{P-\ord} \widehat{\otimes}_{\co_E} R_{p,\fB} \twoheadlongrightarrow \cA,
\end{equation*}
and hence a natural embedding:
\begin{equation*}
  (\Spf \cA)^{\rig} \hooklongrightarrow (\Spf \widetilde{\bT}(U^p)_{\overline{\rho}, \fB}^{P-\ord})^{\rig}\times (\Spf R_{p,\fB})^{\rig}
\end{equation*}
 such that the composition (which one can view as an analogue of the  weight map of Hida families)
\begin{equation*}
  \kappa: (\Spf \cA)^{\rig} \hooklongrightarrow (\Spf \widetilde{\bT}(U^p)_{\overline{\rho}, \fB}^{P-\ord})^{\rig}\times (\Spf R_{p,\fB})^{\rig} \lra (\Spf R_{p,\fB})^{\rig}
\end{equation*}
is finite. For each point $z_{\widetilde{v},i}$ of $(\Spf R_{\fB_{\widetilde{v},i}}^{\pss})^{\rig}$, we can attach a representation $\widehat{\pi}_{z_{\widetilde{v},i}}$ of $\GL_{n_{\widetilde{v},i}}(\Q_p)$ such that if $n_{\widetilde{v},i}=1$, $\widehat{\pi}_{z_{\widetilde{v},i}}$ is the corresponding continuous (unitary) character of $\Q_p^{\times}$, and if $n_{\widetilde{v},i}=2$, $\widehat{\pi}_{z_{\widetilde{v},i}}:=\widehat{\pi}_{\fB_{\widetilde{v},i}, \fm_{z_{\widetilde{v},i}}}$ (see the discussion below Remark \ref{normB}).
\begin{proposition}\label{pts0}
  Let $y=(x,z)=(x,(z_{\widetilde{v},i}))\in  (\Spf \widetilde{\bT}(U^p)_{\overline{\rho}, \fB}^{P-\ord})^{\rig}\times (\Spf R_{p,\fB})^{\rig}$. Then $y\in (\Spf \cA)^{\rig}$ if and only if
  \begin{equation}\label{hom0}
    \Hom_{L_P(\Q_p)}\Big(\widehat{\otimes}_{\substack{v\in S_p \\ i=1,\cdots, k_{\widetilde{v}}}} (\widehat{\pi}_{z_{\widetilde{v},i}} \otimes_{k(y)} \varepsilon^{s_{\widetilde{v},i+1}-1}\circ \dett), \Ord_P\big(\widehat{S}(U^p,E)_{\overline{\rho}}\big)_{\fB}[\fm_x]\Big)\neq 0.
  \end{equation}
\end{proposition}
\begin{proof}
Without loss of generality, we assume $k(y)=E$.
Denote by $\fp_x\subset \widetilde{\bT}(U^p)^{P-\ord}_{\overline{\rho}, \fB}$ (resp. $\fp_z \subset R_{p,\fB}$, resp. $\fp_{z_{\widetilde{v},i}}\subset R^{\pss}_{\fB_{\widetilde{v},i}}$) the prime ideal associated to $x$ (resp. $z$, resp. $z_{\widetilde{v},i}$). By definition of $\widehat{\pi}_{z_{\widetilde{v},i}}$, we see (\ref{hom0}) is equivalent to
\begin{equation}\label{exipt}
  \Hom_{\fC}\Big(\Ord_P(\widehat{S}(U^p,\co_E)_{\overline{\rho}}[\fp_x])^d_{\fB}, \widehat{\otimes}_{\substack{v\in S_p \\ i=1,\cdots, k_{\widetilde{v}}}}  \big(\widetilde{P}_{\fB_{\widetilde{v},i}} \widehat{\otimes}_{\widetilde{E}_{\fB_{\widetilde{v},i}}} (\widetilde{E}_{\fB_{\widetilde{v},i}}/\fp_{z_{\widetilde{v},i}})_{\tf}\big)\Big) \neq 0,
\end{equation}
where ``$\tf$" denotes the $\co_E$-torsion free quotient.
Suppose  (\ref{exipt}) holds. Let $f$ be a non-zero morphism in (\ref{exipt}), and consider the composition
\begin{equation*}
\Ord_P(\widehat{S}(U^p,\co_E)_{\overline{\rho}})^d_{\fB} \twoheadlongrightarrow \Ord_P(\widehat{S}(U^p,\co_E)_{\overline{\rho}}[\fp_x])^d_{\fB} \xlongrightarrow{f} \widehat{\otimes}_{\substack{v\in S_p \\ i=1,\cdots, k_{\widetilde{v}}}}  \big(\widetilde{P}_{\fB_{\widetilde{v},i}} \widehat{\otimes}_{\widetilde{E}_{\fB_{\widetilde{v},i}}} (\widetilde{E}_{\fB_{\widetilde{v},i}}/\fp_{z_{\widetilde{v},i}})_{\tf}\big).
\end{equation*}
Applying $\Hom_{\fC}(\widetilde{P}_{\fB},-)$, we obtain a non-zero $ \widetilde{\bT}(U^p)_{\overline{\rho}, \fB}^{P-\ord}\times R_{p,\fB}$-equivariant map
\begin{equation}\label{ptexist}
\tm(U^p,\fB)/\fp_x \lra \Hom_{\fC}\big(\widetilde{P}_{\fB}, \widehat{\otimes}_{\substack{v\in S_p \\ i=1,\cdots, k_{\widetilde{v}}}} \big(\widetilde{P}_{\fB_{\widetilde{v},i}} \widehat{\otimes}_{\widetilde{E}_{\fB_{\widetilde{v},i}}} (\widetilde{E}_{\fB_{\widetilde{v},i}}/\fp_{z_{\widetilde{v},i}})_{\tf}\big)\big) \cong \otimes_{\substack{v\in S_p\\ i=1,\cdots, k_{\widetilde{v}}}} (\widetilde{E}_{\fB_{\widetilde{v},i}}/\fp_{z_{\widetilde{v},i}})_{\tf},
\end{equation}
where the tensor product on the right hand side is over $\co_E$ and the second isomorphism follows from \cite[Lem. 2.9]{Pas13} and the proof of \cite[Lem. B.8]{GN16}. Note also the $R_{p,\fB}$-action on the right hand side of (\ref{ptexist}) factors through $R_{p,\fB}/\fp_z$. In particular, $\tm(U^p,\fB)$ admits an $\co_E$-torsion free quotient on which  $\widetilde{\bT}(U^p)_{\overline{\rho}, \fB}^{P-\ord}$ (resp. $R_{p,\fB}$) acts via $\widetilde{\bT}(U^p)_{\overline{\rho}, \fB}^{P-\ord}/\fp_x$ (resp. $R_{p,\fB}/\fp_z$). We deduce then $y=(x,z)\in (\Spf \cA)^{\rig}$.
\\

\noindent Conversely, suppose $y =(x,z) \in (\Spf \cA)^{\rig}$, and let $\fp_y\subset \cA$ be the prime ideal associated to $y$. We have
\begin{equation*}
 \tm(U^p,\fB) \twoheadlongrightarrow (\tm(U^p,\fB)\otimes_{\cA} \cA/\fp_y)_{\tf}\neq 0.
\end{equation*}
From which we deduce
\begin{equation*}
 \Ord_P(\widehat{S}(U^p,\co_E)_{\overline{\rho}})^d_{\fB} \cong \widetilde{P}_{\fB}\widehat{\otimes}_{\widetilde{E}_{\fB}} \tm(U^p,\fB) \twoheadlongrightarrow  \widetilde{P}_{\fB}\widehat{\otimes}_{\widetilde{E}_{\fB}} (\tm(U^p,\fB)\otimes_{\cA} \cA/\fp_y)_{\tf}.
\end{equation*}
Considering the $\widetilde{\bT}(U^p)_{\overline{\rho},\fB}^{P-\ord}$-action, we see the above map factors through
\begin{equation}\label{ptmap0}
  \Ord_P(\widehat{S}(U^p,\co_E)_{\overline{\rho}}[\fp_x])^d_{\fB} \twoheadlongrightarrow  \widetilde{P}_{\fB}\widehat{\otimes}_{\widetilde{E}_{\fB}} (\tm(U^p,\fB)\otimes_{\cA} \cA/\fp_y)_{\tf}.
\end{equation} Applying $\Hom_{\co_E}^{\cts}(-, E)$ to  (\ref{ptmap0}), we obtain an injection
\begin{equation}\label{ptmap1}
  \Hom_{\co_E}^{\cts}(\widetilde{P}_{\fB}\widehat{\otimes}_{\widetilde{E}_{\fB}} (\tm(U^p,\fB)\otimes_{\cA} \cA/\fp_y)_{\tf}, E) \hooklongrightarrow \Ord_P(\widehat{S}(U^p,E)_{\overline{\rho}}[\fm_x])_{\fB}.
\end{equation}
Since the $R_{p,\fB}$-action on $\tm(U^p,\fB)\otimes_{\cA} \cA/\fp_y$ factors through $R_{p,\fB}/\fp_z$, we see the right hand side of (\ref{ptmap0}) is a quotient of $ \widetilde{P}_{\fB}\widehat{\otimes}_{\widetilde{E}_{\fB}} (\widetilde{E}_{\fB}/\fp_z)_{\tf}^{\oplus r}$ for certain $r$.
Together with  the fact that $\widehat{\pi}_{z_{\widetilde{v},i}}$ is semi-simple for all $v$, $i$, we see the right hand side of (\ref{ptmap1}) contains a direct summand of a copy of $\widehat{\otimes}_{\substack{v\in S_p \\ i=1,\cdots, k_{\widetilde{v}}}} (\widehat{\pi}_{z_{\widetilde{v},i}}\otimes_E \varepsilon^{s_{\widetilde{v},i+1}-1}\circ \dett)$. This concludes the proof.
\end{proof}

\noindent Let $x$ be a benign point of $\Spec \widetilde{\bT}(U^p)_{\overline{\rho},\fB}^{P-\ord}[1/p]$. By Proposition \ref{beni2}, for $v\in S_p$, and $i=1,\cdots, k_{\widetilde{v}}$, $\rho_{x,\widetilde{v},i}$ is crystalline.  We denote by $ \widehat{\pi}(\rho_{x, \widetilde{v}, i})_1$ the universal completion of $\widehat{\pi}(\rho_{x,\widetilde{v},i})^{\lalg}$ (see \cite{BeBr} \cite{Pas09}, noting that by Proposition \ref{beni2}(2), $\widehat{\pi}(\rho_{x,\widetilde{v},i})^{\lalg}$ is isomorphic to the tensor product of an irreducible algebraic representation of $\GL_2(\Q_p)$ with a smooth \emph{irreducible} principal series). By \cite[Lem. 3.4(i)]{BH2}, $ \widehat{\otimes}_{v\in S_p}\big(\widehat{\otimes}_{i=1,\cdots, k_{\widetilde{v}}} \widehat{\pi}(\rho_{x, \widetilde{v}, i})_1\otimes_E \varepsilon^{s_{\widetilde{v},i+1}-1}\circ \dett\big)$ is the universal completion of $ \otimes_{\substack{v\in S_p \\ i=1,\cdots, k_{\widetilde{v}}}} (\widehat{\pi}(\rho_{z_{\widetilde{v},i}})^{\lalg} \otimes_{k(y)} \varepsilon^{s_{\widetilde{v},i+1}-1}\circ \dett)$.
The injection (\ref{Pord: emb0}) (see Corollary \ref{beni3} (2)) induces hence a non-zero morphism of $L_P(\Q_p)$-representations
\begin{equation}\label{injBeni0}
  \widehat{\otimes}_{v\in S_p}\big(\widehat{\otimes}_{i=1,\cdots, k_{\widetilde{v}}} \widehat{\pi}(\rho_{x, \widetilde{v}, i})_1\otimes_E \varepsilon^{s_{\widetilde{v},i+1}-1}\circ \dett\big) \longrightarrow \Ord_P\big(\widehat{S}(U^p, E\big)_{\overline{\rho}}[\fm_x]\big)_{\fB}.
\end{equation}
Let $\Lambda$ be a $\Gal_{F}$-equivariant lattice of $\rho_{x,F}=\rho_{\fm_x,F}$ (where $\fm_x$ denotes the associated maximal ideal of $\widetilde{\bT}(U^p)_{\overline{\rho},\fB}^{P-\ord}[1/p]$). Since $\overline{\rho}_F$ is absolutely irreducible, $\Lambda$ is unique up to scalar, and we have $\Lambda/\varpi_{k(x)}\cong \overline{\rho}_F$ and hence $\Lambda/\varpi_{k(x)} |_{\Gal_{F_{\widetilde{v}}}} \cong \overline{\rho}_{\widetilde{v}}$. The $P_{\widetilde{v}}$-filtration on $\rho_{x,\widetilde{v}}$ induces a $P_{\widetilde{v}}$-filtration on $\Lambda$, and hence induces a $P_{\widetilde{v}}$-filtration on $\overline{\rho}_{\widetilde{v}}$:
 \begin{equation}\label{barFilx}
   \bar{\cF}_{x,\widetilde{v}}:\ 0=\Fil^0_x \overline{\rho}_{\widetilde{v}} \subsetneq \Fil^1_x \overline{\rho}_{\widetilde{v}} \subsetneq \cdots \subsetneq \Fil^{k_{\widetilde{v}}}_x \overline{\rho}_{\widetilde{v}}=\overline{\rho}_{\widetilde{v}}
 \end{equation}
such that the graded piece $\gr^i \bar{\cF}_x$ is a reduction of $\rho_{x,\widetilde{v},i}$.
 \begin{proposition}\label{beni5}
(1) For $v\in S_p$, $i=1, \cdots, k_{\widetilde{v}}$, we have $\overline{\rho}_{x,\widetilde{v},i}^{\sss} \cong \overline{\rho}_{\fB_{\widetilde{v},i}}$,
where $\overline{\rho}_{x,\widetilde{v},i}^{\sss} $ denotes the semi-simplification of one (or any) modulo $\varpi_{k(x)}$ reduction of $\rho_{x,\widetilde{v},i}$.

\noindent (2) We have $(x, \{z_{\widetilde{v},i}\})\in (\Spf \cA)^{\rig}$ where $z_{\widetilde{v},i}$ is the point associated to $\tr \rho_{x,\widetilde{v},i}$.
 \end{proposition}
\begin{proof}
We have $\widehat{\pi}(\rho_{x,\widetilde{v},i})_1\hookrightarrow \widehat{\pi}(\rho_{x,\widetilde{v},i})$. Let $\Theta_{x,\widetilde{v},i}$ be a $\GL_{n_{\widetilde{v},i}}(\Q_p)$-invariant lattice of $\widehat{\pi}(\rho_{x,\widetilde{v},i})$, thus $\Theta_{x,\widetilde{v},i}^d\in \fC_{\GL_{n_{\widetilde{v},i}}(\Q_p)}(\co_E)^{\fB_{x,\widetilde{v},i}}$ where $\fB_{x,\widetilde{v},i}$ is the block corresponding to $\overline{\rho}_{x,\widetilde{v},i}^{\sss}$.
Let $\Theta_{x,\widetilde{v},i,1}:=\Theta_{x,\widetilde{v},i}\cap \widehat{\pi}(\rho_{x,\widetilde{v},i})_1$. Let $\Theta_x$ be the preimage of $ \Ord_P\big(\widehat{S}(U^p, \co_E\big)_{\overline{\rho}}[\fp_x]\big)_{\fB}$ via (\ref{injBeni0}) ($\fp_x=\fm_x\cap \widetilde{\bT}(U^p)_{\overline{\rho}, \fB}^{P-\ord}$). By (\ref{injBeni0}), we deduce $\Hom_{\fC}(\widetilde{P}_{\fB}, \Theta_x^d)\neq 0$. Since $\Theta_x$ and $\Theta_x':=\widehat{\otimes}_{\substack{v\in S_p\\ i=1,\cdots, k_{\widetilde{v}}}} (\Theta_{x,\widetilde{v},i,1} \otimes_{\co_E} \varepsilon^{s_{\widetilde{v},i+1}-1}\circ \dett)$ are commensurable, we see $\Hom_{\fC}(\widetilde{P}_{\fB},(\Theta_x')^d)\neq 0$.
We have a natural projection
\begin{equation}\label{theta1}
  \Hom_{\fC}\Big(\widetilde{P}_{\fB}, \Big(\widehat{\otimes}_{\substack{v\in S_p\\ i=1,\cdots, k_{\widetilde{v}}}} (\Theta_{x,\widetilde{v},i} \otimes_{\co_E} \varepsilon^{s_{\widetilde{v},i+1}-1}\circ \dett)\Big)^d\Big) \twoheadlongrightarrow   \Hom_{\fC}\big(\widetilde{P}_{\fB}, (\Theta_x')^d\big)\ (\neq 0).
\end{equation}
By \cite[Lem. 3.4.9]{Pan2} (see also \cite[Lem. B.8]{GN16}) and the fact (e.g. using (\ref{duala}), and noting that ``$\widehat{\otimes}$" on the left hand side is the $p$-adic completion of the tensor product, while ``$\widehat{\otimes}$" on the right hand side is defined in the following way: $M_1\widehat{\otimes} M_2=\varprojlim (M_1/U_1 \otimes_{\co_E} M_2/U_2)$ for compact $\co_E$-modules $M_i$, where $U_i$ runs through open $\co_E$-submodules of $M_i$):
\begin{equation*}
  \Big(\widehat{\otimes}_{\substack{v\in S_p\\ i=1,\cdots, k_{\widetilde{v}}}} (\Theta_{x,\widetilde{v},i} \otimes_{\co_E} \varepsilon^{s_{\widetilde{v},i+1}-1}\circ \dett)\Big)^d\cong  \widehat{\otimes}_{\substack{v\in S_p\\ i=1,\cdots, k_{\widetilde{v}}}} (\Theta_{x,\widetilde{v},i} \otimes_{\co_E} \varepsilon^{s_{\widetilde{v},i+1}-1}\circ \dett)^d
\end{equation*}
we deduce
$\Hom_{\fC_{\GL_{n_{\widetilde{v},i}}(\Q_p)}}\big(\widetilde{P}_{\fB_{\widetilde{v},i}}, \Theta_{x,\widetilde{v},i}^d\big)\neq 0$,
and hence $\fB_{\widetilde{v},i}=\fB_{x,\widetilde{v},i}$ for all $v\in S_p$, $i=1,\cdots, k_{\widetilde{v}}$. (1) follows. Part (2) of the proposition follows from part (1), (\ref{injBeni0}) and Proposition \ref{pts0}.
\end{proof}
\begin{corollary}
  The set $\cY$ of points $y=(x, z)=(x,\{z_{\widetilde{v},i}\})$ as in Proposition \ref{beni5} (2) is Zariski-dense in $\Spec \cA[1/p]$.
\end{corollary}
\begin{proof}
  For $y\in \cY$, we denote by $\fm_y\subset \cA[1/p]$ (resp. $\fp_y\subset \cA$) the maximal ideal (resp. the prime ideal) associated to $y$. To show $\cap_{y\in \cY} \fm_y=0$, it is sufficient to show $\cap_{y\in \cY} \fp_y=0$.  Let $f\in \cap_{y\in \cY}\fp_y$, which by definition (see (\ref{PordcA})) corresponds to an $\widetilde{E}_{\fB}$-equivariant morphism $f: \tm(U^p,\fB) \ra \tm(U^p,\fB)$ such that the composition $\tm(U^p,\fB) \xrightarrow{f} \tm(U^p,\fB) \twoheadrightarrow \tm(U^p,\fB)/\fp_y$ is equal to zero for all $y\in \cY$. Using the first isomorphism in (\ref{fB0}), the morphism $f$ induces a continuous  $L_P(\Q_p)$-equivariant morphism
  \begin{equation}\label{dualf}
      f: \Ord_P(\widehat{S}(U^p,\co_E)_{\overline{\rho}})_{\fB}  \lra \Ord_P(\widehat{S}(U^p,\co_E)_{\overline{\rho}})_{\fB}.
  \end{equation}
For $y=(x,z)\in \cY$, let $i_x$ be an injection as in (\ref{Pord: emb0}), which induces a non-zero $L_P(\Q_p)$-equivariant morphism as in (\ref{injBeni0}), still denoted by $i_x$. %denote by $\fp_x\subset \widetilde{\bT}(U^p)_{\overline{\rho},\fB}^{P-\ord}$ (resp. $\fm_x \subset \widetilde{\bT}(U^p)_{\overline{\rho},\fB}^{P-\ord}[1/p]$) the prime ideal (resp. the maximal ideal) associated to $x$.
Let $\Theta_x$ be the preimage of $ \Ord_P\big(\widehat{S}(U^p, \co_E\big)_{\overline{\rho}}[\fp_x]\big)_{\fB}$ via $i_x$. We use the notation of the proof of Proposition \ref{beni5}. Since the $R_{p,\fB}$-action on the left hand side of (\ref{theta1}) factors through $R_{p,\fB}/\fp_z$ (where $\fp_z$ is the prime ideal of $R_{p,\fB}$ associated to $z$), the same holds for the $R_{p,\fB}$-action on $\Hom_{\fC}(\widetilde{P}_{\fB}, (\Theta_x')^d)$. Since $\Theta_x$ and $\Theta_x'$  are commensurable, we deduce that the $R_{p,\fB}$-action on $\tm:=\Hom_{\fC}(\widetilde{P}_{\fB}, \Theta_x^d)$ factors through $R_{p,\fB}/\fp_z$ as well. It is also clear that the $\widetilde{\bT}(U^p)^{P-\ord}_{\overline{\rho}, \fB}$-action on $\tm$ (as a quotient of $\tm(U^p,\fp)$ via the projection induced by $i_x$) factors through  $\widetilde{\bT}(U^p)^{P-\ord}_{\overline{\rho}, \fB}/\fp_x$. Hence the $\cA$-action on $\tm$ factors through $\cA/\fp_y$, and $\tm$ is thus a quotient of $\tm(U^p,\fp)/\fp_y$. So the following composition is zero:
\begin{equation*}
  \widetilde{P}_{\fB} \widehat{\otimes}_{\widetilde{E}_{\fB}} \tm(U^p,\fB) \xlongrightarrow{f} \widetilde{P}_{\fB} \widehat{\otimes}_{\widetilde{E}_{\fB}} \tm(U^p,\fB) \twoheadlongrightarrow \widetilde{P}_{\fB} \widehat{\otimes}_{\widetilde{E}_{\fB}} \tm.
\end{equation*}
Applying $\Hom_{\co_E}^{\cts}(-,E)$, we deduce that the composition of $f$ with $i_x$ is zero. Consequently, we deduce that the image of any injection as in (\ref{Pord: emb0}) is annihilated by (\ref{dualf}) (for any benign point $x$). By the same argument as in the proof of \cite[Prop. 7.5 (2)]{BD1}, we deduce then the map (\ref{dualf}) is zero, and hence $f=0$. This concludes the proof.
%, $\fp_y$ (resp. $\fp_x$, resp. $\fp_z$) the prime ideal of $\cA$ (resp.  $\widetilde{\bT}(U^p)_{\overline{\rho},\fB}^{P-\ord}$, resp. $R_{p,\fB}$) associated to $y$ (resp. $x$, resp. $z$).
  %The $\cA$-action on $\tm(U^p,\fB)$ induces a faithful $\cA$-action on $\Ord_P(\widehat{S}(U^p,\co_E)_{\overline{\rho}})_{\fB}$ (via the first isomorphism in (\ref{fB0})).
\end{proof}
\noindent %In general, we call a point $y=(x, \{z_{\widetilde{v},i}\})$ of $ (\Spf \cA)^{\rig}$ \emph{classical} if $x$ is classical.
We have the following classicality criterion.
\begin{proposition}[Classicality]\label{class}
Let $y=(x,\{z_{\widetilde{v},i}\})\in  (\Spf \cA)^{\rig}$. Suppose
\begin{itemize}
  \item for all $v\in S_p$, $i=1, \cdots, k_{\widetilde{v}}$, the pseudo-character associated to $z_{\widetilde{v},i}$ is absolutely irreducible, and let $\rho_{z_{\widetilde{v},i}}: \Gal_{\Q_p} \ra \GL_{n_{\widetilde{v},i}}(k(y))$ be the associated  absolutely irreducible representation (enlarging $k(y)$ if necessary);
  \item $\rho_{z_{\widetilde{v},i}}$ is de Rham of distinct Hodge-Tate weights, and any Hodge-Tate weight of $\rho_{z_{\widetilde{v},i}}$ is strictly bigger than that of $\rho_{z_{\widetilde{v},j}}$ for $j>i$.
\end{itemize}
Then $x$ is classical.
\end{proposition}
\begin{proof}The proposition follows by the same argument of  the proof of \cite[Cor. 7.34]{BD1}. We give a sketch for the convenience of the reader.
By Proposition \ref{pts0}, we have a non-zero map
\begin{equation}\label{injcla}\widehat{\otimes}_{\substack{v\in S_p \\ i=1,\cdots, k_{\widetilde{v}}}} (\widehat{\pi}_{z_{\widetilde{v},i}} \otimes_{k(y)} \varepsilon^{s_{\widetilde{v},i+1}-1}\circ \dett)\lra  \Ord_P\big(\widehat{S}(U^p,E)_{\overline{\rho}}\big)_{\fB}[\fm_x].\end{equation}
By assumption, $\widehat{\pi}_{z_{\widetilde{v},i}}\cong \widehat{\pi}(\rho_{z_{\widetilde{v},i}})$ is absolutely irreducible. Consider the restriction of (\ref{injcla}):
\begin{equation}\label{injcla2}
 \otimes_{\substack{v\in S_p \\ i=1,\cdots, k_{\widetilde{v}}}} (\widehat{\pi}(\rho_{z_{\widetilde{v},i}})^{\lalg} \otimes_{k(y)} \varepsilon^{s_{\widetilde{v},i+1}-1}\circ \dett)\lra  \Ord_P\big(\widehat{S}(U^p,E)_{\overline{\rho}}\big)_{\fB}[\fm_x].
\end{equation}
If (\ref{injcla2}) is zero, there exists $u=\otimes_{\substack{v\in S_p \\ i=1,\cdots, k_{\widetilde{v}}}} u_i\in \widehat{\otimes}_{\substack{v\in S_p \\ i=1,\cdots, k_{\widetilde{v}}}} (\widehat{\pi}_{z_{\widetilde{v},i}} \otimes_{k(y)} \varepsilon^{s_{\widetilde{v},i+1}-1}\circ \dett)$ which is sent to zero via (\ref{injcla}). Since $\widehat{\pi}(\rho_{z_{\widetilde{v},i}})$ is absolutely irreducible, it is easy to see that the left hand side of (\ref{injcla}) can be topologically generated by $u$ under the action of $L_P(\Q_p)$, which contradicts to that (\ref{injcla}) is non-zero. So (\ref{injcla2}) is non-zero (and is actually injective).
Using the assumption on the Hodge-Tate weights and the adjunction property \cite[Prop. 4.21]{BD1}, the proposition follows from (\ref{equ: lalgaut}).
\end{proof}
\noindent By Lemma \ref{densbl} (2) and the same argument of the proof of \cite[Thm. 3.6.1]{Pan2} (one can also use  an infinite fern argument as in \cite[\S~7.1.3]{BD1} for the part on dimension), we have
\begin{theorem}\label{dim0}
  Each irreducible component of $\cA$ is of characteristic zero and of dimension at least
  \begin{equation*}
  1+\sum_{v\in S_p}\big(|\{i|n_{\widetilde{v},i}=1\}|+3|\{i|n_{\widetilde{v},i}=2\}|\big)=   1+ \sum_{v\in S_p} (2n-k_{\widetilde{v}}).
  \end{equation*}
\end{theorem}
%\noindent By the same argument in the proof of \cite[Prop. 7.5 (2)]{BD1} and using Lemma \ref{densbl} (2), it is not difficult to deduce that the points associated to benign points as in Proposition \ref{beni5} (2) are Zariski-dense in $\Spec \cA[1/p]$. %
\noindent We end this section by some discussions on criterions of $\Ord_P(\widehat{S}(U^p,\co_E)_{\overline{\rho}})_{\fB}\neq 0$.  .
\begin{definition}\label{fBord}
  Let $\pi=\pi^{\infty} \otimes \pi_{\infty}=(\otimes'_{v \nmid \infty} \pi_v) \otimes \pi_{\infty}$ be an automorphic representation of $G$ (where $\pi$ is defined over $\overline{\Q_p}$, and $\pi_v$ is defined over $E$ for $v\nmid \infty$)  with $W_p=\otimes_{v\in S_p} W_{v}$ the associated algebraic representation of $G(F^+\otimes_{\Q} \Q_p)\cong \prod_{v\in S_p} \GL_n(\Q_p)$. Let $U^p\subseteq G(\bA_{F^+}^{\infty,p})$ be a sufficiently small compact open subgroup such that $(\pi^{\infty})^{U^p}\neq 0$.  Then $\pi$ is called $\fB$-ordinary if there is an injection $\iota: \otimes_{v\in S_p}(\pi_v\otimes_E W_{v}) \hooklongrightarrow \widehat{S}(U^p,E)_{\overline{\rho}}[\fm_{\pi}]$ such that
  \begin{equation*}
    \overline{\Ord}_P\big(\otimes_{v\in S_p}(\pi_v\otimes_E W_{v})\big)^{\vee}_{\fB}\neq 0
  \end{equation*}
where $\fm_{\pi}$ is the maximal idea of $\widetilde{\bT}(U^p)_{\overline{\rho}}[1/p]$ attached to $\pi$, and  $\overline{\Ord}_P\big(\otimes_{v\in S_p}(\pi_v\otimes_E W_{v})\big) \in  \Mod_{L_P(\Q_p)}^{\lfin}(\co_E)$ denotes the modulo $\varpi_E$ reduction of (where we also use $\iota$ to denote the induced morphism on $\Ord(-)$, cf. \cite[Lem. 4.18]{BD1})
   \begin{equation*}
   \iota\big( \otimes_{v\in S_p} \Ord_{P_{\widetilde{v}}}(\pi_v\otimes_E W_{v})\big) \cap \Ord_P(\widehat{S}(U^p,\co_E))_{\overline{\rho}}.
   \end{equation*}
\end{definition}
\begin{remark}
Note that the definition is independent of the choice of $U^p$ (e.g. by using an isomorphism induced by (\ref{ordUY0}) below via taking inverse limit on $r$). However, it is not clear to the author if the definition depends only on the $p$-factor $\otimes_{v\in S_p} \pi_v$ of $\pi$ (unless under further assumptions, see Remark \ref{locp} below). Hence it is not clear to the author if a $p$-split base-change of a $\fB$-ordinary automorphic representation is still $\fB$-ordinary.
\end{remark}
\begin{lemma}\label{Bord}
  For a block $\fB$ of $\Mod_{L_P(\Q_p)}^{\lfin}(\co_E)$ and a sufficiently small subgroup $U^p\subset G(\bA^{\infty,p}_{F^+})$, the followings are equivalent:
\begin{enumerate}
\item[(1)] there exists a $\fB$-ordinary automorphic representation $\pi=\pi^{\infty}\otimes \pi_{\infty}$ such that $(\pi^{\infty})^{U^p} \neq 0$,
\item[(2)] $\Ord_P(\widehat{S}(U^p,\co_E)_{\overline{\rho}})_{\fB}\neq 0$,
\item[(3)] there exists a benign point of $x$ of  $\Spec \widetilde{\bT}(U^p)_{\overline{\rho}}^{P-\ord}[1/p]$ such that $\overline{\rho}_{x,\widetilde{v},i}^{\sss} \cong \overline{\rho}_{\fB_{\widetilde{v},i}}$ for $v\in S_p$, $i=1,\cdots, k_{\widetilde{v}}$.
\end{enumerate}
\end{lemma}
\begin{proof}
  (1) $\Rightarrow$ (2) is clear. (2) $\Rightarrow$ (3) follows from Proposition \ref{beni5} (1). We show (3) $\Rightarrow$ (1). By the decomposition (\ref{decomp00}), there exists a block $\fB'$ such that the following composition is non-zero (hence injective since the left hand side is irreducible)
\begin{multline}\label{fB'}\otimes_{v\in S_p}\big(\otimes_{i=1,\cdots, k_{\widetilde{v}}} \widehat{\pi}(\rho_{x, \widetilde{v}, i})^{\lalg}\otimes_{k(x)} \varepsilon^{s_{\widetilde{v},i+1}-1}\circ \dett\big) \xlongrightarrow{(\ref{Pord: emb0})} \Ord_P\big(\widehat{S}(U^p, E\big)_{\overline{\rho}}[\fm_x]\big) \\
\twoheadlongrightarrow \Ord_P\big(\widehat{S}(U^p, E\big)_{\overline{\rho}}[\fm_x]\big)_{\fB'}.
\end{multline}
We deduce hence a  similar non-zero map as in (\ref{injBeni0}) with $\fB$ replaced by $\fB'$. Using the same argument as in the proof of Proposition \ref{beni5} (1), we deduce $\overline{\rho}_{x,\widetilde{v},i}^{\sss} \cong \overline{\rho}_{\fB'_{\widetilde{v},i}}$ for all $v\in S_p$ and $i$, and hence $\fB'=\fB$ (by the statement in (3)). By (\ref{fB'}) (with $\fB'=\fB$), it is easy to deduce that any automorphic representation $\pi$ associated to the point $x$ is $\fB$-ordinary and satisfies $\pi^{U^p}\neq 0$. This concludes the proof.
\end{proof}
\begin{remark}\label{locp}
The condition (3) is equivalent to that there exists an automorphic representation $\pi=(\otimes'_{v\nmid \infty} \pi_v)\otimes \pi_{\infty}$ such that \big(where $(-)^{L_P(\Z_p)-\alg}_+$, $(-)^{L_{P_{\widetilde{v}}}(\Z_p)-\alg}_+$ are  defined similarly as in (\ref{LpZp+}), and the first equality easily holds by definition\big)
  \begin{enumerate}
    \item[(a)] $\Ord_P(\otimes_{v\in S_p} (\pi_v \otimes_E W_v))^{L_P(\Z_p)-\alg}_+=\otimes_{v\in S_p} \Ord_{P_{\widetilde{v}}}(\pi_v \otimes_E W_v)^{L_{P_{\widetilde{v}}}(\Z_p)-\alg}_+\neq 0$,
    \item[(b)] the unique Hodge-Tate weights descending $P_{\widetilde{v}}$-filtration on $\rho_{\pi,\widetilde{v}}$ (where the existence follows from Proposition \ref{beni2} (2)) satisfies  $\overline{\gr^i \cF_{\widetilde{v}}}^{\sss}\cong  \overline{\rho}_{\fB_{\widetilde{v},i}}$, where $\rho_{\pi}$ denotes the $\Gal_F$-representation associated to $\pi$,
    \item[(c)] $(\otimes_{v\nmid p, \infty}'\pi_v)^{U_p}\neq 0$.
  \end{enumerate}
Note that the conditions (a) and (b) depend only on the $p$-factor of $\pi$ and $\{\rho_{\pi,\widetilde{v}}\}_{v\in S_p}$.
\end{remark}
%\noindent (2) The existence of $\fB$-ordinary automorphic representations with $U^p$ fixed vectors is equivalence to $\Ord_P(\widehat{S}(U^p,E)_{\overline{\rho}})_{\fB}\neq 0$. Indeed, if $\Ord_P(\widehat{S}(U^p,E)_{\overline{\rho}})_{\fB}\neq 0$, any benign point in $\Spec \widetilde{\bT}(U^p)_{\overline{\rho}, \fB}^{P-\ord}[1/p]$ corresponds to a $\fB$-ordinary automorphic representation. The other direction is also clear.

\subsection{Local-global compatibility}\label{sec: lg}
\noindent We show a local-global compatibility result for our  $\GL_2(\Q_p)$-ordinary families under certain generic assumptions.
\begin{definition}\label{Bgene}We call  $\overline{\rho}$ is $\fB$-generic if for all $v\in S_p$,
\begin{enumerate} \item[(1)]$\overline{\rho}_{\widetilde{v}}$ admits a unique filtration $\cF_{\fB_{\widetilde{v}}}$ such that $(\gr^i \cF_{\fB_{\widetilde{v}}})^{\sss} \cong \overline{\rho}_{\fB_{\widetilde{v},i}}$ for $i=1,\cdots, k_{\widetilde{v}}$;
\item[(2)] the filtration $\cF_{\fB_{\widetilde{v}}}$ satisfies Hypothesis \ref{hypo: gene1}. \end{enumerate}
\end{definition}
\noindent We assume $\overline{\rho}$ is $\fB$-generic. By Proposition \ref{beni5} (1), we see $\bar{\cF}_{x,\widetilde{v}}=\cF_{\fB_{\widetilde{v}}}$ (cf. (\ref{barFilx})) for all benign points $x$. Consider the deformation problem (where $R_{\overline{\rho}_{\widetilde{v}}, \cF_{\fB_{\widetilde{v}}}}^{P_{\widetilde{v}}-\ord,\bar{\square}}$ denotes the reduced quotient of $R_{\overline{\rho}_{\widetilde{v}}, \cF_{\fB_{\widetilde{v}}}}^{P_{\widetilde{v}}-\ord,\square}$, cf. \S~\ref{Porddef})
\begin{equation}\label{def2g}
  \big(F/F^+, S, \widetilde{S}, \co_E, \overline{\rho}, \varepsilon^{1-n} \delta_{F/F^+}, \{R_{\overline{\rho}_{\widetilde{v}}, \cF_{\fB_{\widetilde{v}}}}^{P_{\widetilde{v}}-\ord,\bar{\square}}\}_{v\in S_p} \cup \{R_{\overline{\rho}_{\widetilde{v}}}^{\bar{\square}}\}_{v\in S\setminus S_p}\big).
\end{equation}
By \cite[Prop. 3.4]{Thor}, the corresponding deformation functor is represented by a complete local noetherian $\co_E$-algebra, denoted by $R_{\overline{\rho},\cS,\fB}^{P-\ord}$, which is a quotient of $R_{\overline{\rho}, \cS}$.  By Proposition \ref{beni3} (2) and Proposition \ref{beni5}(1), for any benign point $x\in \Spec \widetilde{\bT}(U^p)_{\overline{\rho},\fB}^{P-\ord}[1/p]$, the action of $R_{\overline{\rho},\cS}$ on $\Ord_P(\widehat{S}(U^p,E)_{\overline{\rho}})_{\fB}[\fm_x]$ factors through $R_{\overline{\rho},\cS,\fB}^{P-\ord}$. Since $(\Ord_P(\widehat{S}(U^p,E)_{\overline{\rho}})_{\fB})^{L_P(\Z_p)-\alg}_+$ is dense in $\Ord_P(\widehat{S}(U^p,E)_{\overline{\rho}})_{\fB}$ (Proposition \ref{densalg}), we deduce by the same argument as in \cite[Thm. 6.12]{BD1}:
\begin{proposition}\label{Pordlg}
(1) The action of $R_{\overline{\rho},\cS}$ on $\Ord_P(\widehat{S}(U^p,E)_{\overline{\rho}})_{\fB}$ factors through $R_{\overline{\rho},\cS,\fB}^{P-\ord}$.

\noindent (2) Let $x$ be a closed point of $\Spec \widetilde{\bT}(U^p)_{\overline{\rho}}[1/p]$ such that $\Ord_P(\widehat{S}(U^p,E)_{\overline{\rho}}[\fm_x])_{\fB}\neq 0$, then for $v\in S_p$, $\rho_{x,\widetilde{v}}$ admits a $P_{\widetilde{v}}$-filtration with the induced $P_{\widetilde{v}}$-filtration on $\overline{\rho}_{\widetilde{v}}$ equal to $\cF_{\fB_{\widetilde{v}}}$.% inducing a $P_{\widetilde{v}}$-filtration of $\overline{\rho}_{x,\widetilde{v}}$ which satisfies $\gr^i \overline{\rho}_{x,\widetilde{v}}\cong \overline{\rho}_{\widetilde{v},i}$.
\end{proposition}
\noindent Let $A$ be an artinian local $\co_E$-algebra $A$, and $\rho_A\in R_{\overline{\rho},\cS,\fB}^{P-\ord}(A)$. For $v\in S_p$, by definition $\rho_{A,\widetilde{v}}:=\pr_2\circ \rho_A|_{\Gal_{F_{\widetilde{v}}}}$ (with $\pr_2: \GL_1\times \GL_n \twoheadrightarrow \GL_n$) admits a $P_{\widetilde{v}}$-filtration $\cF_{\widetilde{v}}$ such that $\tr(\gr^i \cF_{\widetilde{v}})$ is a deformation of $\tr(\gr^i \cF_{\fB_{\widetilde{v}}})$ over $A$ for $i=1,\cdots,k_{\widetilde{v}}$. We obtain thus a natural morphism
\begin{equation}\label{wtR}
  R_{p,\fB} \lra R_{\overline{\rho},\cS,\fB}^{P-\ord}.
\end{equation}
The composition
\begin{equation}\label{RpB2}
   R_{p,\fB} \lra R_{\overline{\rho},\cS,\fB}^{P-\ord}\lra \widetilde{\bT}(U^p)_{\overline{\rho}, \fB}^{P-\ord}.
\end{equation}
equips with $\tm(U^p,\fB)$ another $R_{p,\fB}$-action. The following theorem follows by the same argument as in the proof of \cite[Thm. 3.5.5]{Pan2} \cite[Prop. 5.5]{Pas18}. Indeed,  one can easily obtain an analogue of \cite[Lem. 3.5.6]{Pan2} with $\fp$ of \emph{loc. cit.} replaced by the benign points, using  Lemma \ref{densbl} (2), Corollary \ref{beni3} (2). Note also that an analogue of \cite[Lem. 3.5.7]{Pan2} is already contained in Proposition \ref{beni4} (1) (whose proof builds upon the local-global compatibility result  in classical local Langlands correspondence, see the proof of \cite[Prop. 7.6 (2)]{BD1}).
\begin{theorem}[local-global compatibility]\label{lg}
The two actions of $R_{p,\fB}$ on $\tm(U^p, \fB)$ coincide, i.e. the composition
\begin{equation*}
  R_{p,\fB}\lra \widetilde{\bT}(U^p)_{\overline{\rho}, \fB}^{P-\ord} \hooklongrightarrow \End_{R_{p,\fB}}(\tm(U^p,\fB))
\end{equation*}
coincides with the structure map.
\end{theorem}
\begin{remark}
One might prove (e.g. by putting more hypothesis on $\overline{\rho}_p$) a stronger local-global compatibility result (e.g. by replacing the universal pseudo-deformation rings by the universal deformation rings of certain Galois representations) as in \cite[\S~7.1.4]{BD1} (where the formulation is more close to \cite{Em4}), but we don't need such result in this note.
\end{remark}
\noindent By definition, Theorem \ref{lg} and Theorem \ref{dim0}, we have
\begin{corollary}\label{dim1}
  We have $\widetilde{\bT}(U^p)_{\overline{\rho},\fB}^{P-\ord}\xrightarrow{\sim} \cA$, and each irreducible component of $\widetilde{\bT}(U^p)_{\overline{\rho}, \fB}^{P-\ord}$ is of characteristic zero and of dimension at least $1+ \sum_{v\in S_p} (2n-k_{\widetilde{v}})$.
\end{corollary}
\noindent We end this section by a $\fB$-generic criterion, which is easier to check in certain cases.
\begin{lemma}\label{bgene1}For $v\in S_p$, let $\cF_{\widetilde{v}}$ be a $P_{\widetilde{v}}$-filtration on $\overline{\rho}_{\widetilde{v}}$ satisfying Hypothesis \ref{hypo: gene1} and that $(\gr^i \cF_{\widetilde{v}})^{\sss} \cong \overline{\rho}_{\fB_{\widetilde{v},i}}$ for all $i$. Suppose moreover $\fB_{\widetilde{v},i}\neq \fB_{\widetilde{v},j}$ for all $i\neq j$ and $v\in S_p$. Then $\overline{\rho}$ is $\fB$-generic.
\end{lemma}
\begin{proof}
Suppose we have another filtration $\cF_{\widetilde{v}}'$ satisfying the same property.
\\

\noindent (1) We show first $\gr^1 \cF_{\widetilde{v}}'\cong \gr^1 \cF_{\widetilde{v}}$. If $\overline{\rho}_{\fB_{\widetilde{v},1}}$ is irreducible, then it is clear. Suppose $\overline{\rho}_{\fB_{\widetilde{v},1}}\cong \chi_1\oplus \chi_2$, and we prove the statement case by case.

\noindent (a) If $\gr^1 \cF_{\widetilde{v}}'\cong \chi_1\oplus \chi_2$, and if $\gr^1 \cF_{\widetilde{v}}$ is not isomorphic to $\gr^1 \cF_{\widetilde{v}}'$. Without loss of generality, we assume $\gr^1 \cF_{\widetilde{v}}$ is a non-split extension of $\chi_2$ by $\chi_1$. We see
\begin{equation}\label{homBge}\Hom_{\Gal_{F_{\widetilde{v}}}}(\chi_2, \overline{\rho}_{\widetilde{v}}) \longrightarrow \Hom_{\Gal_{F_{\widetilde{v}}}}(\chi_2, \overline{\rho}_{\widetilde{v}}/\gr^1 \cF_{\widetilde{v}}) \hooklongrightarrow \Hom_{\Gal_{F_{\widetilde{v}}}}(\gr^1 \cF_{\widetilde{v}}, ,\overline{\rho}_{\widetilde{v}}/\gr^1 \cF_{\widetilde{v}}),
\end{equation}
where the first map is non-zero by assumption (using $\gr^1 \cF_{\widetilde{v}}'\hookrightarrow \overline{\rho}_{\widetilde{v}}$).We deduce thus the right hand side of (\ref{homBge}) is non-zero,  which leads to a contradiction with Hypothesis \ref{hypo: gene1}.

\noindent (b) Suppose $\gr^1 \cF_{\widetilde{v}}'$ is a non-split extension of $\chi_2$ by $\chi_1$, and suppose $\gr^1 \cF_{\widetilde{v}}$ is an extension of $\chi_1$ by $\chi_2$ (e.g.  if $\gr^1 \cF_{\widetilde{v}}$ is a direct sum of $\chi_1$ and $\chi_2$). If $\chi_1=\chi_2=\chi$, the sum $\gr^1 \cF_{\widetilde{v}}'+\gr^1 \cF_{\widetilde{v}}$ will be a subrepresentation of $\overline{\rho}_{\widetilde{v}}$ of dimension bigger than $3$. We deduce that $\Hom_{\Gal_{\Q_p}}(\chi, \gr^1 \cF_{\widetilde{v}}'/(\gr^1 \cF_{\widetilde{v}}' \cap \gr^1 \cF_{\widetilde{v}}))\neq 0$, and hence $\Hom_{\Gal_{\Q_p}}(\gr^1 \cF_{\widetilde{v}}, \overline{\rho}_{\widetilde{v}}/\gr^1 \cF_{\widetilde{v}})\neq 0$, a contradiction with Hypothesis \ref{hypo: gene1}. If $\chi_1\neq \chi_2$, then we have
\begin{equation*}
  \Hom_{\Gal_{F_{\widetilde{v}}}}(\chi_2, \gr^1 \cF_{\widetilde{v}}) \hooklongrightarrow \Hom_{\Gal_{F_{\widetilde{v}}}}(\chi_2, \rho_{\widetilde{v}}) \lra \Hom_{\Gal_{F_{\widetilde{v}}}}(\chi_2, \rho_{\widetilde{v}}/\gr^1 \cF'_{\widetilde{v}})
\end{equation*}
 with the second map non-zero. There exists thus $i>1$ such that $\chi_2\hookrightarrow \gr^i \cF'_{\widetilde{v}}$, thus $$\Hom_{\Gal_{F_{\widetilde{v}}}}(\gr^1 \cF'_{\widetilde{v}}, \gr^i \cF'_{\widetilde{v}})\neq 0,$$ which (again) contradicts with Hypothesis \ref{hypo: gene1}.

\noindent (c) Consider the last case where both of $\gr^1 \cF_{\widetilde{v}}'$ and $\gr^1 \cF_{\widetilde{v}}$ are isomorphic to a non-split extension of $\chi_2$ by $\chi_1$, and  $\gr^1 \cF_{\widetilde{v}}'\ncong  \gr^1 \cF_{\widetilde{v}}$. In this case, $\chi_1\chi_2^{-1}=1$ or $\omega$.

\noindent If $\chi_1=\chi_2=\chi$. It is easy to see $\gr^1 \cF_{\widetilde{v}}+\gr^1 \cF_{\widetilde{v}}'$ is isomorphic to a successive extension of $\chi$ of dimension bigger three. Using the same argument in (b) (in the case $\chi_1=\chi_2=\chi$), we easily deduce a contradiction.

\noindent Suppose $\chi_1=\chi_2 \omega$. Denote by $\iota: \chi_1\hookrightarrow \gr^1 \cF_{\widetilde{v}}\hookrightarrow \overline{\rho}_{\widetilde{v}}$, and $\iota': \chi_1 \hookrightarrow \gr^1 \cF_{\widetilde{v}}' \hookrightarrow \overline{\rho}_{\widetilde{v}}$.

\noindent If $\Ima(\iota)=\Ima(\iota')$, since $\gr^1 \cF_{\widetilde{v}}'\ncong  \gr^1 \cF_{\widetilde{v}}$, we deduce the composition
\begin{equation*}
  \chi_2 \cong \gr^1 \cF_{\widetilde{v}}/\Ima(\iota) \hookrightarrow \overline{\rho}_{\widetilde{v}}/\Ima(\iota') \lra \overline{\rho}_{\widetilde{v}}/\gr^1 \cF'_{\widetilde{v}}
\end{equation*}
is non-zero. Hence there exists $i>1$ such that $\chi_2\hookrightarrow \gr^i \cF'_{\widetilde{v}}$, thus $\Hom_{\Gal_{F_{\widetilde{v}}}}(\gr^1 \cF'_{\widetilde{v}}, \gr^i \cF'_{\widetilde{v}})\neq 0$,  a contradiction with Hypothesis \ref{hypo: gene1}.

\noindent If $\Ima(\iota')\neq \Ima(\iota)$. Let $i>1$ be maximal  such that $\chi_1\xrightarrow{\iota'} \overline{\rho}_{\widetilde{v}}\ra \overline{\rho}_{\widetilde{v}}/\Fil^{i-1}_{\cF_{\widetilde{v}}} \overline{\rho}_{\widetilde{v}}$ is non-zero. Thus $\chi_1$ is an irreducible sub of $\gr^i \cF_{\widetilde{v}}$. Since $\cF_{\widetilde{v}}$ satisfies Hypothesis \ref{hypo: gene1}, we deduce $\dim_{k_E} \gr^i \cF_{\widetilde{v}}=2$ and $\soc_{\Gal_{F_{\widetilde{v}}}} \gr^i \cF_{\widetilde{v}}\cong \chi_1$ (otherwise, $\Hom_{\Gal_{F_{\widetilde{v}}}}(\gr^i \cF_{\widetilde{v}},\gr^1 \cF_{\widetilde{v}})\neq 0$).   Consider
\begin{equation}\label{compBge}\chi_2\cong \gr^1\cF_{\widetilde{v}}'/\chi_1 \hooklongrightarrow \overline{\rho}_{\widetilde{v}}/\Ima(\iota') \twoheadlongrightarrow \overline{\rho}_{\widetilde{v}}/(\Ima(\iota)\oplus \Ima(\iota')).
 \end{equation}If the image of (\ref{compBge}) lies in $\gr^1 \cF_{\widetilde{v}}/\chi_1$, then it is not difficult to see the first injection in (\ref{compBge}) has image in $\gr^1 \cF_{\widetilde{v}}/\Ima(\iota')\cong \gr^1\cF_{\widetilde{v}}$. But if so, $\gr^1\cF_{\widetilde{v}}$ is split, a contradiction. So there exists $j>1$ such that
\begin{equation*}\chi_2 \cong \gr^1\cF_{\widetilde{v}}'/\chi_1  \lra \overline{\rho}_{\widetilde{v}}/\big(\Ima(\iota')+\Fil_{\cF_{\widetilde{v}}}^{j-1} \overline{\rho}_{\widetilde{v}}\big)\end{equation*} is non-zero. We let $j$ be the maximal integer satisfying this property. Then $\chi_2$ is an irreducible constituent of $\gr^j \cF_{\widetilde{v}}$. Using (again) Hypothesis \ref{hypo: gene1}, we easily deduce $\dim_{k_E} \gr^j \cF_{\widetilde{v}}=2$ and $\cosoc_{\Gal_{F_{\widetilde{v}}}} \gr^j \cF_{\widetilde{v}}\cong \chi_2$ (otherwise, $\Hom_{\Gal_{F_{\widetilde{v}}}}(\gr^j \cF_{\widetilde{v}},\gr^1 \cF_{\widetilde{v}})\neq 0$). If $j<i$, let $V_j$ be the kernel of $\Fil^{j-1}_{\cF_{\widetilde{v}}}\twoheadrightarrow \gr^j \cF_{\widetilde{v}} \twoheadrightarrow \chi_2$, then we see
\begin{equation*}
  \gr^1 \cF_{\widetilde{v}}' \lra \overline{\rho}_{\widetilde{v}}/V_j
\end{equation*}
is injective, and the injection $\chi_2\hookrightarrow \overline{\rho}_{\widetilde{v}}/V_j$ induces a section of $\gr^1 \cF_{\widetilde{v}}'\twoheadrightarrow \chi_2$, a contradiction. If $j>i$, since $\chi_2 \hookrightarrow \overline{\rho}_{\widetilde{\chi}}/\Ima(\iota')\twoheadrightarrow \overline{\rho}_{\widetilde{\chi}}/\Fil^{j-1}_{\cF_{\widetilde{v}}}$, we deduce $\gr^j \cF_{\widetilde{v}}$ splits, a contradiction. Finally if $j=i$, then $\fB_{v,i}=\fB_{v,1}$ a contradiction.
\\

\noindent (2) We show $\Fil^1_{\cF_{\widetilde{v}}}=\Fil^1_{\cF_{\widetilde{v}}'}$ (as subrepresentation of $\overline{\rho}_{\widetilde{v}}$), from which the lemma follows by an easy induction argument.  Suppose $\Fil^1_{\cF_{\widetilde{v}}}\neq \Fil^1_{\cF_{\widetilde{v}}'}$, consider $V:=(\Fil^1_{\cF_{\widetilde{v}}} + \Fil^1_{\cF_{\widetilde{v}}'})/\Fil^1_{\cF_{\widetilde{v}}}$, which is thus a non-zero subrepresentation of $\overline{\rho}_{\widetilde{v}}/\Fil^1_{\cF_{\widetilde{v}}}$, and whose irreducible constituents appear in $\gr^1 \cF_{\widetilde{v}}$. If $V$ is irreducible, it has to be a sub of $\gr^j \cF_{\widetilde{v}}$ for certain $j>1$. However, we have
 \begin{equation*}
   \gr^1 \cF_{\widetilde{v}} \cong \Fil^1_{\cF_{\widetilde{v}}'} \twoheadrightarrow (\Fil^1_{\cF_{\widetilde{v}}} + \Fil^1_{\cF_{\widetilde{v}}'})/\Fil^1_{\cF_{\widetilde{v}}}=V.
 \end{equation*}
 Thus $\Hom_{\Gal_{F_{\widetilde{v}}}}(\gr^1 \cF_{\widetilde{v}}, \gr^j \cF_{\widetilde{v}})\neq 0$, a contradiction with Hypothesis \ref{hypo: gene1}.
 If $V$ is isomorphic to a direct sum of characters, we similarly obtain a contradiction. Suppose there exist characters $\chi_1$, $\chi_2$ such that $V$ is a non-split extension of $\chi_2$ by $\chi_1$. In this case we see  $V\cong \gr^1 \cF_{\widetilde{v}}$. Using the same argument as in (c) for the case $\Ima(\iota)\neq \Ima(\iota')$, we can also obtain a contradiction. This concludes the proof.
\end{proof}
\section{Patching and automorphy lifting}
\noindent We apply the Taylor-Wiles patching argument to our $\GL_2(\Q_p)$-ordinary families, and prove our main result on automorphy lifting.
\subsection{Varying levels outside $p$}
\noindent Let $\fB$ be a block of $\Mod_{L_P(\Q_p)}^{\lfin}(\co_E)$ such that $\Ord_P(\widehat{S}(U^p,\co_E)_{\overline{\rho}})_{\fB}\neq 0$.
\begin{lemma}\label{flat0}
  Let $H$ be a finite group, $r\geq 1$, $M$ be a finitely generated flat $\co_E/\varpi_E^r[H]$-module. Then $M^{\vee}$ is also a finitely generated flat $\co_E/\varpi_E^r[H]$-modules.
\end{lemma}
\begin{proof}
It is easy to check that if  $M$ is a finite free $\co_E/\varpi_E^r [H]$-module, then $M^{\vee}\cong M$, and hence is also a finite free $\co_E/\varpi_E^r[H]$-module. The lemma follows.
\end{proof}
\noindent Let $Y^p=\prod_{v\nmid p} Y_v$ be a compact open normal subgroup of $U^p$ (which is thus also sufficiently small). For any compact open subgroup $U_p$ of $\prod_{v\in S_p} \GL_n(\Z_p)$, by \cite[Lem. 3.3.1]{CHT}, $S(Y^p U_p, \co_E/\varpi_E^r)$ is a finite free $\co_E/\varpi_E^r[U^p/Y^p]$-module.
\begin{lemma}\label{flat1}Let  $i\in \Z_{\geq 0}$, $r\in \Z_{\geq 1}$.

\noindent (1)  $\Ord_P(S(Y^p,\co_E/\varpi_E^r)_{\overline{\rho}})_{\fB}^{L_i}$ and $ \big(\Ord_P(S(Y^p,\co_E/\varpi_E^r)_{\overline{\rho}})_{\fB}^{L_i}\big)^{\vee}$ are finite flat $\co_E/\varpi_E^r[U^p/Y^p]$-modules.

\noindent (2) We have a natural isomorphism
\begin{equation*}
  \big(\Ord_P(S(Y^p,\co_E/\varpi_E^r)_{\overline{\rho}})_{\fB}^{L_i}\big)^{\vee}_{U^p/Y^p}\xlongrightarrow{\sim} \big(\Ord_P(S(U^p,\co_E/\varpi_E^r)_{\overline{\rho}})_{\fB}^{L_i}\big)^{\vee}.
\end{equation*}
\end{lemma}
\begin{proof}
Recall (cf. (\ref{isoOrd}))
\begin{equation*}\Ord_P(S(Y^p,\co_E/\varpi_E^r)_{\overline{\rho}})^{L_i}\cong S(Y^p K_{i,i},\co_E/\varpi_E^r)_{\overline{\rho},\ord} \cong \oplus_{\fm \text{ ordinary}} S(Y^p K_{i,i},\co_E/\varpi_E^r)_{\overline{\rho}, \fm},
\end{equation*}
where $\fm$ runs through the ordinary maximal ideals of  the $\co_E/\varpi_E^r$-subalgebra $B(Y^p)$ of
$$\End_{\co_E/\varpi_E^r}\big( S(Y^p K_{i,i},\co_E/\varpi_E^r)_{\overline{\rho}}\big)$$ generated by the operators in $Z_{L_P^+}$. The same statement holds with $Y^p$ replaced by $U^p$. Note also that the natural inclusion
\begin{equation*}
  S(U^pK_{i,i},\co_E/\varpi_E^r)_{\overline{\rho}} \cong S(Y^pK_{i,i}, \co_E/\varpi_E^r)_{\overline{\rho}}^{U^p/Y^p}\hooklongrightarrow S(Y^p K_{i,i}, \co_E/\varpi_E^r)_{\overline{\rho}}
\end{equation*}
induces a projection $\pr: B(Y^p) \twoheadrightarrow B(U^p)$. By definition, it is straightforward to see that for a maximal ideal $\fm$ of $B(U^p)$, $\fm$ is ordinary if and only if $\pr^{-1}(\fm)$ is ordinary.
 Since the action of $B(Y^p)$ and $U^p/Y^p$ commute, we deduce that $\Ord_P(S(Y^p,\co_E/\varpi_E^r)_{\overline{\rho}})^{L_i}$ is an $\co_E/\varpi_E^r[U^p/Y^p]$-equivariant direct summand of $ S(Y^p K_{i,i},\co_E/\varpi_E^r)_{\overline{\rho}}$, and hence is a finite flat $\co_E/\varpi_E^r[U^p/Y^p]$-module.
\\

\noindent Using the isomorphism (which follows by the same argument as for (\ref{fB0}))
\begin{equation*}
  \Ord_P(S(Y^p, \co_E/\varpi_E^r)_{\overline{\rho}})^{\vee}_{\fB}\cong \widetilde{P}_{\fB} \widehat{\otimes}_{\widetilde{E}_{\fB}} \Hom_{\fC}\big(\widetilde{P}_{\fB},  \Ord_P(S(Y^p, \co_E/\varpi_E^r)_{\overline{\rho}})^{\vee}\big),
\end{equation*}
we see that $\Ord_P(S(Y^p, \co_E/\varpi_E^r)_{\overline{\rho}})^{\vee}_{\fB}$ inherits a natural $U^p/Y^p$-action from $\Ord_P(S(Y^p, \co_E/\varpi_E^r)_{\overline{\rho}})^{\vee}$ satisfying that  the natural injection (given by the evaluation map)
\begin{multline*}
  \Ord_P(S(Y^p, \co_E/\varpi_E^r)_{\overline{\rho}})^{\vee}_{\fB}\cong \widetilde{P}_{\fB} \widehat{\otimes}_{\widetilde{E}_{\fB}} \Hom_{\fC}\big(\widetilde{P}_{\fB},  \Ord_P(S(Y^p, \co_E/\varpi_E^r)_{\overline{\rho}})^{\vee}\big)\\ \hooklongrightarrow \Ord_P(S(Y^p, \co_E/\varpi_E^r)_{\overline{\rho}})^{\vee}
\end{multline*}
is $U^p/Y^p$-equivariant. Hence, the decomposition
\begin{equation*}
 \Ord_P(S(Y^p, \co_E/\varpi_E^r)_{\overline{\rho}}) \cong \oplus_{\fB}  \Ord_P(S(Y^p, \co_E/\varpi_E^r)_{\overline{\rho}})_{\fB}
\end{equation*}
is $U^p/Y^p$-equivariant. This implies that $\Ord_P(S(Y^p, \co_E/\varpi_E^r)_{\overline{\rho}})_{\fB}^{L_i}$ is a $U^p/Y^p$-equivariant direct summand of $\Ord_P(S(Y^p,\co_E/\varpi_E^r)_{\overline{\rho}})^{L_i}$, hence is also a finite flat $\co_E/\varpi_E^r[U^p/Y^p]$-module. Together with Lemma \ref{flat0}, (1) follows.
\\

\noindent For each $i\geq 1$, we have a natural isomorphism
\begin{equation}\label{Up0}
  S(U^p K_{i,i}, \co_E/\varpi_E^r)_{\overline{\rho}} \xlongrightarrow{\sim} S(Y^p K_{i,i}, \co_E/\varpi_E^r)_{\overline{\rho}}^{U^p/Y^p}.
\end{equation}
By the discussion in the first paragraph, we deduce that (\ref{Up0}) induces an isomorphism
\begin{equation}\label{UpOrd}
  \Ord_P(S(U^p, \co_E/\varpi_E^r)_{\overline{\rho}}^{L_i}) \xlongrightarrow{\sim} \Ord_P(S(Y^p, \co_E/\varpi_E^r)_{\overline{\rho}}^{L_i})^{U^p/V^p}.
\end{equation}
Indeed, for all maximal ideals of $B(U^p)$, we have
\begin{equation}\label{Up01}
   S(U^p K_{i,i}, \co_E/\varpi_E^r)_{\overline{\rho},\fm} \hooklongrightarrow \big(S(Y^p K_{i,i}, \co_E/\varpi_E^r)_{\overline{\rho}, \pr^{-1}(\fm)}\big)^{U^p/Y^p},
\end{equation}
hence
\begin{equation*}
  \oplus_{\fm}  S(U^p K_{i,i}, \co_E/\varpi_E^r)_{\overline{\rho},\fm} \hookrightarrow \oplus_{\fm} \big(S(Y^p K_{i,i}, \co_E/\varpi_E^r)_{\overline{\rho}, \pr^{-1}(\fm)}\big)^{U^p/Y^p}
  \hookrightarrow S(Y^p K_{i,i}, \co_E/\varpi_E^r)_{\overline{\rho}}^{U^p/Y^p}.
\end{equation*}
By (\ref{Up0}), the above composition is surjective, hence each map in (\ref{Up01}) has to be bijective. The isomorphism (\ref{UpOrd}) follows. Taking direct limit on $i$, (\ref{UpOrd}) induces an isomorphism
\begin{equation}\label{ordUY0}
   \Ord_P(S(U^p, \co_E/\varpi_E^r)_{\overline{\rho}}) \xlongrightarrow{\sim} \Ord_P(S(Y^p, \co_E/\varpi_E^r)_{\overline{\rho}})^{U^p/V^p}.
\end{equation}
Which together with the obvious injections (with the composition bijective)
\begin{multline*}
 \oplus_{\fB} \Ord_P(S(U^p, \co_E/\varpi_E^r)_{\overline{\rho}})_{\fB} \hooklongrightarrow \oplus_{\fB}\big(\Ord_P(S(Y^p, \co_E/\varpi_E^r)_{\overline{\rho}})_{\fB}\big)^{U^p/V^p}\\ \hooklongrightarrow \Ord_P(S(Y^p, \co_E/\varpi_E^r)_{\overline{\rho}})^{U^p/V^p}
\end{multline*}
 imply $\Ord_P(S(U^p, \co_E/\varpi_E^r)_{\overline{\rho}})_{\fB} \cong \big(\Ord_P(S(Y^p, \co_E/\varpi_E^r)_{\overline{\rho}})_{\fB}\big)^{U^p/V^p}$ for all $\fB$. Hence
\begin{equation}\label{isomInv}
  \Ord_P(S(U^p, \co_E/\varpi_E^r)_{\overline{\rho}})_{\fB}^{L_i} \xlongrightarrow{\sim} \big(\Ord_P(S(Y^p,\co_E/\varpi_E^r)_{\overline{\rho}})_{\fB}^{L_i}\big)^{U^p/Y^p}.
\end{equation}
By taking Pontryagain dual, (2) follows.
\end{proof}

\subsection{Auxiliary primes}
\noindent Choose $q\geq [F^+:\Q]\frac{n(n-1)}{2}$. By \cite[Prop. 4.4]{Thor} (see also the proof of \cite[Thm. 6.8]{Thor}),  for all $N\geq 1$, there exists a set $Q_N$ (resp.  $\widetilde{Q}_N$) of primes of $F^+$ (resp. of $F$) such that
\begin{itemize}
  \item $|Q_N|=q$, $Q_N$ is disjoint from $S$, and any primes in $Q_N$ is split in $F$;
  \item $\widetilde{Q}_N=\{\widetilde{v}|v\ |\ v \in Q_N\}$;
  \item ${\rm Nm}(v)\equiv 1\pmod{p^N}$ for $v\in Q_N$;
  \item for $\widetilde{v}\in \widetilde{Q}_N$, $\overline{\rho}_{\widetilde{v}} \cong \overline{s}_{\widetilde{v}} \oplus \overline{\psi}_{\widetilde{v}}$, where $\overline{\psi}_{\widetilde{v}}$ is the (generalized) eigenspace of Frobenius of an eigenvalue $\alpha_{\widetilde{v}}$ on which Frobenius acts semisimply.
\end{itemize}
For $\widetilde{v}\in \widetilde{Q}_N$, denote by $D_{\widetilde{v}}$  the local deformation problem such that for $A\in \Art(\co_E)$, $D_{\widetilde{v}}(A)$ consists of all lifts which are $1+M_n(\fm_A)$-conjugate to one of the form $s_{\widetilde{v}} \oplus \psi_{\widetilde{v}}$ where $s_{\widetilde{v}}$ is an unramified lift of $\overline{s}_{\widetilde{v}}$  and where $\psi_{\widetilde{v}}$ is a (possibly ramified) lift of $\overline{\psi}_{\widetilde{v}}$ satisfying that the image of inertial under $\psi_{\widetilde{v}}$ is contained in the set of scalar matrices. The deformation problem $D_{\widetilde{v}}$ is pro-represented by a quotient of $R_{\overline{\rho}_{\widetilde{v}}}^{\bar{\square}}$, denoted dy $R_{\overline{\rho}_{\widetilde{v}}}^{\overline{\psi}_{\widetilde{v}}}$.
Let $\psi_{\widetilde{v}}$ be as above, then the restriction of  $\psi_{\widetilde{v}}$ to the inertial subgroup $I_{\widetilde{v}}$ of $\Gal_{F_{\widetilde{v}}}$ gives a character $\chi_{\widetilde{v}}: I_{\widetilde{v}} \ra 1+\fm_A$ (noting $\chi_{\widetilde{v}}\equiv \overline{\psi}_{\widetilde{v}}|_{I_{\widetilde{v}}} =1 \pmod{\fm_A}$). We can prove (e.g. using $\chi_{\widetilde{v}}^{|\F_{\widetilde{v}}|}=\chi_{\widetilde{v}}$ since $\Frob_{\widetilde{v}}^{-1}\sigma \Frob_{\widetilde{v}}=\sigma^{|\F_{\widetilde{v}}|}$ for all $\sigma\in I_{\widetilde{v}}/P_{\widetilde{v}}$, and using the fact that $1+\fm_A$ is a $p$-group)
that $\chi_{\widetilde{v}}$ factors through
\begin{equation*}
  I_{\widetilde{v}}/P_{\widetilde{v}} \twoheadlongrightarrow \F_{\widetilde{v}}^{\times} \twoheadlongrightarrow  \F_{\widetilde{v}}(p)\cong \Z/p^N \Z
\end{equation*}
where $ \F_{\widetilde{v}}(p)$ denotes the maximal $p$-power order quotient of $\F_{\widetilde{v}}^{\times}$, and $\F_{\widetilde{v}}$ denotes the residue field of $F$ at $\widetilde{v}$. We deduce thus a natural morphism
\begin{equation}\label{qn0}
\chi_{\widetilde{v}}^{\univ}:  \Z/p^N \Z \lra (R_{\overline{\rho}_{\widetilde{v}}}^{\psi_{\widetilde{v}}})^{\times}.
\end{equation}

\noindent Denote by $\cS_{Q_N}$ the following deformation problem
\begin{equation*}
\big(F/F^+, S \cup Q_N, \widetilde{S} \cup \widetilde{Q}_N, \co_E, \overline{\rho}, \varepsilon^{1-n}\delta_{F/F^+}, \{R_{\overline{\rho}_{\widetilde{v}}}^{\bar{\square}}\}_{v\in S} \cup \{D_{\widetilde{v}}\}_{v\in Q_N}\big).
\end{equation*}
Let $R_{\overline{\rho}, \cS_{Q_N}}$ be the corresponding universal deformation ring, and $R_{\overline{\rho},\cS_{Q_N}}^{\square_S}$ the $S$-framed universal deformation ring. By \cite[Prop. 4.4]{Thor}, we can and do choose $Q_N$, $\widetilde{Q}_N$ satisfying moreover (recall $R^{\loc}=\widehat{\otimes}_{v\in S} R_{\widetilde{v}}^{\bar{\square}}$)
\begin{equation}\label{geneG}
  \text{ $R_{\overline{\rho}, \cS_{Q_N}}^{\square_S}$ is topologically generated over $R^{\loc}$ by $g:=q-[F^+:\Q]\frac{n(n-1)}{2}$ elements.}
\end{equation}
Denote by $\Delta_{Q_N}:=\prod_{v\in Q_N} \F_{\widetilde{v}}(p)\cong (\Z/p^N \Z)^{\oplus q}$. We deduce from  (\ref{qn0}) a natural morphism $\Delta_{Q_N}\ra (R_{\overline{\rho}, \cS_{Q_N}}^{\square_S})^{\times}$. Using the fact $\chi_{\widetilde{v}}^{\univ}$ does not depend on the choice of basis, it is not difficult to see this morphism factors through $(R_{\overline{\rho}, \cS_{Q_N}})^{\times}$. We have thus  morphisms of $\co_E$-algebras
\begin{equation*}\co_E[\Delta_{Q_N}]\lra R_{\overline{\rho}, \cS_{Q_N}} \lra R_{\overline{\rho}, \cS_{Q_N}}^{\square_S}.
\end{equation*}
Denote by $\fa_{Q_N}$ the augmentation ideal of $\co_E[\Delta_{Q_N}]$. Then we have
\begin{equation*}
  R_{\overline{\rho}, \cS_{Q_N}}/\fa_{Q_N} \cong R_{\overline{\rho}, \cS}, \ R_{\overline{\rho}, \cS_{Q_N}}^{\square_S}/\fa_{Q_N} \cong R_{\overline{\rho},\cS}^{\square}.
\end{equation*}

\noindent For $v\in Q_N$, denote by $\fp_N^{\widetilde{v}}:=\bigg\{g\in \GL_n(\co_{F_{\widetilde{v}}})\ |\ g \pmod{\varpi_{F_{\widetilde{v}}}} \in \begin{pmatrix}
  \GL_{n-d_N^{\widetilde{v}}} & * \\ 0 & \GL_{d_N^{\widetilde{v}}}
\end{pmatrix}\bigg\}$, where $d_N^{\widetilde{v}}:=\dim_{k_E} \overline{\psi}_{\widetilde{v}}$. Denote by $\fp_{N,1}^{\widetilde{v}}$ the kernel of the following composition
\begin{equation*}
  \fp_N^{\widetilde{v}} \twoheadlongrightarrow \GL_{d_N^{\widetilde{v}}}(\F_{\widetilde{v}}) \xlongrightarrow{\dett} \F_{\widetilde{v}}^{\times} \twoheadlongrightarrow \F_{\widetilde{v}}(p),
\end{equation*}
where the first map is given by the composition of the modulo $\varpi_{F_{\widetilde{v}}}$ map and the natural projection.
Put
\begin{eqnarray*}
  &&U_0(Q_N)_{\widetilde{v}}:=i_{\widetilde{v}}^{-1} (\fp_N^{\widetilde{v}}), \ U_1(Q_N)_{\widetilde{v}}:=i_{\widetilde{v}}^{-1}(\fp_{N,1}^{\widetilde{v}}), \\
  &&U_i(Q_N)^p:=\big(\prod_{\substack{v\nmid p \\ v\notin Q_N}} U_v\big)\big( \prod_{v\in Q_N} U_i(Q_N)_{\widetilde{v}}\big) \subset U^p, \ i=0,1.
\end{eqnarray*}
We have $U_0(Q_N)^p/U_1(Q_N)^p\cong \Delta_{Q_N}$.
\\

\noindent We have by definition (cf. \S~\ref{prelprel}) $\bT(U_i(Q_N)^p)\hookrightarrow \bT(U^p)$, and we use $\fm(\overline{\rho})$ to denote $\fm(\overline{\rho})\cap \bT(U_i(Q_N)^p)$ which is the  maximal ideal of $\bT(U_i(Q_N)^p)$ associated to $\overline{\rho}$ via the relations (\ref{equ: ord-ideal}). As before, we also use the subscript $\overline{\rho}$ to denote the localizations at the maximal ideal $\fm(\overline{\rho})\subset \bT(U_i(Q_N)^p)$.
\\

\noindent Let $i\in \{0,1\}$. For a compact open subgroup $U_p$ of $G(F^+ \otimes_{\Q} \Q_p)$, and for a uniformiser $\varpi_{\widetilde{v}}$ of $\co_{F_{\widetilde{v}}}$ for $v\in Q_N$, we have as in \cite[Prop. 5.9]{Thor} (see also \cite[\S~5.5]{EG14}\cite[\S~2.6]{CEGGPS}) a projection operator $$\pr_{\varpi_{\widetilde{v}}}\in \End_{\co_E}\big(S(U_i(Q_N)^pU_p, \co_E/\varpi_E^r)_{\overline{\rho}}\big)$$
(defined using Hecke operators at $\widetilde{v}$). We denote by $\pr_N:=\prod_{v\in Q_N} \pr_{\varpi_{\widetilde{v}}}$. By \cite[Prop. 5.9]{Thor},  the following composition
\begin{equation}\label{pr0}
  S(U^pU_p, \co_E/\varpi_E^r)_{\overline{\rho}} \hooklongrightarrow S(U_0(Q_N)^p U_p, \co_E/\varpi_E^r)_{\overline{\rho}} \xlongrightarrow{\pr_N} \pr_N(S(U_0(Q_N)^p U_p, \co_E/\varpi_E^r)_{\overline{\rho}})
\end{equation}
is an isomorphism. Since $\pr_N$ is defined using Hecke operators for $\widetilde{v}\in \widetilde{Q}_N$,  (\ref{pr0}) is $\bT(U_0(Q_N)^p)$-equivariant. We also have
\begin{equation*}
  \pr_N\big(S(U_i(Q_N)^p, \co_E/\varpi_E^r)_{\overline{\rho}}\big) \cong \varinjlim_{U_p} \pr_N\big(S(U_i(Q_N)^pU_p, \co_E/\varpi_E^r)_{\overline{\rho}}\big),
\end{equation*}
as a $\widetilde{\bT}(U_i(Q_N)^p)_{\overline{\rho}}\times G(F^+\otimes_{\Q} \Q_p)$-equivariant direct summand of $S(U_i(Q_N)^p, \co_E/\varpi_E^r)_{\overline{\rho}}$ (see \cite[\S~2.6]{CEGGPS}). Similarly, we deduce from (\ref{isoOrd}) isomorphisms
\begin{multline}\label{isoPrN}\pr_N\big(\Ord_P\big(S(U_i(Q_N)^p, \co_E/\varpi_E^r)_{\overline{\rho}}\big)\big) \cong \varinjlim_j \pr_N\big(S(U_i(Q_N)^pK_{j,j},\co_E/\varpi_E^r)_{\overline{\rho}, \ord}\big)
\\
\cong  \Ord_P\big(\pr_N \big(S(U_i(Q_N)^p, \co_E/\varpi_E^r)_{\overline{\rho}}\big) \big),
\end{multline}
and we see that the object in (\ref{isoPrN}) is a $\widetilde{\bT}(U_i(Q_N)^p)_{\overline{\rho}}^{P-\ord}\times L_P(\Q_p)$-equivariant direct summand of $$\Ord_P\big(S(U_i(Q_N)^p, \co_E/\varpi_E^r)_{\overline{\rho}}\big).$$ It is also clear (e.g. by a similar argument as in the proof of Lemma \ref{flat1}) that the decomposition
 \begin{equation*}
\Ord_P\big(S(U_i(Q_N)^p, \co_E/\varpi_E^r)_{\overline{\rho}}\big) \cong \oplus_{\fB} \Ord_P\big(S(U_i(Q_N)^p, \co_E/\varpi_E^r)_{\overline{\rho}}\big)_{\fB}
 \end{equation*}
commutes with $\pr_N$, and hence we have that
\begin{multline*}
V_i(N,\fB,r):=\pr_N\big(\Ord_P\big(S(U_i(Q_N)^p, \co_E/\varpi_E^r)_{\overline{\rho}}\big)_{\fB}\big) \\
\cong \varinjlim_j  \pr_N\big(\Ord_P(S(U_i(Q_N)^p, \co_E/\varpi_E^r)_{\overline{\rho}})_{\fB}^{L_j}\big)
\end{multline*}
is a $\widetilde{\bT}(U_i(Q_N)^p)^{P-\ord}_{\overline{\rho}, \fB}\times L_P(\Q_p)$-equivariant direct summand of $\Ord_P\big(S(U_i(Q_N)^p, \co_E/\varpi_E^r)_{\overline{\rho}}\big)_{\fB}$.  The isomorphism (\ref{pr0}) induces a $\bT(U_0(Q_N)^p)\times L_P(\Q_p)$-equivariant isomorphism
\begin{equation}\label{pr1}
\Ord_P(S(U^p, \co_E/\varpi_E^r)_{\overline{\rho}})_{\fB}\xlongrightarrow{\sim} V_0(N,\fB,r).
\end{equation}
Note that $V_1(N,\fB,r)$ is naturally equipped with a $U_0(Q_N)^p/U_1(Q_N)^p\cong \Delta_{Q_N}$-action, which commutes with $\widetilde{\bT}(U_1(Q_N)^p)^{P-\ord}_{\overline{\rho}, \fB}\times L_P(\Q_p)$.
\begin{lemma}\label{DelNmo1} Let $j\in \Z_{\geq 0}$.

\noindent  (1) $V_1(N,\fB,r)^{L_j}$ and $(V_1(N,\fB,r)^{L_j})^{\vee}$ are finite flat $\co_E/\varpi_E^r[\Delta_{Q_N}]$-modules.

  \noindent (2) There is a natural isomorphism $(V_1(N,\fB,r)^{L_j})^{\vee}_{\Delta_{Q_N}} \xlongrightarrow{\sim} V_0(N,\fB,r)^{L_j}$.
\end{lemma}
\begin{proof}(1) follows from Lemma \ref{flat1} and the fact  that
\begin{equation*}
  V_1(N,\fB,r)^{L_j}\cong  \pr_N\big(\Ord_P(S(U_1(Q_N)^p, \co_E/\varpi_E^r)_{\overline{\rho}})_{\fB}^{L_j}\big)
\end{equation*}
is a $\Delta_{Q_N}$-equivariant direct summand of $\Ord_P(S(U_1(Q_N)^p, \co_E/\varpi_E^r)_{\overline{\rho}})_{\fB}^{L_j}$. We have (e.g. see \cite[\S~2.6]{CEGGPS})
\begin{equation*}
  \pr_N\big(S(U_1(Q_N)^pK_{j,j}, \co_E/\varpi_E^r)_{\overline{\rho}}\big)^{\Delta_{Q_N}} \xlongrightarrow{\sim}   \pr_N\big(S(U_0(Q_N)^pK_{j,j}, \co_E/\varpi_E^r)_{\overline{\rho}}\big).
\end{equation*}
Using the same argument as for (\ref{UpOrd}), we deduce an isomorphism
\begin{equation*}
    \pr_N\big(\Ord_P\big(S(U_1(Q_N)^p, \co_E/\varpi_E^r)_{\overline{\rho}}\big)^{L_j}\big)^{\Delta_{Q_N}} \xlongrightarrow{\sim}   \pr_N\big(\Ord_P\big(S(U_0(Q_N)^p, \co_E/\varpi_E^r)_{\overline{\rho}}\big)^{L_j}\big).
\end{equation*}
(2) follows then by the same argument as for (\ref{isomInv}) (and Pontryagain duality).
\end{proof}
\noindent Put $V_i(N,\fB):=\varprojlim_r V_i(N,\fB,r)$, and we have a natural isomorphism
\begin{equation*}
  V_i(N,\fB)\cong \pr_N\big(\Ord_P(\widehat{S}(U_i(Q_N)^p, \co_E)_{\overline{\rho}, \fB}\big).
\end{equation*}
Put $M_i(N,\fB,r,L_j):=\big(V_i(N,\fB,r)^{L_j}\big)^{\vee}$, %$M_i(N,\fB,L_j):=\varprojlim_r M_i(N,\fB,r,L_j)\cong (V_i(N,\fB)^{L_j})^d$,
and $M_i(N,\fB):= \varprojlim_{j,r}M_i(N, \fB,r, L_j) \cong V_i(N,\fB)^d$.
By Lemma \ref{DelNmo1}(1),  $M_1(N,\fB,r,L_j)$ is a finite flat  $\co_E/\varpi_E^{r}[\Delta_{Q_N}]$-module. By \cite[Lem. 4.4.4]{Pan2}, we deduce (where the conditions of \emph{loc. cit.} are easy to verify in our case)
\begin{proposition}\label{flat2}
$M_1(N,\fB)$ is a flat $\co_E[\Delta_{Q_N}]$-module, and
  \begin{equation*}
  M_1(N,\fB)/\fa_{Q_N} M_1(N,\fB) \xlongrightarrow{\sim} M_0(N,\fB).
  \end{equation*}
\end{proposition}

\subsection{Patching I}\label{patch1}
\noindent By \cite[Prop. 5.3.2]{EG14} (see also the proof \cite[Thm. 6.8]{Thor}), we have a natural surjection
\begin{equation*}
  R_{\overline{\rho},\cS(Q_N)} \twoheadlongrightarrow \widetilde{\bT}(U_1(Q_N)^p)_{\overline{\rho}}.
\end{equation*}
By the local-global compatibility in classical local Langlands correspondance, for any compact open subgroup $U_p$ of $\prod_{v|p} \GL_n(\co_{F_{\widetilde{v}}})$, the induced action of $\co_E[\Delta_{Q_N}]$ on $S(U_1(Q_N)^p U_p, \co_E/\varpi_E^r)_{\overline{\rho}}$ via
\begin{equation}\label{DltN0}
  \co_E[\Delta_{Q_N}] \lra R_{\overline{\rho},\cS(Q_N)}  \lra \widetilde{\bT}(U_1(Q_N)^p)_{\overline{\rho}},
\end{equation}
coincides with the $\co_E[\Delta_{Q_N}]$-action coming from the natural $\Delta_{Q_N}\cong U_0(Q_N)^p/U_1(Q_N)^p$-action.
\\

\noindent We assume henceforth $\overline{\rho}$ is $\fB$-generic (cf. Definition \ref{Bgene}).
By Proposition \ref{Pordlg} (applied with $U^p=U_1(Q_N)$), we can deduce that the morphism
\begin{equation*}
  R_{\overline{\rho},\cS(Q_N)} \twoheadlongrightarrow  \widetilde{\bT}(U_1(Q_N)^p)_{\overline{\rho}, \fB}^{P-\ord}
\end{equation*}
factors through $R_{\overline{\rho}, \cS(Q_N),\fB}^{P-\ord}:=R_{\overline{\rho}, \cS(Q_N)} \otimes_{R_{\overline{\rho}, \cS}} R_{\overline{\rho}, \cS, \fB}^{P-\ord}$, which is also the universal deformation ring of the deformation problem
\begin{equation}\label{defPt}
   \big(F/F^+, S, \widetilde{S}, \co_E, \overline{\rho}, \varepsilon^{1-n} \delta_{F/F^+}, \{R_{\overline{\rho}_{\widetilde{v}}, \cF_{\fB_{\widetilde{v}}}}^{P_{\widetilde{v}}-\ord,\bar{\square}}\}_{v\in S_p} \cup \{R_{\overline{\rho}_{\widetilde{v}}}^{\bar{\square}}\}_{v\in S\setminus S_p} \cup \{D_{\widetilde{v}}\}_{v\in Q_N}\big).
\end{equation}
 Since  $V_1(N,\fB,r)^{L_j}$ is a $\widetilde{\bT}(U_1(Q_N)^p)_{\overline{\rho}}\times U_0(Q_N)^p/U_1(Q_N)^p$-equivariant direct summand of  $$S(U_1(Q_N)^p K_{j,j}, \co_E/\varpi_E^r)_{\overline{\rho}},$$ the two $\co_E[\Delta_{Q_N}]$-actions on $V_1(N,\fB,r)^{L_j}$, obtained by the following two ways (noting the first composition is compatible with (\ref{DltN0}))
\begin{equation*}
  \co_E[\Delta_{Q_N}]\lra R_{\overline{\rho},\cS(Q_N), \fB}^{P-\ord}  \lra \widetilde{\bT}(U_1(Q_N)^p)_{\overline{\rho},\fB}^{P-\ord},
\end{equation*}
\begin{equation*}
  \co_E[\Delta_{Q_N}]\cong \co_E[U_0(Q_N)^p/U_1(Q_N)^p],
\end{equation*}
coincide. By taking limit, we  obtain a similar statement for  $V_1(N,\fB)$ and $M_1(N,\fB)$.
\\

\noindent Denote by $R_{\overline{\rho}, \cS, \fB}^{\square_S, P-\ord}$ \big(resp. by $R_{\overline{\rho}, \cS(Q_N), \fB}^{\square_S, P-\ord}$\big) the $S$-framed deformation ring of the deformation problem (\ref{def2g}) (resp. of (\ref{defPt})). For $*\in \{\cS, \cS(Q_N)\}$, the composition $R^{\loc} \ra R_{\overline{\rho}, *}^{\square_S} \ra R_{\overline{\rho}, *, \fB}^{\square_S, P-\ord}$ thus factors through
\begin{equation*}
  R^{\loc,P-\ord}_{\fB}:= \big(\widehat{\otimes}_{v\in S\setminus S_p} R_{\overline{\rho}_{\widetilde{v}}}^{\bar{\square}}\big) \widehat{\otimes}_{\co_E} \big(\widehat{\otimes}_{v\in S_p} R_{\overline{\rho}_{\widetilde{v}, \fB_{\widetilde{v}}}}^{\bar{\square}, P_{\widetilde{v}}-\ord}\big).
\end{equation*}
For $v\in S_p$, we have a natural morphism $\widehat{\otimes}_{i=1,\cdots, k_{\widetilde{v}}} R_{\fB_{\widetilde{v},i}}^{\pss} \ra R_{\overline{\rho}_{\widetilde{v}, \fB_{\widetilde{v}}}}^{\bar{\square}, P_{\widetilde{v}}-\ord}$,
which induces (see (\ref{RpB}) for $R_{p,\fB}$)
\begin{equation}\label{compPat}
  R_{p,\fB} \hooklongrightarrow R_{\fB} \lra R^{\loc,P-\ord}_{\fB} \lra R_{\overline{\rho}, \cS(Q_N), \fB}^{\square_S, P-\ord},
\end{equation}
where $R_{\fB}:=R_{p,\fB}\widehat{\otimes}_{\co_E} \big(\widehat{\otimes}_{v\in S\setminus S_p} R_{\overline{\rho}_{\widetilde{v}}}^{\bar{\square}}\big)$. It is clear that (\ref{compPat})
factors through $R_{\overline{\rho},\cS(Q_N), \fB}^{P-\ord}$ (e.g. see the argument above (\ref{wtR})). Note that we have
\begin{equation*}
 R_{\overline{\rho}, \cS(Q_N), \fB}^{P-\ord}/\fa_{Q_N} \cong R_{\overline{\rho}, \cS, \fB}^{P-\ord}, \   R_{\overline{\rho}, \cS(Q_N), \fB}^{\square_S, P-\ord}/\fa_{Q_N} \cong R_{\overline{\rho}, \cS, \fB}^{\square_S, P-\ord}.
\end{equation*}

\noindent Let $\co_{\infty}:=\co_E[[z_1,\cdots, z_{n^2|S|}]]$ with the maximal ideal $\ub$, $S_{\infty}:=\co_E[[y_1,\cdots, y_q, z_1,\cdots, z_{n^2|S|}]] $ with the maximal ideal $\fa$, and $R_{\infty}:=R^{\loc,P-\ord}_{\fB}[[x_1,\cdots, x_g]]$.
Denote by $\fa_0=(y_1, \cdots, y_q)$ the kernel of $S_{\infty} \twoheadrightarrow \co_{\infty}$, and $\fa_1:=(z_1, \cdots, z_{n^2|S|}, y_1,\cdots, y_q)$. For an open ideal $\fc$ of $S_{\infty}$, we denote by $s(\fc)$ be the integer such that $S_{\infty}/\fc\cong \co_E/\varpi_E^{s(\fc)}$.
\\

\noindent For each $N\geq 1$, we fix a surjection $\co_E[[y_1,\cdots, y_q]] \twoheadrightarrow \co_E[\Delta_{Q_N}]$ with kernel $\fc_N=((1+y_1)^{p^N}-1, \cdots, (1+y_q)^{p^N}-1)$, which, together with the morphism $\co_E[\Delta_{Q_N}]\ra R_{\overline{\rho},\cS(Q_N),\fB}^{P-\ord}$, induce
$$\co_E[[y_1, \cdots, y_q]]\lra R_{\overline{\rho},\cS(Q_N),\fB}^{P-\ord}.$$
Together with the isomorphism $R_{\overline{\rho},\cS(Q_N),\fB}^{\square_S, P-\ord}\cong R_{\overline{\rho},\cS(Q_N),\fB}^{P-\ord}\widehat{\otimes}_{\co_E} \co_{\infty}$,
 we obtain a morphism of complete noetherian $\co_E$-algebras
\begin{equation}\label{Sinf}
  S_{\infty} \lra   R_{\overline{\rho},\cS(Q_N),\fB}^{\square_S, P-\ord}.
\end{equation}
By (\ref{geneG}), $R_{\overline{\rho},\cS(Q_N),\fB}^{\square_S, P-\ord}$ can be topologically generated by $g$ elements over $R^{\loc, P-\ord}_{\fB}$, hence there exists a surjective map
\begin{equation}\label{RinfR}
  R_{\infty}=R^{\loc,P-\ord}_{\fB}[[x_1,\cdots, x_g]]\twoheadlongrightarrow R_{\overline{\rho},\cS(Q_N),\fB}^{\square_S, P-\ord}\cong R_{\overline{\rho},\cS(Q_N),\fB}^{P-\ord}[[z_1,\cdots, z_{n^2|S|}]].
\end{equation}
We lift the morphism (\ref{Sinf}) to a morphism $S_{\infty} \ra R_{\infty}$. For $i\in \{0,1\}$, $j\geq 0$ and $k>0$, we put
\begin{equation*}
  M_i^{\square}(N,\fB,k, L_j):= M_i(N,\fB,k,L_j)\otimes_{\co_E} \co_{\infty}.
\end{equation*}
Since $M_i(N,\fB,k,L_j)$ is equipped with a natural $R_{\overline{\rho}, \cS(Q_N),\fB}^{P-\ord}$-action via $$R_{\overline{\rho}, \cS(Q_N),\fB}^{P-\ord} \lra  \widetilde{\bT}(U_1(Q_N)^p)_{\overline{\rho}, \fB}^{P-\ord},$$ we see $M_i^{\square}(N,\fB,k, L_j)$ is equipped with a natural $R_{\overline{\rho},\cS(Q_N),\fB}^{\square_S, P-\ord}$($\cong R_{\overline{\rho},\cS(Q_N),\fB}^{P-\ord}\widehat{\otimes}_{\co_E} \co_{\infty}$)-action, and hence with an $S_{\infty}$-action via (\ref{Sinf}).
Moreover, for any open ideal of $S_{\infty}$ containing $\fc_N$ and $\varpi_E^k$,  by Lemma \ref{DelNmo1}, we deduce that  $M_1^{\square}(N,\fB,k,L_j)/I$ is a finite flat $S_{\infty}/I$-module of rank equal to $\rk_{\co_E/\varpi_E^k} M_0(N,\fB,k,L_j)$ (and hence has finite cardinality).
\\

\noindent We use the language of ultrafilters for the patching argument (cf. \cite[\S~8]{SchoLT}). Let $\fF$ be a non-principal  ultrafilter of $\cI:=\Z_{\geq 0}$, $\hR:=\prod_I \co_E$. Let $S_{\fF}\subset \hR$ be the multiplicative set consisting of all idempotents $e_I$ with $I\in \fF$ where $e_I(i)=1$ if $i\in I$, and $e_I(i)=0$ otherwise. Denote by $\hR_{\fF}:=S_{\fF}^{-1} \hR$.  For $k\in \Z_{\geq 1}$, put (noting that the cardinality of $ M_1^{\square}(N,\fB,k, L_j)/\fa^k$ and $M_0^{\square}(N,\fB,k,L_j)/\ub^k$ is finite and bounded by a certain integer independent of $N$)
\begin{eqnarray*}
 M_1^{\infty}(\fB,k)&:=& \varprojlim_j M_1^{\infty}(\fB,k,L_j):=\varprojlim_j \big((\prod_{N\in \cI} M_1^{\square}(N,\fB,k, L_j)/\fa^k) \otimes_{\hR} \hR_{\fF}\big),\\
 M_0^{\infty}(\fB,k)&:=& \varprojlim_j\big((\prod_{N\in \cI} M_0^{\square}(N,\fB,k,L_j)/\ub^k) \otimes_{\hR} \hR_{\fF}\big).
\end{eqnarray*}
The diagonal $S_{\infty}$-action (resp. $\co_{\infty}$-action) on
\begin{equation*}\prod_{N\in \cI} M_1^{\square}(N,\fB,k,L_j)/\fa^k \text{ \big(resp. }\prod_{N\in \cI} M_0^{\square}(N,\fB,k,L_j)/\ub^k\text{\big)}
\end{equation*}induces an $S_{\infty}$-module (resp. an $\co_{\infty}$-module) structure on $M_1^{\infty}(\fB,k)$ (resp. on  $M_0^{\infty}(\fB,k)$). Moreover, $M_1^{\infty}(\fB,k)$ (resp. $M_0^{\infty}(\fB,k)$) is equipped with a natural $S_{\infty}$-linear (resp. $\co_{\infty}$-linear) $L_P(\Q_p)$-action.
By similar (and easier) arguments as in \cite[\S~4.5.5]{Pan2}, we have:
\begin{proposition}\label{pat0} (1) $M_0^{\infty}(\fB,k)\cong M_0(N,\fB) \otimes_{\co_E} \co_{\infty}/\ub^k \cong \Ord_P(\widehat{S}(U^p,\co_E)_{\overline{\rho}})_{\fB}^d\otimes_{\co_E} \co_{\infty}/\ub^k$.

\noindent (2) $M_1^{\infty}(\fB,k)$ is a flat $S_{\infty}/\fa^k$-module, $M_1^{\infty}(\fB,k)/\fa_0\cong M_0^{\infty}(\fB,k)$ and \begin{equation}\label{isomoda}M_1^{\infty}(\fB,k)/\fa_1 \cong \Ord_P(\widehat{S}(U^p,\co_E/\varpi_E^{s(\fa^k+\fa_1)})_{\overline{\rho}})_{\fB}^{\vee}.
\end{equation}

\noindent (3) $M_i^{\infty}(\fB,k)$ is a finitely generated $\co_E[[L_P(\Z_p)]]$-module, and in particular,  $M_i^{\infty}(\fB,k)\in \fC$.
\end{proposition}

\noindent For $j\geq 0$, $k\geq 1$, there exists $d(j,k)> 0$ (independent of $N$) such that the $R_{\overline{\rho}, \cS(Q_N), \fB}^{\square_S, P-\ord}$-action on $M_1^{\square}(N,\fB,k,L_j)/\fa^k$ factors through $R_{\overline{\rho}, \cS(Q_N),\fB}^{\square_S, P-\ord}/\fm_N^{d(j,k)}$ where $\fm_N$ denotes the maximal ideal of $R_{\overline{\rho},\cS(Q_N),\fB}^{\square_S, P-\ord}$. Actually, when $k=1$, it follows easily from the fact (by (\ref{pr1}) and Proposition \ref{flat2}):
\begin{equation*}
  M_1^{\square}(N,\fB,1,L_j)/\fa\cong \big(\Ord_P(S(U^p, \co_E/\varpi_E)_{\overline{\rho}})_{\fB}^{L_j}\big)^{\vee};
\end{equation*} the general case then follows by considering the $\fa$-adic filtration on $M_1^{\square}(N,\fB,k,L_j)/\fa^k$. Note that ${d(j,k)}\ra +\infty$ when $k\ra +\infty$, and we choose $d(j,k)$ such that $d(j',k')\geq d(j,k)$ if $j'\geq j$ and $k'\geq k$. Denote by $R(N, k, {d(j,k)}):=R_{\overline{\rho}, \cS(Q_N),\fB}^{\square_S, P-\ord}/\fm_N^{d(j,k)} \otimes_{S_{\infty}} S_{\infty}/\fa^k$. When $N$ is sufficiently large (satisfying $\fc_N \subset \fa^k$), then we have
\begin{multline}\label{moda1}
  R(N,k,{d(j,k)})/\fa_1\cong \big(R_{\overline{\rho}, \cS(Q_N),\fB}^{\square_S, P-\ord}/\fm_N^{d(j,k)} \otimes_{\co_{\infty}[\Delta_{Q_N}]} O_{\infty}[\Delta_{Q_N}]/\fa^k\big)/\fa_1\\ \cong R_{\overline{\rho}, \cS,\fB}^{ P-\ord}/(\fm^{d(j,k)}, \varpi_E^{s(\fa^k+\fa_1)}),
\end{multline}
where $\fm$ denotes the maximal ideal of $R_{\overline{\rho}, \cS,\fB}^{ P-\ord}$.
In particular, we see $R(N,k,{d(j,k)})$ is an $S_{\infty}/\fa^k$-module of bounded rank (with $N$ varying). Put \begin{equation*}R(\infty,k,{d(j,k)}):=(\prod_{N\in \cI} R(N,k,{d(j,k)})) \otimes_{\hR} \hR_{\fF},\end{equation*}
which naturally acts on $ M_1^{\infty}(\fB,k,j)$.  We have a natural injection $S_{\infty}/\fa^k\hookrightarrow R(\infty,k,{d(j,k)})$ (since $S_{\infty}/\fa^k \hookrightarrow R(N,k,d(j,k))$ for all $N$). By \cite[Lem. 2.2.4]{GN16}, we deduce from (\ref{moda1}) an isomorphism
\begin{equation*}
  R(\infty,k,{d(j,k)})/\fa_1\cong R_{\overline{\rho}, \cS,\fB}^{P-\ord}/(\fm^{d(j,k)}, \varpi_E^{s(\fa^k+\fa_1)}).
\end{equation*}
For $N$ sufficiently large, we have
\begin{equation}\label{moda10}
  M_1^{\square}(N,\fB,k,L_j)/\fa_1\cong \big(\Ord_P(S(U^p, \co_E/\varpi_E^{s(\fa_1+\fa^k)})_{\overline{\rho}})_{\fB}^{L_j}\big)^{\vee},
\end{equation}
and the isomorphism is $R(N,k,{d(j,k)})$-equivariant, where $R(N,k,{d(j,k)})$ acts on the right hand side via the isomorphism (\ref{moda1}).
We deduce then the isomorphism  (which is obtained via the same way as in Proposition \ref{pat0} (2))
\begin{equation*}M_1^{\infty}(\fB,k,L_j)/\fa_1\cong \big(\Ord_P(\widehat{S}(U^p,\co_E/\varpi_E^{s(\fa^k+\fa_1)})_{\overline{\rho}})_{\fB}^{L_j}\big)^{\vee}
\end{equation*} is $R(\infty, k,{d(j,k)})$-equivariant.
The map $R_{\infty} \twoheadrightarrow R(N, k, {d(j,k)})$ induces a natural projection $R_{\infty}\twoheadrightarrow R(\infty,k,{d(j,k)})$, and equips $M_1^{\infty}(\fB,k,L_j)$ with a natural $R_{\infty}$-action. Taking inverse limit on $j$, we see  $M_1^{\infty}(\fB,k)$ is equipped with a natural $R_{\infty}$-action via
\begin{equation}\label{Rinfmp}
  R_{\infty} \twoheadlongrightarrow \varprojlim_{j} R(\infty,k,{d(j,k)}),
\end{equation}
satisfying that the isomorphism in (\ref{isomoda}) is $R_{\infty}$-equivariant, where the $R_{\infty}$-action on the right hand side is induced from the natural action of $R_{\overline{\rho}, \cS,\fB}^{P-\ord}$ via
\begin{equation}\label{Rinfma1}
 R_{\infty}/\fa_1 \twoheadlongrightarrow  \varprojlim_{j} R(\infty,k,{d(j,k)})/\fa_1 \cong \varprojlim_j R_{\overline{\rho}, \cS,\fB}^{P-\ord}/(\fm^{d(j,k)}, \varpi_E^{s(\fa^k+\fa_1)}).
\end{equation}
Note also we have $S_{\infty}/\fa^k \hookrightarrow  \varprojlim_{j} R(\infty,k,{d(j,k)})$.
\\

\noindent
Let $M_i^{\infty}(\fB):=\varprojlim_k M_i^{\infty}(\fB,k)$. By Proposition \ref{pat0} (3), $M_i^{\infty}(\fB)\in \fC$. Using $$M_1^{\square}(N,\fB,k,L_j)/\fa^{k} \cong \big(M_1^{\square}(N,\fB,k+1,L_j)/\fa^{k+1}\big)/\fa^k$$
 for all $N$, we see $M_1^{\infty}(\fB,k,L_j)\cong M_1^{\infty}(\fB,k+1,L_j)/\fa^{k}$. By  \cite[Lem. 4.4.4 (1)]{Pan2}, we have thus $M_1^{\infty}(\fB, k)\cong M_1^{\infty}(\fB,k+1)/\fa^k$ for $k> 0$. By \cite[Lem. 4.4.4 (2)]{Pan2}, we see $M_1^{\infty}(\fB)$ is a flat $S_{\infty}$-module. By Proposition \ref{pat0} (1) (2) and \cite[Lem. 4.4.4 (1)]{Pan2}, we also have $M_1^{\infty}(\fB)/\fa_0 \cong M_0^{\infty}(\fB)$, and
\begin{equation}\label{isomoda2}M_1^{\infty}(\fB)/\fa_1\cong \Ord_P(\widehat{S}(U^p,\co_E)_{\overline{\rho}})_{\fB}^d.
\end{equation}
We have a natural injection
\begin{equation}\label{Sinfinj}
  S_{\infty}\cong \varprojlim_k S_{\infty}/\fa^k \hooklongrightarrow \varprojlim_k \varprojlim_{j} R(\infty,k,{d(j,k)}),
\end{equation}
and $M_1^{\infty}(\fB)$ is equipped with a natural $S_{\infty}$-linear action of $\varprojlim_{k} \varprojlim_j R(\infty,k,d(j,k))$. The morphism (\ref{Rinfmp}) induces
\begin{equation}\label{Rinfsurj}
  R_{\infty} \twoheadlongrightarrow \varprojlim_k \varprojlim_{j} R(\infty,k,{d(j,k)}).
\end{equation}
We can hence lift (\ref{Sinfinj}) to an injection $S_{\infty} \hookrightarrow R_{\infty}$. The $R_{\infty}$-action on  $M_1^{\infty}(\fB)$ (induced by  (\ref{Rinfsurj})) is then $S_{\infty}$-linear. By (\ref{Rinfma1}) (and taking inverse limit over $k$), we have a projection
\begin{equation}\label{Rinfma10}
  R_{\infty}/\fa_1 \twoheadlongrightarrow R_{\overline{\rho}, \cS,\fB}^{P-\ord},
\end{equation}
and the isomorphism (\ref{isomoda2}) is equivariant under the $R_{\infty}$-action, where the $R_{\infty}$-action on the right hand side of (\ref{isomoda2}) is induced from the natural $ R_{\overline{\rho}, \cS,\fB}^{P-\ord}$-action via (\ref{Rinfma10}).
\\

\noindent  We apply Pa{\v{s}}k{\=u}nas' theory. Put
\begin{eqnarray*}
  \tm_i^{\infty}(\fB,k)&:=&\Hom_{\fC}\big(\widetilde{P}_{\fB}, M_i^{\infty}(\fB,k)\big), \\
  \tm_i^{\infty}(\fB)&:=&\Hom_{\fC}\big(\widetilde{P}_{\fB}, M_i^{\infty}(\fB)\big).
\end{eqnarray*}
We have $\tm_i^{\infty}(\fB)\cong \varprojlim_k \tm_i^{\infty}(\fB,k)$. By the same argument as in \cite[Lem. 4.7.4]{Pan2}, $\tm_1^{\infty}(\fB)$ is a flat $S_{\infty}$-module. By the same argument as in the proof of \cite[Prop. 4.7.7 (1)]{Pan2}, $\tm_1^{\infty}(\fB)$ is a finitely generated $R_{\infty}$-module. Denote by $\ub_1:=(z_1,\cdots, z_{n^2|S|})\subset \co_{\infty}$. We have by (\ref{isomoda2})
\begin{equation}\label{isomoda3}
  \tm_1^{\infty}(\fB)/\fa_1 \cong \tm_0^{\infty}(\fB)/\ub_1 \cong \tm(U^p,\fB).
\end{equation}
It is clear that $y_1,\cdots, y_q, z_1,\cdots, z_{n^2|S|}$ form a regular sequence of $\tm_1^{\infty}(\fB)$. Hence by \cite[Prop. 1.2.12]{BruHer}, they can extend to a system of parameters of $\tm_1^{\infty}(\fB)$. By \cite[Prop. A.4]{BruHer}, we have $\dim_{R_{\infty}} \tm_1^{\infty}(\fB)=\dim_{R_{\infty}} \tm(U^p,\fB)+q+n^2|S|$.  Note that the $R_{\infty}$-action on $\tm(U^p,\fB)$ factors through $R_{\overline{\rho}, \cS,\fB}^{P-\ord}$. By Corollary \ref{dim1} (and the fact $\cA$ is finite over $R_{\overline{\rho}, \cS,\fB}^{P-\ord}$), we deduce:
\begin{proposition}\label{suppdim}We have
\begin{multline*}\dim_{R_{\infty}} \tm_1^{\infty}(\fB)\geq 1+q+n^2|S|+\sum_{v\in S_p} (2n-k_{\widetilde{v}})\\ =1+g+n^2|S|+\sum_{v\in S_p}\big(|\{i|n_{\widetilde{v},i}=1\}|+3|\{i|n_{\widetilde{v},i}=2\}|+\frac{n(n-1)}{2}\big) \\ =
 1+g+n^2|S|+\sum_{v\in S_p} \big(\sum_{i=1}^{k_{\widetilde{v}}} n_{\widetilde{v},i}(n-s_{\widetilde{v},i})\big).\end{multline*}
\end{proposition}

\subsection{Patching II}\label{patch2}
\noindent We construct certain patched modules to apply Taylor's Ihara avoidance.
\\

\noindent
%Assume $G$ is quasi-split at all finite places.
Let $\Omega$ be a finite set of finite places $v$ of $F^+$ satisfying that
\begin{itemize}
  \item  $v=\widetilde{v}\widetilde{v}^c$ in $F$,
  \item $p\mid (\Nm(\widetilde{v})-1)$, and if $p^m \parallel  (\Nm(\widetilde{v})-1)$ then $n\leq p^m$.
\end{itemize}
Let $U_\Omega:=\prod_{v\in \Omega} \Iw(\widetilde{v})$, $Y_\Omega:=\prod_{v\in \Omega} \Iw_1(\widetilde{v})$, where $\Iw(\widetilde{v})$ (resp. $\Iw_1(\widetilde{v})$) is the standard Iwahori subgroup (resp. pro-$\ell$ Iwahori subgroup with $\widetilde{v}|\ell$) of $G(F_{\widetilde{v}})$, i.e. the preimage via $i_{G,\widetilde{v}}^{-1}$ of the matrices in $\GL_n(\co_{F_{\widetilde{v}}})$ (recalling that we assume $G$ quasi-split at all finite places of $F^+$) that are upper triangular (resp. upper triangular unipotent)  modulo $\varpi_{\widetilde{v}}$.  We have thus
\begin{equation*}
 \Delta_\Omega:=\prod_{v\in \Omega} (\F_{\widetilde{v}}^{\times})^n \xlongrightarrow{\sim} U_\Omega/Y_\Omega.
\end{equation*}
Enlarging $E$ if necessary, we assume $E$ contains $p^m$-th roots of unity if $p^m\parallel (\Nm(\widetilde{v})-1)$ for $v\in \Omega$.
\\

\noindent Let $U^p=U_\Omega \times \prod_{v\notin S_p\cup \Omega} U_v$, and $Y^p=Y_\Omega \times \prod_{v\notin S_p\cup \Omega} U_v$, and suppose $U^{\Omega,p}:=\prod_{v\notin S_p \cup \Omega}U_v$ is sufficiently small. For any compact open subgroup $U_p$ of $G(F^+\otimes_{\Q} \Q_p)$, $S(Y^pU_p, \co_E/\varpi_E^k)_{\overline{\rho}}$ is equipped with a natural action of $\Delta_\Omega$. For a continuous character $\chi: \Delta_\Omega \ra \co_E^{\times}$, we denote by $S_{\chi}(U^pU_p,\co_E/\varpi_E^k)$ the sub $\co_E/\varpi_E^k$-module of $S(Y^pU_p, \co_E/\varpi_E^k)$ on which $\Delta_\Omega$ acts via $\chi$. Using the fact $U^{\Omega,p}$ is sufficiently small, we have \footnote{Using that $U^{\Omega,p}$ is sufficiently small, we can reduce to the following fact: let $H$ be a finite cyclic group, then $\co_E/\varpi_E^k[H]\cong \co_E/\varpi_E^k[x]/(x^m-1)$ (using a generator $\sigma$ of $H$); let $\chi$ be a character of $H$, then $\co_E/\varpi_E^k[H]_{\chi}\cong \frac{x^m-1}{x-\chi(\sigma)} \co_E/\varpi_E^k[H]$, and hence $\co_E/\varpi_E^k[H]_{\chi} \cong \co_E/\varpi_E^{k+1}[H]_{\chi} \otimes_{\co_E/\varpi_E^{k+1}} \co_E/\varpi_E^k$. } \begin{equation}\label{keychi}S_{\chi}(U^pU_p,\co_E/\varpi_E^{k}) \cong S_{\chi}(U^pU_p,\co_E/\varpi_E^{k+1})\otimes_{\co_E/\varpi_E^{k+1}} \co_E/\varpi_E^k.\end{equation}Consequently, we see
\begin{equation*}
  \widehat{S}_{\chi}(U^p, \co_E):= \varprojlim_k \varinjlim_{U_p} S_{\chi}(U^pU_p, \co_E/\varpi_E^k)
\end{equation*}
is also the subspace of $\widehat{S}(U^p,\co_E)$ on which $\Delta_S$ acts via $\chi$. By the same argument as in \cite[Lem. 6.1]{BD1}, one can show that $\widehat{S}_{\chi}(U^p, \co_E)$ is a finite projective $\co_E[[G(\Z_p)]]$-module.\footnote{Actually, by \emph{loc. cit.}, we can show $\widehat{S}(U^p, \co_E)$ is a finite projective $\co_E[\Delta_{\Omega}][[G(\Z_p)]]$-module, from which we deduce the statement. }
Let $S$ be a finite set containing $\Omega$, $S_p$ and the places $v$ such that $U_v$ is not hyperspecial.
We define $\bT_{\chi}(U^pU_p, *)_{\overline{\rho}}$ in the same way as $\bT(U^pU_p, *)_{\overline{\rho}}$ with $S(U^pU_p,*)_{\overline{\rho}}$ replaced by $S_{\chi}(U^pU_p,*)_{\overline{\rho}}$ for $*\in \{\co_E, \co_E/\varpi_E^k\}$, and $\widetilde{\bT}_{\chi}(U^p)_{\overline{\rho}}:=\varprojlim_{U_p} \bT_{\chi}(U^p,\co_E)_{\overline{\rho}}$ (where $\overline{\rho}$ is as in \S~3.1). Similarly, we can define $\widetilde{\bT}_{\chi}(U^p)_{\overline{\rho},\fB}^{P-\ord}$ which acts faithfully on $\Ord_P(\widehat{S}_{\chi}(U^p, \co_E)_{\overline{\rho}})_{\fB}$. We have natural projections
\begin{equation}\label{RTchi}
R_{\overline{\rho}, \cS} \twoheadlongrightarrow \widetilde{\bT}_{\chi}(U^p)_{\overline{\rho}} \twoheadlongrightarrow \widetilde{\bT}_{\chi}(U^p)_{\overline{\rho},\fB}^{P-\ord}
\end{equation}
satisfying that the composition factors through $R_{\overline{\rho}, \cS, \fB}^{P-\ord}$ (assuming $\overline{\rho}$ is $\fB$-generic). We have thus a natural morphism
$R_{p,\fB} \ra \widetilde{\bT}_{\chi}(U^p)_{\overline{\rho},\fB}^{P-\ord}$.
By the same arguments, we have a similar  local-global compatibility result as in Theorem \ref{lg} for the $\widetilde{\bT}_{\chi}(U^p)_{\overline{\rho},\fB}^{P-\ord}$-action on
\begin{equation}\label{mUpBchi}\fm(U^p,\fB)_{\chi}:= \Hom_{\fC}(\widetilde{P}_{\fB}, \Ord_P(\widehat{S}_{\chi}(U^p, \co_E)_{\overline{\rho}})_{\fB}^d),
 \end{equation}and we have $\dim \widetilde{\bT}_{\chi}(U^p)_{\overline{\rho}, \fB}^{P-\ord}[1/p]\geq   1+ \sum_{v\in S_p} (2n-k_{\widetilde{v}})$ (noting that $\Ord_P(\widehat{S}_{\chi}(U^p,\co_E)_{\overline{\rho}})_{\fB}$ is a finite projective $\co_E[[L_P(\Z_p)]]$-module by \cite[Cor. 4.6]{BD1}).
\\

\noindent Suppose for  $v\in \Omega$, $\overline{\rho}_{\widetilde{v}}$ is trivial. Let $\chi_{\widetilde{v}}=(\chi_{\widetilde{v},i}): \Delta_v:=(\F_{\widetilde{v}}^{\times})^n\ra \co_E^{\times}$. We view $\chi_{\widetilde{v},i}$ as a character of $I_{\widetilde{v}}$ via
\begin{equation*}
  I_{\widetilde{v}} \twoheadlongrightarrow I_{\widetilde{v}}/P_{\widetilde{v}} \twoheadlongrightarrow \F_{\widetilde{v}}^{\times} \xlongrightarrow{\chi_{\widetilde{v},i}} \co_E^{\times}.
\end{equation*}
Denote by $D_{\chi_{\widetilde{v}}}$ the deformation problem consisting of liftings $\rho$ (over artinian $\co_E$-algebras) of $\overline{\rho}_{\widetilde{v}}$ such that for all $\sigma \in I_{\widetilde{v}}$ the characteristic polynomial of $\rho(\sigma)$ is given by $\prod_{i=1}^n (X-\chi_{\widetilde{v},i}(\sigma))$.
Denote by $R_{\overline{\rho}_{\widetilde{v}},\chi_{\widetilde{v}}}^{\bar{\square}}$ the reduced universal deformation ring of $D_{\chi_{\widetilde{v}}}$, which is a quotient of $R_{\overline{\rho}_{\widetilde{v}}}^{\bar{\square}}$. Let $\chi:=\prod_{v\in \Omega} \chi_{\widetilde{v}}$. For the \emph{global} Galois deformation rings considered in the previous sections, we add $\chi$ in the subscript to denote the corresponding universal deformation ring with the local deformation problem $R_{\overline{\rho}_{\widetilde{v}}}^{\bar{\square}}$ replaced by $D_{\chi_{\widetilde{v}}}$ for $v\in \Omega (\subset S)$. For example, we have deformation rings $R_{\overline{\rho}, \cS,\chi}$, $R_{\overline{\rho}, \cS, \chi}^{\square}$, $R_{\overline{\rho},\cS,\fB,\chi}^{P-\ord}$, $R_{\overline{\rho}, \cS,\fB, \chi}^{P-\ord, \square_S}$ etc. By \cite[Prop. 8.5]{Thor}, the morphism $R_{\overline{\rho}, \cS} \twoheadrightarrow \widetilde{\bT}_{\chi}(U^p)_{\overline{\rho}}$ (resp. the composition in (\ref{RTchi})) factors through $R_{\overline{\rho}, \cS,\chi}$ (resp. through $R_{\overline{\rho}, \cS, \fB,\chi}^{P-\ord}$).
\\

\noindent We are particularly interested in the following setting. Let $v_1$ is a finite place of $F^+$ split in $F$ (with $v_1=\widetilde{v}_1\widetilde{v}_1^c$) such that  $p\nmid (\Nm(\widetilde{v}_1)-1)$, and suppose $U^p$ has the following form
\begin{equation}\label{iharaUp}
  U^p= U_\Omega \times \Iw(\widetilde{v}_1) \times  \prod_{v\notin \Omega\cup S_p \cup\{v_1\}} U_v
\end{equation}
where $U_v$ is hyperspecial for all $v\notin \Omega\cup S_p \cup \{v_1\}$ (hence $S=\Omega\cup S_p \cup \{v_1\}$). By the assumption on $v_1$, one can check that $U^p$ is sufficiently small. For $i=0,1$, $k\geq 1$, $j\geq 0$, $N\geq 1$, we define $V_i(N,\fB,k)_{\chi}^{L_j}$ by  the same way as $V_i(N,\fB,k)^{L_j}$ with $S(U_i(Q_n)^p, \co_E/\varpi_E^k)$ replaced by $S_{\chi}(U_i(Q_N)^p, \co_E/\varpi_E^k)$. Put $M_i(N,\fB,k,L_j)_{\chi}:=(V_i(N,\fB,k)_{\chi}^{L_j})^{\vee}$. Let
\begin{equation*}
  R_{\fB,\chi}^{\loc, P-\ord}:=\big(\widehat{\otimes}_{v\in \Omega} R_{\overline{\rho}_{\widetilde{v}}, \chi_{\widetilde{v}}}^{\bar{\square}}\big) \widehat{\otimes}_{\co_E} R_{\overline{\rho}_{\widetilde{v}_1}}^{\bar{\square}} \widehat{\otimes}_{\co_E} \big(\widehat{\otimes}_{v\in S_p} R_{\overline{\rho}_{\widetilde{v}}, \fB_{\widetilde{v}}}^{P_{\widetilde{v}}-\ord, \bar{\square}}\big).
\end{equation*}
The morphisms (\ref{Sinf}), (\ref{RinfR}), and $S_{\infty} \ra R_{\infty}$ (lifting (\ref{Sinf})) induce morphisms
\begin{equation*}S_{\infty} \lra R_{\overline{\rho}, \cS(Q_N), \fB,\chi}^{\square_S, P-\ord},
 \end{equation*}
\begin{equation*}R_{\infty,\chi}:= R_{\fB,\chi}^{\loc, P-\ord}[[x_1,\cdots, x_g]] \twoheadlongrightarrow R_{\overline{\rho}, \cS(Q_N), \fB,\chi}^{\square_S, P-\ord},
\end{equation*}
and $S_{\infty} \ra R_{\infty,\chi}$ respectively. The module
\begin{equation*}M_i^{\square}(N,\fB,r,L_j)_{\chi}:=M_i(N,\fB,r,L_j)_{\chi}\otimes_{\co_E} \co_{\infty}
 \end{equation*}
 is equipped with a natural $S_{\infty}$-linear action of $R_{\infty,\chi}$. We can run the patching argument as in \S~\ref{patch1} with $\{S_{\infty},R_{\infty}, \{M_i^{\square}(N,\fB,k, L_j)\}\}$ replaced by $\{S_{\infty}, R_{\infty,\chi},\{M_i^{\square}(N,\fB,k,L_j)_{\chi}\}\}$, to obtain $R_{\infty,\chi}$-modules $\fm_i^{\infty}(\fB)_{\chi}$ replacing the $R_{\infty}$-modules $\fm_i^{\infty}(\fB)$. By the same arguments, we have that $\fm_1^{\infty}(\fB)_{\chi}$ is flat over $S_{\infty}$ and (cf. (\ref{mUpBchi}))
\begin{equation}\label{isomoda4}\fm_1^{\infty}(\fB)_{\chi}/\fa_1 \cong \fm(U^p, \fB)_{\chi}.\end{equation}
As in Proposition \ref{suppdim}, we have (if $\fm_1^{\infty}(\fB)_{\chi}\neq 0$)
\begin{equation}\label{dim2}
  \dim_{R_{\infty,\chi}} \fm_1^{\infty}(\fB)_{\chi}\geq 1+g+n^2|S|+\sum_{v\in S_p}\big(\sum_i^{k_{\widetilde{v}}} n_{\widetilde{v},i}(n-s_{\widetilde{v},i})\big).
\end{equation}
\noindent
Let $\chi': \Delta_R \ra \co_E^{\times}$ be another character such that $\chi'\equiv \chi \pmod{\varpi_E}$. We have  natural isomorphisms
\begin{eqnarray}\label{RSchibar}R_{\overline{\rho}, \cS,\chi}/\varpi_E &\cong& R_{\overline{\rho}, \cS,\chi'}/\varpi_E,\\ \label{Rinfchibar}R_{\infty,\chi}/\varpi_E&\cong& R_{\infty,\chi'}/\varpi_E.\end{eqnarray}
We have natural isomorphisms (compatible with (\ref{RSchibar}))
\begin{equation*}
M_i(N,\fB,1,L_j)_{\chi}\cong M_i(N,\fB,1,L_j)_{\chi'}
\end{equation*}
for all $N\in \Z_{\geq 1}$, $i\in \{0,1\}$, $j\in \Z_{\geq 0}$, from which we deduce (using (\ref{keychi})) natural isomorphisms $M_i^{\square}(N,\fB,k, L_j)_{\chi}/\varpi_E \cong M_i^{\square}(N,\fB, k,L_j)_{\chi'}/\varpi_E$ which are compatible with (\ref{Rinfchibar}). These isomorphisms finally induce an isomorphism
\begin{equation}\label{iharaiso}
  \fm_1^{\infty}(\fB)_{\chi}/\varpi_E \xlongrightarrow{\sim} \fm_1^{\infty}(\fB)_{\chi'}/\varpi_E
\end{equation}
which is compatible with (\ref{Rinfchibar}). %Finally note that when $\chi=1$, we get nothing but a special case for the objects considered in the previous sections.

%\emph{Let
%\begin{equation*}M_1^{\square}(N,\fB):=M_1(N,\fB)\widehat{\otimes}_{R_{\overline{\rho},\cS(Q_N),\fB}^{P-\ord}} R_{\overline{\rho},\cS(Q_N),\fB}^{\square_S,P-\ord}\cong M_1(N,\fB)\widehat{\otimes}_{\co_E}\co_E[[z_1,\cdots, z_{n^2|S|}]]
%\end{equation*}
%where $\widehat{\otimes}$ denotes the $p$-adic complete tensor product (might not be good). }
\subsection{Automorphy lifting}
\noindent In this section, we prove our main results on automorphy lifting. Recall that $S\supset S_p$ is a finite set of finite places of $F^+$ which split in $F$, and for all $v\in S$, we fix a place $\widetilde{v}$ of $F$ above $v$, and that we assume Hypothesis \ref{hypoglobal}.
\begin{theorem}\label{main}
Let $\rho: \Gal_F \ra \GL_n(E)$ be a continuous representation satisfying the following conditions:
  \begin{enumerate}
    \item $\rho^c\cong \rho^{\vee} \varepsilon^{1-n}$.
    \item $\rho$ is unramified outside $S$.
    \item $\overline{\rho}$ absolutely irreducible, $\overline{\rho}(\Gal_{F(\zeta_p)})\subseteq \GL_n(k_E)$ is adequate  and $\overline{F}^{\Ker \ad \overline{\rho}}$ does not contain $F(\zeta_p)$.
    \item For all $v\in S_p$, $\rho_{\widetilde{v}}$ is $P_{\widetilde{v}}$-ordinary, i.e.
    \begin{equation}\label{Pvfil}
      \rho_{\widetilde{v}} \cong
      \begin{pmatrix}
        \rho_{\widetilde{v},1} &* & \cdots & * \\
        0 & \rho_{\widetilde{v},2} & \cdots & * \\
        \vdots & \vdots & \ddots & \vdots \\
        0 & 0 &\cdots & \rho_{\widetilde{v},k_{\widetilde{v}}}
      \end{pmatrix}
    \end{equation} with $\dim_E \rho_{\widetilde{v},i}=n_{\widetilde{v},i} (\leq 2)$.
    \item For all $v\in S_p$, $\rho_{\widetilde{v}}$ is de Rham of distinct Hodge-Tate weights. Suppose moreover one of the following two conditions holds
     \begin{enumerate} \item for all $v\in S_p$, and $i=1,\cdots, k_{\widetilde{v}}$, $\rho_{\widetilde{v},i}$ is absolutely irreducible and the Hodge-Tate weights of $\rho_{\widetilde{v},i}$ are strictly bigger than those of $\rho_{\widetilde{v},i-1}$;
     \item for all $v|p$, $\rho_{\widetilde{v}}$ is  crystalline with the eigenvalues $(\phi_1,\cdots, \phi_n)$ of the crystalline Frobenius satisfying $\phi_i\phi_j^{-1}\neq 1, p^{\pm 1}$ for $i\neq j$.
     \end{enumerate}
    \item Let $\bar{\cF}_{\widetilde{v}}$ be $P_{\widetilde{v}}$-filtration on $\overline{\rho}_{\widetilde{v}}$ induced by (\ref{Pvfil}), $\fB_{\widetilde{v},i}$ be the block associated to $\tr (\gr^i \bar{\cF}_{\widetilde{v}})$ (cf. \S~\ref{sec: Pas})  and $\fB:=\otimes_{\substack{v\in S_p\\ i=1,\cdots, k_{\widetilde{v}}}} \fB_{\widetilde{v},i}(s_{\widetilde{v},i+1}-1)$ (which is a block of $\Mod_{L_P(\Q_p)}^{\lfin}(\co_E)$). Suppose:
    \begin{enumerate} \item $\overline{\rho}$ is $\fB$-generic (Definition \ref{Bgene}, see also Lemma \ref{bgene1});
    \item $\Hom_{\Gal_{\Q_p}}(\Fil^i_{\bar{\cF}_{\widetilde{v}}} \overline{\rho}_{\widetilde{v}}, \overline{\rho}_{\widetilde{v},i}\otimes_{k_E} \omega)=0$ for all $v|p$, $i=1,\cdots, k_{\widetilde{v}}$.
    \end{enumerate}
    \item There is an automorphic representation $\pi$ of $G$ with the associated representation $\rho_{\pi}: \Gal_F \ra \GL_n(E)$ satisfying
    \begin{enumerate}
    \item $\overline{\rho}_{\pi}\cong \overline{\rho}$;
      \item $\pi_v$ is unramified for all $v\notin S$;
      \item $\pi$ is $\fB$-ordinary (cf. Definition \ref{fBord}, see also Lemma \ref{Bord}).
      %\item for all $v\in S\setminus S_p$, $\rho_{\pi,\widetilde{v}} \rightsquigarrow \rho_{\widetilde{v}}$.
    \end{enumerate}
  \end{enumerate}
  Then $\rho$ is automorphic, i.e. there exists an automorphic representation $\pi'$ of $G$ such that $\rho\cong \rho_{\pi'}$.
\end{theorem}
\begin{proof}
Step (1): Let $U^p=\prod_{v\notin S_p} U_v$ be a sufficiently small compact open subgroup of $G(\bA_{F^+}^{\infty,p})$ such that $U_v$ is hyperspecial for  all $v\notin S$ (enlarging $S$ if necessary). By Lemma \ref{Bord}, the condition 7 is equivalent to the following condition
\begin{enumerate}
  \item[7'.] there exists an automorphic representation $\pi''$ such that 7(a), 7(b) hold for $\pi''$ and that the conditions (a) (b) in Remark \ref{locp} hold for $\pi''$.
\end{enumerate} We replace $\pi$ by $\pi''$, and hence assume $\pi$ satisfies the condition 7'. By solvable base change, we can reduce to the case where
\begin{itemize}
  \item $S=\Omega\cup S_p \cup  \{v_1\}$ is as in \S~\ref{patch2} (in particular, $p\nmid (|\F_{\widetilde{v}_1}|-1)$),
  \item for $v\in \Omega$, $\overline{\rho}_{\widetilde{v}}$ is trivial, and $\rho_{\widetilde{v}}|_{I_{\widetilde{v}}}^{\sss}\cong \rho_{\pi,\widetilde{v}}|_{I_{\widetilde{v}}}^{\sss}\cong 1^{\oplus n}$,
  \item $\ad \overline{\rho}(\Frob_{\widetilde{v}_1})=1$ (hence $\rho_{\widetilde{v}_1}$ and $\rho_{\pi,\widetilde{v}_1}$ are unramified, e.g. by \cite[Lem. 2.4.9, Cor. 2.4.21]{CHT}),
  \item the conditions stay unchanged (with the condition 7 replaced by 7').
\end{itemize}
Actually,
for $v$ such that $\rho_{\widetilde{v}}$ or $\rho_{\pi,\widetilde{v}}$ is ramified (we denote by $S_1$ the set of such places, thus $S_1\subseteq S$), there exists a finite extension $M_{v}$ of $F_v^+$ such that $\overline{\rho}|_{\Gal_{M_v}}$ is trivial, and $\rho_{\widetilde{v}}|_{I_{M_v}}^{\sss}\cong \rho_{\pi,\widetilde{v}}|_{I_{M_v}}^{\sss}\cong 1^{\oplus n}$. If $p\nmid (|\F_{\widetilde{v}}|-1)$, then we can enlarge $M_v$ such that $\rho_{\widetilde{v}}|_{\Gal_{M_v}}$ and $\rho_{\pi,\widetilde{v}}|_{\Gal_{M_v}}$ are unramified; otherwise, we enlarge $M_v$ such that if $p^m \parallel  (|\F_{\widetilde{v}}|-1)$ for $m\geq 1$, then $n\leq p^m$. Since $\overline{F}^{\Ker \ad \overline{\rho}}$ does not contain $F(\zeta_p)$, we choose a finite place $v_1'\notin S$ of $F^+$, split in $F$ (with $v_1'=\widetilde{v}_1'(\widetilde{v}_1')^c$) such that  $\widetilde{v}_1'$ does not split completely in $F(\zeta_p)$ (hence $p\nmid (|\F_{\widetilde{v}}|-1)$) and that $\ad \overline{\rho}(\Frob_{\widetilde{v}_1'})=1$. By \cite[Lem. 4.1.2]{Tay08}, we let $L^+/F^+$ be a solvable totally real extension linearly disjoint from $\overline{F}^{\Ker(\overline{\rho})}(\zeta_p)$, and that for a finite place $w$ of $L^+$, and  $v$ the place of $F^+$ with $w|v$, we have
\begin{itemize} \item if $v\in S_1$, then $L^+_w\cong M_v$,
\item if $v\in S_p$, then $L^+_w\cong F^+_v$ ($\cong \Q_p$),
\item if $v=v_1'$, then $L^+_w \cong F^+_{v_1'}$.
\end{itemize}
Let $v_1$ be a place of $L^+$ above $v_1'$. We replace $P$ by $\prod_{w|p} P_{\widetilde{w}}$ where $P_{\widetilde{w}}=P_{\widetilde{v}}$ for a $p$-adic place $w$ of $L^+$ with $v$ the place of $F^+$ such that $w|v$ (and where $\widetilde{w}$ is a place of $L$ above $w$ that we fix as we have done for places in  $F$, cf. \S~\ref{prelprel}), and we replace  $F/F^+$ by $L/L^+$. Using \cite[Prop. 2.7]{Ger19} \cite[Lem. 4.2.2]{CHT}, we reduce to the situation below the condition 7'  (noting that the base change of $\pi$ to $L^+$ still satisfies the condition 7').
\\

\noindent Step (2): Let $U^p$ be as in (\ref{iharaUp}), and let $\chi=\prod_{v\in \Omega} \chi_{\widetilde{v}}:\Delta_\Omega \ra \co_E^{\times}$ with $\chi_{\widetilde{v}}=(\chi_{\widetilde{v},i})_{i=1,\cdots, n}$ satisfying that $\chi_{\widetilde{v},i}\equiv 1\pmod{\varpi_E}$ and that the $\chi_{\widetilde{v},i}$'s are distinct for $i=1,\cdots, n$. We have as in \S~\ref{patch2} an $R_{\infty,1}$-module $\fm_1^{\infty}(\fB)_1$, and an $R_{\infty,\chi}$-module $\fm_1^{\infty}(\fB)_{\chi}$. By the assumptions (i.e. $\pi$ is unramified for places not in $S_p\cup \Omega$, and  for $v\in \Omega$, we have $\pi_{\widetilde{v}}^{\Iw(\widetilde{v})}\neq 0$ (since $\rho_{\pi,\widetilde{v}}|_{I_{\widetilde{v}}}^{\sss}\cong 1^{\oplus n}$)), $\pi^{U^p}\neq 0$. By the condition 7' and Lemma \ref{Bord}, we have thus
\begin{equation*}
  \Ord_P(\widehat{S}_1(U^p,\co_E)_{\overline{\rho}})_{\fB}=\Ord_P(\widehat{S}(U^p,\co_E)_{\overline{\rho}})_{\fB}\neq 0.
\end{equation*}
 Hence $\fm(U^p,\fB)_1=\fm(U^p,\fB)\neq 0$, and $\fm_1^{\infty}(\fB)_1\neq 0$.
Let $\fm_x$  be the maximal ideal of $R_{\overline{\rho},\cS,\fB,1}^{P-\ord}[1/p]$  associated to $(\rho, \{\rho_{\widetilde{v},i}\})$, $\fm_x^{\infty}$ be the preimage of $\fm_x$  via the projection $R_{\infty,1}[1/p] \twoheadrightarrow R_{\overline{\rho},\cS,\fB,1}^{P-\ord}[1/p]$ (cf. (\ref{Rinfma10})).
Let $x \in \Spec R_{\overline{\rho},\cS,\fB,1}^{P-\ord}[1/p]$, $x_{\infty}\in \Spec R_{\infty,1}[1/p]$ be the closed points associated to  $\fm_x$, $\fm_x^{\infty}$ respectively.
%% We can thus naturally  attach to $\pi$ a maximal ideal $\fm_y^T$ of $\widetilde{\bT}_1(U^p)_{\overline{\rho},\fB}^{P-\ord}[1/p]$ ($=\widetilde{\bT}(U^p)_{\overline{\rho}, \fB}^{P-\ord}[1/p]$), in particular we have $\fm_1^{\infty}(\fB)_1\neq 0$.$y_T\in \Spec \widetilde{\bT}_1(U^p)_{\overline{\rho},\fB}^{P-\ord}[1/p]$, $y_R \in \Spec R_{\overline{\rho},\cS,\fB,1}^{P-\ord}[1/p]$, $y_{\infty} \in \Spec R_{\infty,1}[1/p]$, $\fm_y^T$, $\fm_y^R$, $\fm_{y}^{\infty}$%Let $\fm_y^R\subset R_{\overline{\rho},\cS,\fB,1}^{P-\ord}[1/p]$ ($=R_{\overline{\rho},\cS,\fB}^{P-\ord}[1/p]$) be the preimage of $\fm_y^T$ via the natural projection $R_{\overline{\rho},\cS,\fB,1}^{P-\ord}[1/p] \twoheadrightarrow \widetilde{\bT}_1(U^p)_{\overline{\rho},\fB}^{P-\ord}[1/p]$.
\\

\noindent By Corollary \ref{Porddim0} and Condition 6(b) (for places in $S_p$), \cite[Prop. 3.1]{Tay08} (for places in $\Omega$) and \cite[Lem. 2.4.9, Cor. 2.4.21]{CHT} (for $v_1$),  $R_{\infty,1}$, $R_{\infty,\chi}$ are both equidimensional of relative dimension
\begin{equation*}
  g+n^2|S|+\sum_{v\in S_p} \big(\sum_{i=1}^{k_{\widetilde{v}}} n_{\widetilde{v},i}(n-s_{\widetilde{v},i})\big)
\end{equation*}
over $\co_E$.
By (\ref{dim2}), we see $\fm_1^{\infty}(\fB)_1$ is supported on a union of irreducible components of $\Spec R_{\infty,1}$. Giving an irreducible component $\cC$ of $\Spec R_{\infty,1}$ or $\Spec R_{\infty,\chi}$ is the same as giving an irreducible component $\cC_v$ of each $v\in S$.  However,  $\cC_{v}$ is unique if $v\in S_p$ by  Condition 6(b) and Corollary \ref{Porddim0} or $v=v_1$ (noting $R_{\overline{\rho}_{\widetilde{v}_1}}^{\bar{\square}}$ is formally smooth over $\co_E$). So giving $\cC$ (as above) is the same as giving an irreducible component $\cC_v$ of each $v\in \Omega$, and we denote by $\cC=\otimes_{v\in \Omega} \cC_v$. By (\ref{isomoda3}), Proposition \ref{pts0} (and the fact $\fm_1^{\infty}(\fB)_1\neq 0$),  there is an irreducible component  $\cC=\otimes_{v\in \Omega}\cC_v$ of $\Spec R_{\infty,1}$ contained in the support of $\fm_1^{\infty}(\fB)_1$. Denote by $\overline{\cC}$ the modulo $\varpi_E$ reduction of $\cC$. We see $\overline{\cC}$ is contained in the support of $\fm_1^{\infty}(\fB)_1/\varpi_E$. Using the isomorphism in (\ref{iharaiso}), we deduce that $\overline{\cC}$ is contained in the support of $\fm_1^{\infty}(\fB)_{\chi}/\varpi_E$. Thus $\Supp_{R_{\infty,\chi}} \fm_1^{\infty}(\fB)_{\chi}$ contains an irreducible component $\cC'=\otimes_{v\in \Omega} \cC'_v$. Since $R_{\overline{\rho}_{\widetilde{v}}, \chi_{\widetilde{v}}}^{\bar{\square}}$ is irreducible for $v\in \Omega$ (cf. \cite[Prop. 3.1 (1)]{Tay08}), we see $\cC'_v=R_{\overline{\rho}_{\widetilde{v}}, \chi_{\widetilde{v}}}^{\bar{\square}}$ for $v\in \Omega$. Using again the isomorphism in (\ref{iharaiso}),  $\Supp_{R_{\infty,1}/\varpi_E} \fm_1^{\infty}(\fB)_1/\varpi_E$ contains the modulo $\varpi_E$ reduction $\overline{\cC'}$ of $\cC'$. However, by \cite[Prop. 3.1 (2)]{Tay08}, we deduce that the modulo $\varpi_E$ reduction of any irreducible component of $\Spec R_{\infty,1}$ is contained in  $\overline{\cC'}$. Together with \cite[Prop. 3.1 (3)]{Tay08}, we deduce any irreducible component  $\Spec R_{\infty,1}$  is contained in $\Supp_{R_{\infty,1}} \fm_1^{\infty}(\fB)_1$. In particular $x_{\infty} \in \Supp_{R_{\infty,1}[1/p]} \fm_1^{\infty}(\fB)_1[1/p]$. And we have thus $\fm_1^{\infty}(\fB)_1[1/p]/\fm_x^{\infty}\neq 0$.
Using (\ref{isomoda4}), we deduce then  $\tm(U^p,\fB)[1/p]/\fm_x\neq 0$ (noting $\tm(U^p,\fB)_1=\tm(U^p,\fB)$). The $R_{\overline{\rho}, \cS,\fB,1}^{P-\ord}$-action on $\tm(U^p,\fB)$ factors through $\widetilde{\bT}(U^p)_{\overline{\rho}, \fB}^{P-\ord}=\widetilde{\bT}_1(U^p)_{\overline{\rho}, \fB}^{P-\ord}$, we deduce there exists a maximal idea $\fm_x^T$ of $\widetilde{\bT}(U^p)_{\overline{\rho}, \fB}^{P-\ord}[1/p]$ (which is actually equal to the image of $\fm_x$) such that $\tm(U^p,\fB)[1/p]/\fm_x^T\neq 0$.
By Theorem \ref{lg}, we obtain a closed point $z=(x_T, \{z_{\widetilde{v},i}\})\in \Spec \cA[1/p]$ where $x_T$ is the point associated to $\fm_x^T$, and $z_{\widetilde{v},i}=\tr \rho_{\widetilde{v},i}$ (noting that the preimage of $\fm_x$ of the first morphism in (\ref{RpB2}) is the prime ideal corresponding to $\{\tr \rho_{\widetilde{v},i}\}$). % where $\cA_1\cong \widetilde{\bT}_1(U^p)_{\overline{\rho}, \fB}^{P-\ord}[1/p]$ is constructed in the same way as $\cA$ in \S~\ref{GL2ord}, replacing $\{\widetilde{\bT}(U^p)_{\overline{\rho}, \fB}^{P-\ord}, \fm(U^p,\fB)\}$ by $\{\widetilde{\bT}_1(U^p)_{\overline{\rho}, \fB}^{P-\ord}, \fm(U^p,\fB)_1\}$. It is straightforward to see all the results in \S~\ref{GL2ord} hold with $\cA$ replaced by $\cA_1$.
\\

\noindent Step (3):
Suppose the condition 5 (a), then by Proposition \ref{class}, $x_T$ is classical, and the theorem follows.

\noindent Suppose now the condition 5 (b).
By Proposition \ref{pts0}, we have a non-zero map
\begin{equation*}
  \widehat{\otimes}_{\substack{v\in S_p \\ i=1,\cdots, k_{\widetilde{v}}}} (\widehat{\pi}_{z_{\widetilde{v},i}} \otimes_{E} \varepsilon^{s_{\widetilde{v},i+1}-1}\circ \dett)\lra \Ord_P\big(\widehat{S}(U^p,E)_{\overline{\rho}}\big)_{\fB}[\fm_x^T].
\end{equation*}
Since $\rho_{\widetilde{v},i}$ is crystalline, by the $p$-adic Langlands correspondence for $\GL_2(\Q_p)$, the irreducible constituents of  $\widehat{\pi}_{z_{\widetilde{v},i}}^{\an}$ are all subquotients of locally analytic principal series (induced from locally algebraic characters of $T(\Q_p)$). We deduce then there exist locally algebraic characters $\chi_{\widetilde{v}}$ of $T(\Q_p)$ for $v \in S_p$ such that
\begin{equation}\label{jac}
  \otimes_{v\in S_p} \chi_{\widetilde{v}} \hooklongrightarrow J_{B\cap L_P}\big(\Ord_P\big(\widehat{S}(U^p,E)_{\overline{\rho}}\big)_{\fB}[\fm_x^T]^{\an}\big),
\end{equation}
where $J_{B\cap L_P}(-)$ denotes the Jacquet-Emerton functor (\cite{Em11}).
From the locally analytic representation $J_{B\cap L_P}(\Ord_P\big(\widehat{S}(U^p,E)_{\overline{\rho}}\big)_{\fB}^{\an})$, using Emerton's machinery \cite{Em1}, we can construct an eigenvariety $\cE^{P-\ord}_{\fB}$ as in  \cite[\S 7.1.3]{BD1}, such that
\begin{itemize} \item any point of $\cE^{P-\ord}_{\fB}$ can be parameterized as $(\fm_z, \chi)$ where $\chi$ is a character of $T(F^+\otimes_{\Q} \Q_p)$ ($\cong \prod_{v\in S_p} T(\Q_p)$) and $\fm_z$ is a maximal ideal of $\widetilde{\bT}(U^p)_{\overline{\rho}, \fB}^{P-\ord}[1/p]$;
\item $(\fm_z, \chi)\in \cE^{P-\ord}_{\fB}$ if and only if\begin{equation*}
J_{B\cap L_P}\big(\Ord_P\big(\widehat{S}(U^p,E)_{\overline{\rho}}\big)_{\fB}^{\an}\big)[\fm_z, T(F^+\otimes_{\Q} \Q_p)=\chi]\neq 0.
\end{equation*}
\end{itemize}
In particular, by (\ref{jac}), we see $\fx:=(\fm_x^T, \otimes_{v\in S_p} \chi_{\widetilde{v}})\in \cE^{P-\ord}_{\fB}$. By the same argument as for \cite[(7.28)]{BD1}, one can show there exists a natural injection $(\cE^{P-\ord}_{\fB})^{\red} \hookrightarrow \cE$ where $\cE$ denotes the eigenvariety associated to $G$ with the tame level $U^p$ (constructed from $J_B(\widehat{S}(U^p,E)_{\overline{\rho}}^{\an})$).  Hence we get a point $\fx=(\fm_x^T, \otimes_{v\in S_p} \chi_{\widetilde{v}})\in \cE$. Since $\rho_{x,\widetilde{v}}$ is crystalline and generic for all $v\in S_p$, by \cite[Thm. 5.1.3, Rem. 5.1.4]{BHS3}, $\fm_x^T$ is classical. This concludes the proof.
\end{proof}
\subsection{Locally analytic socle}
\noindent We use the (patched) $\GL_2(\Q_p)$-ordinary families to show some results towards Breuil's locally analytic socle conjecture \cite{Br13I} for certain non-trianguline case. We begin with some preliminaries on representations.
 \begin{lemma}\label{rep0}
   Let $U$ be a unitary admissible Banach representation of $L_P(\Q_p)$ over $E$, $U^{\an}$ be the subrepresentation of locally analytic vectors. Then the following diagram commutes
   \begin{equation}\label{comDia0}
   \begin{CD}
      U^{\an} @>>> (\Ind_{\overline{P}(\Q_p)}^{G_p} U^{\an})^{\an} \\
      @VVV @VVV \\
      U @>>> (\Ind_{\overline{P}(\Q_p)}^{G_p} U)^{\cC},
   \end{CD}
   \end{equation}
   where $(\Ind -)^{\an}$ (resp. $(\Ind -)^{\cC}$) denotes the locally analytic (resp. the continuous) parabolic induction, where the top horizontal map sends $u$ to $f_u\in \cC^{\la}(N_0, U^{\an}) \hookrightarrow (\Ind_{\overline{P}(\Q_p)}^{G_p} U^{\an})^{\an} $ via \cite[(2.3.7)]{Em2} (see \S~\ref{sec: Pord-3} for $N_0$) with $f_u$ the constant function of value $u$, and where the bottom horizontal map is given by the composition (see \cite[Cor. 4.3.5]{Em2} for the first isomorphism, and see \cite[(3.4.7)]{Em2} for the second map, which is the canonical lifting map of \emph{loc. cit.} with respect to $N_0$)
   \begin{equation*}
     U\xlongrightarrow{\sim} \Ord_P\big((\Ind_{\overline{P}(\Q_p)}^{G_p} U)^{\cC}\big) \longrightarrow (\Ind_{\overline{P}(\Q_p)}^{G_p} U)^{\cC}.
   \end{equation*}
 \end{lemma}
\begin{proof}
  The lemma follows by the same argument as in \cite[Lem. 4.20]{BD1}.
\end{proof}
\begin{lemma}\label{adj01}
  Let $V$ be a unitary admissible Banach representation of $G_p:=G(F^+\otimes_{\Q} \Q_p)$ over $E$, and $U$ be a unitary admissible Banach representation of $L_P(\Q_p)$ over $E$. Suppose that we have an $L_P(\Q_p)$-equivariant non-zero map
$ U \ra \Ord_P(V)$, such that the following composition is non-zero:
  \begin{equation}\label{OrdPlalg}
    U^{\lalg} \hooklongrightarrow U \lra \Ord_P(V).
  \end{equation}
Then the composition (where the second map is induced by the second map of (\ref{OrdPlalg}) by \cite[Thm. 4.4.6]{EOrd1})
\begin{equation}\label{adj00}
  (\Ind_{\overline{P}(\Q_p)}^{G_p} U^{\lalg})^{\an} \hooklongrightarrow (\Ind_{\overline{P}(\Q_p)}^{G_p} U)^{\cC} \lra V
\end{equation}
is non-zero. Moreover, the following diagram commutes
\begin{equation}\label{comDia1}
  \begin{CD}
    U^{\lalg} @>>>  (\Ind_{\overline{P}(\Q_p)}^{G_p} U^{\lalg})^{\an} \\
    @V (\ref{OrdPlalg}) VV @V (\ref{adj00}) VV \\
    \Ord_P(V) @>>> V
  \end{CD}
\end{equation}
where the top horizontal map is given as in the horizontal map of (\ref{comDia0}) with $U^{\an}$ replaced by $U^{\lalg}$, and the bottom horizontal map is the canonical lifting with respect to $N_0$.
\end{lemma}
\begin{proof}
  It is sufficient to show (\ref{comDia1}) is commutative. However, by \cite[Thm. 4.4.6]{EOrd1}, the following diagram commutes
  \begin{equation*}
    \begin{CD}
       U @>>>  (\Ind_{\overline{P}(\Q_p)}^{G_p} U)^{\cC} \\
    @V (\ref{OrdPlalg}) VV @V (\ref{adj00}) VV \\
    \Ord_P(V) @>>> V
  \end{CD}.
  \end{equation*}
The lemma then follows from Lemma \ref{rep0} (noting that as an easy consequence of Lemma \ref{rep0}, the statement of Lemma \ref{rep0} holds also with $U^{\an}$ replaced by any closed subrepresentation of $U^{\an}$, and in particular holds for $U^{\lalg}$).
\end{proof}
\noindent For a weight $\lambda=(\lambda_1,\cdots, \lambda_n)$ of $\GL_n$, let  $\overline{L}(\lambda)$ be the unique simple quotient of $\text{U}(\gl_n) \otimes_{\text{U}(\overline{\ub})} \lambda$ with $\overline{\ub}$ the Lie algebra of the Borel subgroup $\overline{B}$ of lower triangular matrices. If $\lambda$ is integral and is  dominant for a parabolic $\overline{P}\supseteq \overline{B}$, then  $\overline{L}(\lambda)$ lies in the BGG category $\co_{\alg}^{\overline{\fp}}$ (cf. \cite[\S~2]{OS}, where $\overline{\fp}$ is the Lie algebra of $\overline{P}$).
\begin{theorem}\label{main2}
  Suppose that all the assumptions, except the assumption 5, of Theorem \ref{main} hold and suppose that  we are in the situation of Step (1) of the proof of Theorem \ref{main}. Suppose the followings (as a replacement of the assumption 5 of Theorem \ref{main}) hold
\begin{itemize}
  \item for all $v|p$, $\rho_{\widetilde{v}}$ is Hodge-Tate of distinct Hodge-Tate weights $h_{\widetilde{v},1}> \cdots> h_{\widetilde{v},n}$;
  \item for all $v|p$, $i=1,\cdots, k_{\widetilde{v}}$, $\rho_{\widetilde{v},i}$ is de Rham and  absolutely irreducible.
\end{itemize}
Then there exists a non-zero morphism of locally analytic $G_p$-representations:
\begin{equation}\label{soc1}
 \widehat{\otimes}_{v|p} \cF_{\overline{P}_{\widetilde{v}}}^{\GL_n}(\overline{L}(-s_{\widetilde{v}}\cdot \lambda_{\widetilde{v}}), \pi_{\widetilde{v}}^{\infty}) \longrightarrow \widehat{S}(U^p,E)[\fm_{\rho}]^{\an},
\end{equation}
where \begin{itemize}
 \item``$-[\fm_{\rho}]$" denotes the subspace annihilated by the maximal ideal $\fm_{\rho}\subseteq R_{\overline{\rho}, \cS}[1/p]$ associated to $\rho$,
  \item$\cF_{\overline{P}_{\widetilde{v}}}^{\GL_n}(-,-)$ is the Orlik-Strauch functor (\cite{OS}),
  \item $\lambda_{\widetilde{v}}:=(\lambda_{\widetilde{v},1},\cdots, \lambda_{\widetilde{v},n})$ with $\lambda_{\widetilde{v},i}:=h_{\widetilde{v},i}+i-1$ (so $\lambda_{\widetilde{v}}$ is dominant for $B$),
  \item $\pi_{\widetilde{v}}^{\infty}=\otimes_{i=1}^{k_{\widetilde{v}}} \pi_{\widetilde{v},i}^{\infty}(s_{\widetilde{v},i+1}-1)$ with ``$- (s_{\widetilde{v},i+1}-1)$" the twist $\unr(p^{-(s_{\widetilde{v},i+1}-1)})\circ \dett$ and  with $\pi_{\widetilde{v},i}^{\infty}$  the smooth $\GL_{n_{\widetilde{v},i}}(\Q_p)$ representation corresponding to $\WD(\rho_{\widetilde{v},i})$ (normalized in the way that $\widehat{\pi}(\rho_{\widetilde{v},i})^{\lalg}$ is isomorphic to the tensor product of $\pi_{\widetilde{v},i}^{\infty}$ with an algebraic representation of $\GL_{n_{\widetilde{v},i}}(\Q_p)$),
  \item $s_{\widetilde{v}}\in S_n$ satisfies that we have an equality of ordered sets
\begin{equation*}\big(h_{\widetilde{v},s_{\widetilde{v}}^{-1}(1)}, \cdots, h_{\widetilde{v},s_{\widetilde{v}}^{-1}(n)}\big)=\big(h_{\rho_{\widetilde{v},1},1}, h_{\rho_{\widetilde{v},1}, n_{\widetilde{v},1}}, \cdots, h_{\rho_{\widetilde{v},k_{\widetilde{v}}}, 1}, h_{\rho_{\widetilde{v},k_{\widetilde{v}}},n_{\widetilde{v}, k_{\widetilde{v}}}}\big),
\end{equation*}
$\{h_{\rho_{\widetilde{v},i},1}, h_{\rho_{\widetilde{v},i}, n_{\widetilde{v},i}}\}$ being the set of the Hodge-Tate weights of $\rho_{\widetilde{v},i}$ with $h_{\rho_{\widetilde{v},i},1}\geq h_{\rho_{\widetilde{v},i}, n_{\widetilde{v},i}}$ (so $-s_{\widetilde{v}}\cdot \lambda_{\widetilde{v}}$ is dominant for $\overline{P}_{\widetilde{v}}$).
\end{itemize}
\end{theorem}
\begin{proof}
  By Step (2) of the proof of Theorem \ref{main} (and we use the notation there), we have $z=(x_T, \{z_{\widetilde{v},i}\})\in \Spec \cA[1/p]$. By Proposition \ref{pts0}, there exists a non-zero $L_P(\Q_p)$-equivariant morphism (without loss of generality, we assume the residue field at $x$ is equal to $E$)
  \begin{equation}\label{lgOrd1}
    \widehat{\otimes}_{\substack{v\in S_p \\ i=1,\cdots, k_{\widetilde{v}}}} (\widehat{\pi}_{z_{\widetilde{v},i}} \otimes_{E} \varepsilon^{s_{\widetilde{v},i+1}-1}\circ \dett) \lra  \Ord_P\big(\widehat{S}(U^p,E)_{\overline{\rho}}\big)_{\fB}[\fm_x^T].
  \end{equation}
Note that $\Ord_P\big(\widehat{S}(U^p,E)_{\overline{\rho}}\big)_{\fB}[\fm_x^T]=\Ord_P\big(\widehat{S}(U^p,E)_{\overline{\rho}}\big)_{\fB}[\fm_{\rho}]$.
Since $\rho_{\widetilde{v},i}$ is absolutely irreducible, we have $\widehat{\pi}_{z_{\widetilde{v},i}} \cong \widehat{\pi}(\rho_{\widetilde{v},i})$. Also by the same argument as in the proof of Proposition \ref{class}, (\ref{lgOrd1}) induces a non-zero $L_P(\Q_p)$-equivariant morphism
\begin{equation*}
\otimes_{\substack{v\in S_p \\ i=1,\cdots, k_{\widetilde{v}}}} (\widehat{\pi}(\rho_{\widetilde{v},i})^{\lalg} \otimes_{E} \varepsilon^{s_{\widetilde{v},i+1}-1}\circ \dett) \lra  \Ord_P\big(\widehat{S}(U^p,E)_{\overline{\rho}}\big)_{\fB}[\fm_{\rho}].
\end{equation*}
By Lemma \ref{adj01}, we deduce hence a non-zero $G_p$-equivariant morphism
\begin{multline*}
  \widehat{\otimes}_{v\in S_p} \big(\Ind_{\overline{P}_{\widetilde{v}}}^{\GL_n(\Q_p)} \otimes_{i=1}^{k_{\widetilde{v}}} (\widehat{\pi}(\rho_{\widetilde{v},i})^{\lalg} \otimes_{E} \varepsilon^{s_{\widetilde{v},i+1}-1}\circ \dett) \big) ^{\an}\\ \cong \big(\Ind_{\overline{P}}^{G_p} \otimes_{\substack{v\in S_p \\ i=1,\cdots, k_{\widetilde{v}}}} (\widehat{\pi}(\rho_{\widetilde{v},i})^{\lalg} \otimes_{E} \varepsilon^{s_{\widetilde{v},i+1}-1}\circ \dett) \big)^{\an}
  \lra \widehat{S}(U^p,E)_{\overline{\rho}}[\fm_{\rho}].
\end{multline*}
For $v\in S_p$, we have
\begin{equation*}
  \otimes_{i=1}^{k_{\widetilde{v}}} (\widehat{\pi}(\rho_{\widetilde{v},i})^{\lalg} \otimes_{E} \varepsilon^{s_{\widetilde{v},i+1}-1}\circ \dett) \cong \pi_{\widetilde{v}}^{\infty} \otimes_E L(s_{\widetilde{v}}\cdot \lambda_{\widetilde{v}})_{P_{\widetilde{v}}}
\end{equation*}
where $L(s_{\widetilde{v}}\cdot \lambda_{\widetilde{v}})_{P_{\widetilde{v}}}$ denotes the algebraic representation of the Levi subgroup of $P_{\widetilde{v}}$ (containing the diagonal subgroup) of highest weight $s_{\widetilde{v}}\cdot \lambda_{\widetilde{v}}$ (with respect to the Borel subgroup of upper triangular matrices). By \cite[Cor. 2.5]{Br13I} and \cite[Lem. 2.10]{BH2}, we have
\begin{multline}\label{compSoc}
   \widehat{\otimes}_{v|p} \cF_{\overline{P}_{\widetilde{v}}}^{\GL_n}(\overline{L}(-s_{\widetilde{v}}\cdot \lambda_{\widetilde{v}}), \pi_{\widetilde{v}}^{\infty}) \cong \cF_{\overline{P}}^{G_p}(\otimes_{v\in S_p} \overline{L}(-s_{\widetilde{v}}\cdot \lambda_{\widetilde{v}}), \otimes_{v\in S_p} \pi_{\widetilde{v}}^{\infty})\\  \cong \soc_{G_p}  \big(\Ind_{\overline{P}}^{G_p} \otimes_{v\in S_p} (\pi_{\widetilde{v}}^{\infty} \otimes_E L(s_{\widetilde{v}}\cdot \lambda_{\widetilde{v}})_{P_{\widetilde{v}}}) \big)^{\an}
   \lra \widehat{S}(U^p,E)_{\overline{\rho}}[\fm_{\rho}]^{\an}.
\end{multline}
We show the composition is non-zero. By Lemma \ref{adj01}, the composition
\begin{equation*}
  \otimes_{v\in S_p} (\pi_{\widetilde{v}}^{\infty} \otimes_E L(s_{\widetilde{v}}\cdot \lambda_{\widetilde{v}})_{P_{\widetilde{v}}}) \lra  \big(\Ind_{\overline{P}}^{G_p} \otimes_{v\in S_p} (\pi_{\widetilde{v}}^{\infty} \otimes_E L(s_{\widetilde{v}}\cdot \lambda_{\widetilde{v}})_{P_{\widetilde{v}}}) \big)^{\an}
   \lra \widehat{S}(U^p,E)_{\overline{\rho}}[\fm_{\rho}]^{\an}
\end{equation*}
is non-zero. By \cite[Prop. 3.4 (i)]{Br13II}, the first map actually factors through $\cF_{\overline{P}}^{G_p}(\otimes_{v\in S_p} \overline{L}(-s_{\widetilde{v}}\cdot \lambda_{\widetilde{v}}), \otimes_{v\in S_p} \pi_{\widetilde{v}}^{\infty})$. We deduce thus (\ref{compSoc}) is non-zero, and this concludes the proof.
\end{proof}
\begin{remark}\label{soc}
Keep the assumptions and notation in Theorem \ref{main2}, and assume $\rho$ is automorphic such that $\widehat{S}(U^p,E)[\fm_{\rho}]^{\lalg}\neq 0$. Let $\Pi_{\widetilde{v}}^{\infty}$ be the unique generic subquotient of $(\Ind_{\overline{P}_{\widetilde{v}}(\Q_p)}^{\GL_n(\Q_p)} \pi_{\widetilde{v}}^{\infty})^{\infty}$. By the local-global compatibility in classical local Langlands correspondence, we have
\begin{equation}\label{soc0}
  \otimes_{v\in S_p} \big(\Pi_{\widetilde{v}}^{\infty}\otimes_E L(\lambda_{\widetilde{v}})\big)\hooklongrightarrow \widehat{S}(U^p,E)[\fm_{\rho}]^{\an},
\end{equation}
where $L(\lambda_{\widetilde{v}})$ denotes the algebraic representation of $\GL_n(\Q_p)$ of highest weight $\lambda_{\widetilde{v}}$ (with respect to $B$).
Let $Q_{\widetilde{v}}\supseteq P_{\widetilde{v}}$ be the maximal parabolic subgroup such that $s_{\widetilde{v}}\cdot \lambda_{\widetilde{v}}$ is dominant for $Q_{\widetilde{v}}$, and $L_{Q_{\widetilde{v}}}$ be the Levi subgroup of $Q_{\widetilde{v}}$ containing the torus.  Assume $(\Ind_{\overline{P}_{\widetilde{v}}(\Q_p)\cap L_{Q_{\widetilde{v}}}(\Q_p)}^{L_{Q_{\widetilde{v}}}(\Q_p)} \pi_{\widetilde{v}}^{\infty})^{\infty}$ is irreducible for $v\in S_p$.  By \cite[Lem. 2.10]{BH2} \cite[Thm. (iv)]{OS} (and the fact $\otimes_{v\in S_p} (\Ind_{\overline{P}_{\widetilde{v}}(\Q_p)\cap L_{Q_{\widetilde{v}}}(\Q_p)}^{L_{Q_{\widetilde{v}}}(\Q_p)} \pi_{\widetilde{v}}^{\infty})^{\infty}$ is irreducible as smooth $\prod_{v\in S_p} L_{Q_{\widetilde{v}}}(\Q_p)$-representation), we see $\widehat{\otimes}_{v\in S_p} \cF_{\overline{P}_{\widetilde{v}}}^{\GL_n}(\overline{L}(-s_{\widetilde{v}}\cdot \lambda_{\widetilde{v}}), \pi_{\widetilde{v}}^{\infty}) $ is topologically irreducible.  When there exists $v\in S_p$ such that $s_{\widetilde{v}} \neq 1$, then $\cF_{\overline{P}_{\widetilde{v}}}^{\GL_n}(\overline{L}(-s_{\widetilde{v}}\cdot \lambda_{\widetilde{v}}), \pi_{\widetilde{v}}^{\infty})$ is not locally algebraic, and the morphism (\ref{soc1}) provides an injection (other than (\ref{soc0})):
\begin{equation}\label{soc2}
\widehat{\otimes}_{v\in S_p} \cF_{\overline{P}_{\widetilde{v}}}^{\GL_n}(\overline{L}(-s_{\widetilde{v}}\cdot \lambda_{\widetilde{v}}), \pi_{\widetilde{v}}^{\infty})  \hooklongrightarrow \soc_{G_p} \widehat{S}(U^p,E)[\fm_{\rho}]^{\an}.
 \end{equation}
 This extra constituent appearing  in the socle of $\widehat{S}(U^p,E)[\fm_{\rho}]^{\an}$ is  predicted by Breuil's locally analytic socle conjecture (\cite[Conj. 5.3]{Br13II}), which was proved when $\rho_{\widetilde{v}}$ is crystalline and generic in \cite{BHS3} (see also \cite{Br13II}, \cite{Ding6} etc. for partial results on the conjecture). However, all the previous results used eigenvarieties in an essential way, and hence were limited to the case where  $\rho_{\widetilde{v}}$ is trianguline. By contrast, (\ref{soc2}) also applies to the case where $\pi_{\widetilde{v},i}^{\infty}$ is cuspidal for some $i$ (with $s_{\widetilde{v}}\neq 1$, $n_{\widetilde{v},i}=2$), which then gives a non-trivial example (probably the first, to the author's knowledge) towards the conjecture in this case.
%\\
%
%\noindent (2) Keep the situation as in (1). We know then $\rho$ appears in the $P$-ordinary family. But a possible phenomenon (unlike the ordinary family case) is that any point in the $P$-ordinary family attached to $\rho$ is not classical (i.e. does not come from the $P$-ordinary part of the locally algebraic $G_p$-representation associated to $\rho$). For example suppose $n=3$, $\rho_{\widetilde{v}}\cong \begin{pmatrix} \rho_{\widetilde{v},1} & *  \\ 0 & \rho_{\widetilde{v},2}\end{pmatrix}$ where $\rho_{\widetilde{v},1}$ is $2$-dimensional (and absolutely irreducible). Suppose $h_{\rho_{\widetilde{v},1},1}>h_{\rho_{\widetilde{v},2},1}> h_{\rho_{\widetilde{v},1},2}$ (then $\rho_{\widetilde{v}}$ does not have to split). Then using similar arguments as in \cite[Prop. 5.10]{BD1}, one can prove $\Ord_{P_{\widetilde{v}}}(\pi_{\widetilde{v}} \otimes_E W_{\widetilde{v}})=0$.
%
\end{remark}

\end{document}